%% file: fvk_bilayers.tex
\documentclass[11pt,twoside]{article}
\input{usepackages.tex}
\usepackage{import}

\usepackage{fancyhdr}
\input{header.tex}

\newcommand{\Addresses}{{
  \bigskip
  \small

  Sören Bartels, \textsc{Abteilung für angewandte Mathematik, Albert-Ludwigs-Universität Freiburg, Hermann-Herder-Str. 10, 79104 Freiburg im Breisgau, Germany}\par\nopagebreak
  \textit{Email address}: \texttt{bartels@mathematik.uni-freiburg.de}

  \medskip

  Bernd Schmidt, \textsc{Institut für Mathematik, Universität Augsburg, Universitätsstr.\ 14, 86159 Augsburg, Germany}\par\nopagebreak
  \textit{Email address}: \texttt{bernd.schmidt@math.uni-augsburg.de}

  \medskip

  Philipp Tscherner, \textsc{Abteilung für angewandte Mathematik, Albert-Ludwigs-Universität Freiburg, Hermann-Herder-Str. 10, 79104 Freiburg im Breisgau, Germany}\par\nopagebreak
}}

\newcommand{\cb}{\color{black}}

\newcommand{\cc}{\color{black}}
\newcommand{\cv}{\color{black}}

\begin{document}

\thispagestyle{empty}

\begin{center}
	\textbf{\large{NUMERICAL SIMULATION OF A FINE-TUNABLE FÖPPL--VON KÁRMÁN MODEL FOR FOLDABLE AND BILAYER PLATES}}\bigskip\\
	{\large Sören Bartels, Bernd Schmidt and Philipp Tscherner}
\end{center}
\medskip

\begin{abstract}
A numerical scheme is proposed to identify low energy configurations of a Föppl-von Kármán model for bilayer plates. The dependency of the corresponding elastic energy on the in-plane displacement $u$ and the out-of-plane deflection $w$ leads to a practical minimization of the functional via a decoupled gradient flow. In particular, the energies of the resulting iterates are shown to be monotonically decreasing. The discretization of the model relies on $P1$ finite elements for the horizontal part $u$ and utilizes the discrete Kirchhoff triangle for the vertical component $w$. The model allows for analysing various different problem settings via numerical simulation: (i) stable low-energy configurations are detected dependent on a specified prestrain described by elastic material properties, (ii) curvature inversions of spherical and cylindrical configurations are investigated, (iii) elastic responses of foldable cardboards for different spontaneous curvatures and crease geometries are compared.
\end{abstract}
\medskip

\let\thefootnote\relax\footnotetext{
\textit{Date}: \today
}

\section{Introduction} 

%The objective of this work is to numerically analyse a Föppl--von Kármán model for bilayer plates and to investigate energy-optimal configurations dependent on various strengths of prestrain that act on a given material. %In general, bilayer plate models 
{\cc The rigorous justification and numerical treatment of bilayer plate models have recently received considerable attention \cite{MR3602530,MR4156930}}, as they give rise to a wide range of applications including heated materials with inhomogeneous expansion coefficients \cite{bimetallic_strips} and crystallizations on top of substrates \cite{graphene1,graphene2}. Similar mechanics can be observed in natural systems, e.g., in biological materials with internal misfit caused by swelling or growing tissue \cite{flytrap,plant_bending}. Investigating such models can provide a deeper understanding of the elastic processes involved. A well-known model to describe elastic deformations of thin objects including nonlinear effects is the Föppl--von Kármán model \cite{ciarlet_fvk}. The authors of \cite{schmidt_multilayers,schmidt_hierarchy} recently derived such a model for bilayer plates via $\Gamma$-convergence, we refer to the seminal {\cb contributions \cite{fjm,fjm_vK}} for underlying concepts. It includes a parameter $\theta>0$ that determines the strength of prestrain acting on the elastic body. Under the assumption that the material is homogeneous with linear internal misfit, the deformation of a plate $\Omega\subset\mathbb{R}^2$ can be described by an in-plane displacement $u:\Omega\to\mathbb{R}^2$ and an out-of-plane deflection $w:\Omega\to\mathbb{R}$ via minimization of the dimensionally reduced elastic energy
\begin{align*}
	E^{\theta}(u,w) = \frac{1}{2} \int_{\Omega} |D^2w - \alpha I|^2  \dx
	+ \frac{\theta}{2} \int_{\Omega} |\nabla w \otimes \nabla w + \widetilde{\eps}(u)|^2 \dx - \int_\Omega f w \dx , 
	%=: E^B(w) + E^S_\theta(w,u)
\end{align*}
in a set of admissible pairs $ (u,w) \in H^1(\Omega,\mathbb{R}^2) \times H^2(\Omega) $ subject to appropriate boundary conditions and a vertical dead body load $f:\Omega \to \mathbb{R}$. 
The first term in the energy captures bending phenomena via deviations of the Hessian to the identity matrix %$I$
scaled by some parameter $\alpha\in\mathbb{R}$ 
with respect to the squared Frobenius norm $|A|^2:=\sum_{i=1}^2\sum_{j=1}^2|a_{ij}|^2$, thereby leading to the preference of a certain curvature in the out-of-plane deflection $w$. The second part of the energy includes shearing effects and involves twice the symmetric gradient $ \widetilde{\eps}(u) = \nabla u + \nabla u^{\top} $ and the dyadic product $ x \otimes y = x y \T $, {\cb which introduces a coupling of the in- and out-of-plane components $u$ and $w$}, whose strength depends on the parameter choice $\theta\in(0,\infty)$. {\cb Extending previous results which considered limiting energy functionals within regimes in which the typical energy per unit volume scales with powers of the film thickness $\gamma$ (see, in particular, \cite{fjm_vK,lmp,LL20}), the parameter $\theta$ introduces an additional fine scale in the von Kármán regime $\gamma^4$. Indeed, for $\theta = 0$, $E^\theta$ reduces to the energy functional of a linear plate theory whereas, for $\theta = \infty$, the limiting energy $E^\theta$ is that of a linearized Kirchhoff plate theory in which {\cv $\nabla w \otimes \nabla w$ is constrained to be a symmetrized gradient} or, equivalently {\cv (cf.\ \cite{fjm_vK})}, $w$ satisfies the linearized isometry constraint $\det D^2w=0$.} {\cv Thus, the von Kármán functional augmented with the parameter $\theta$ provides a mathematical model that allows to discriminate `thick' and `thin' plates. In particular, it is able to mathematically sustain a basic phenomenon observed in engineering systems (\cite{SalamonMasters93,SalamonMasters95,FinotSuresh96,Freund00,KimLombardo08,EgunovKorvinkLuchnikov16}: Large prestrains in very thin layers result in cylindrical shapes whereas small prestrains in thick layers lead
to spherical caps, see \cite{schmidt_multilayers}.}

{\cc Furthermore, a piecewise minimization of the energy on two adjacent domains coupled with a continuity condition along a given connecting crease line lead to simulations of foldable single- and multilayer devices. Consequently, many interesting phenomena like the elastic response of foldable cardboards, cf. Figure \ref{fig:photos_cardboard}, or the actuation of bilayer mechanisms inspired by Venus flytraps, can be numerically investigated.}

\begin{figure}[H] 
\centering
\begin{minipage}{0.34\textwidth}
\centering
\includegraphics[width=0.5\textwidth]{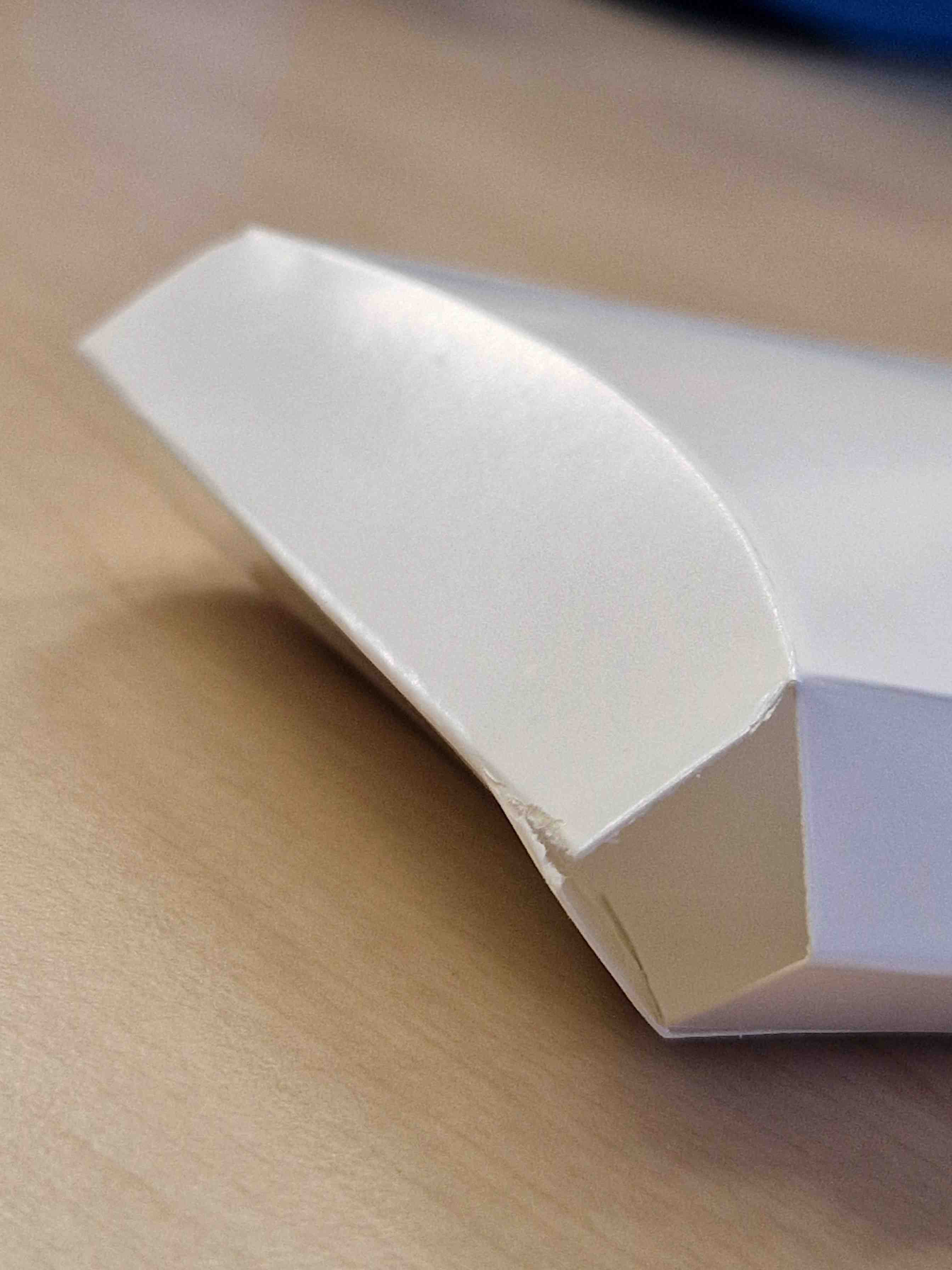}
\end{minipage}
\begin{minipage}{0.34\textwidth}
\centering
\includegraphics[width=0.5\textwidth]{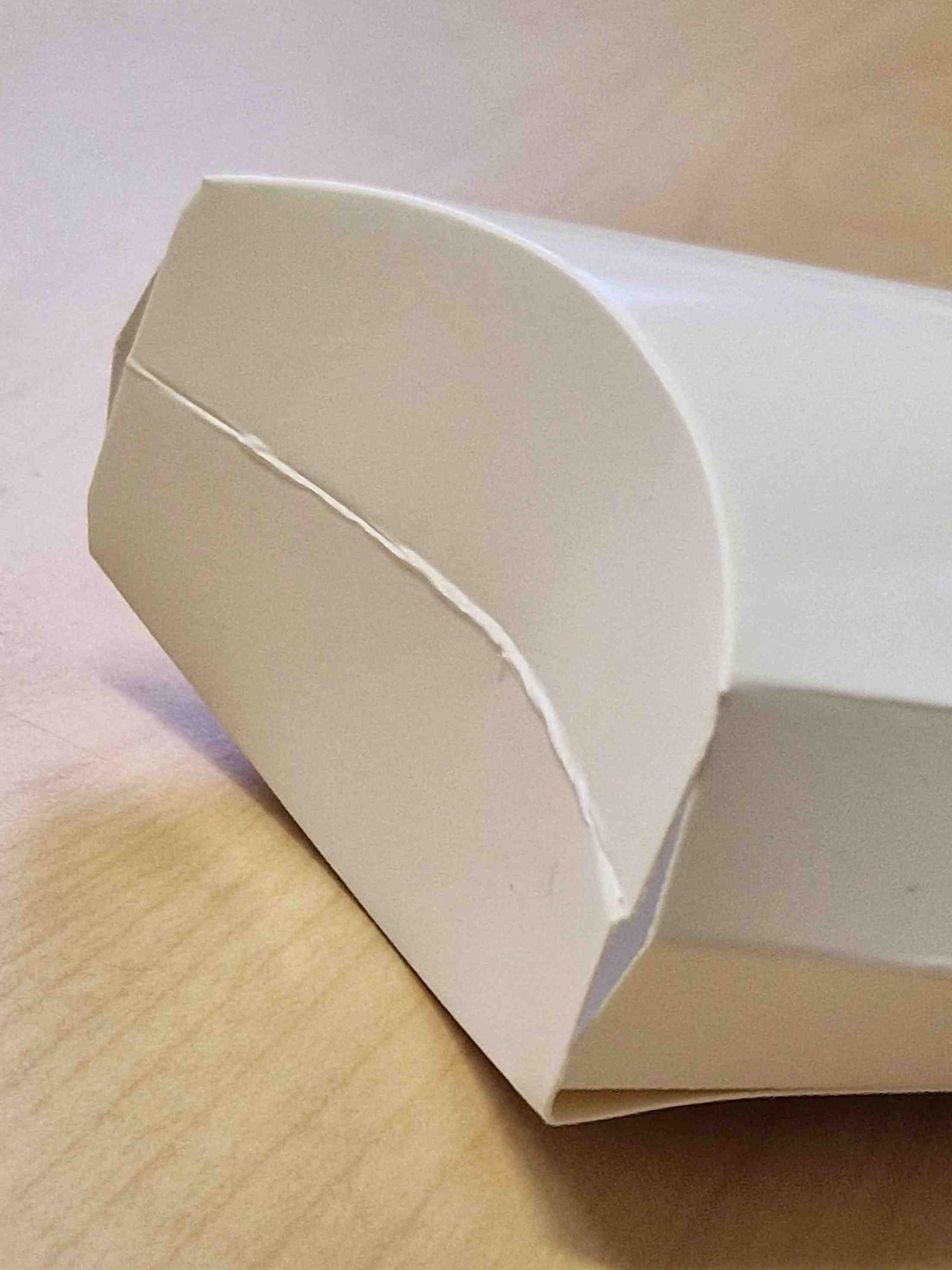}
\end{minipage}
\caption{Bistable mechanism of a foldable cardboard. After repeated actuation, a crack can be observed (left picture) which emerges from the boundary and spreads along the given crease.
% Surprisingly, the crack length is close to the distance between the boundary and the wrinkles observed in the last plots of Figure \ref{fig:cardboard_bending_radial}, suggesting that the crack is predetermined by the formation of such wrinkles.
}
\label{fig:photos_cardboard}
\end{figure} 

The nonconvex structure of the elastic energy complicates the computation of global minimizers. An effective method to numerically detect stationary configurations with low elastic energy has been proposed in \cite{bartels_fvk}. It includes a gradient flow for the energy with respect to both variables $u$ and $w$. In our case, this results in the coupled system of nonlinear evolution equations
\begin{align*}
	\bigl(\partial_t w,v\bigr)_\text{ver}=&-\partial_{w} E^{\theta}(u,w)[v] \\
	= & -\bigl(D^2w - \alpha I , D^2v\bigr) - 2 \theta \bigl(|\nabla w|^2 \nabla w + \widetilde{\eps}(u) \nabla w , \nabla v\bigr) + \bigl(f,v\bigr),\\
	\bigl(\partial_t u,z\bigr)_\text{hor}=&-\partial_u E^{\theta}(u,w)[z] \\
	= & -\theta \bigl(\widetilde{\eps}(u) , \widetilde{\eps}(z)\bigr) - \theta \bigl(\nabla w \otimes \nabla w , \widetilde{\eps}(z)\bigr),
\end{align*}
with the Fréchet derivatives $\partial_w E$, $\partial_u E$ and inner products $(\cdot\,,\,\cdot)$, $(\cdot\,,\,\cdot)_\text{ver}$ and $(\cdot\,,\,\cdot)_\text{hor}$ on $L^2(\Omega)$, $H^2(\Omega)$ and $H^1(\Omega,\mathbb{R}^2)$, respectively. The identities $|a\otimes a |^2=|a|^4$ and $S:(a\otimes b)=(Sa)\cdot b=(Sb)\cdot a$ for symmetric matrices $S\in \mathbb{R}^{2\times2}$ and vectors $a,b\in\mathbb{R}^2$ have been used in the derivation of the above system of equations.

To specify a discrete version of the evolution equations, time derivatives are replaced by the backward difference quotients 
\begin{align*}
d_t u^k = \frac{1}{\tau_k}\bigl(u^k-u^{k-1}\bigr), \quad d_t w^k = \frac{1}{\tau_k}\bigl(w^k-w^{k-1}\bigr),
\end{align*}
where $(\tau_k)_{k\geq1}$ denotes a sequence of positive step sizes. A decoupling of the evolution equations is employed to combine practical stability properties of an implicit discretization with amenable solvability of an explicit treatment. In the same manner, we make use of the delay effect of the discrete product rule and define the average $w^{k-1/2}$ as
\begin{align*}
w^{k-1/2}=\frac{1}{2}\bigl(w^k+w^{k-1}\bigr),
\end{align*} 
to arrive at the following iterative scheme: Given $(u^{k-1},w^{k-1})$ compute $(u^{k},w^{k})$ such that
\begin{align*}
	\bigl(d_t w^k,v\bigr)_\text{ver} = & -\bigl(D^2w^k - \alpha I , D^2v\bigr) - 2 \theta \bigl(|\nabla w^k|^2 \nabla w^k + \widetilde{\eps}(u^{k-1}) \nabla w^{k-1/2} , \nabla v\bigr) + \bigl(f,v\bigr),\\
	\bigl(d_t u^k,z\bigr)_\text{hor}	= & -\theta \bigl(\widetilde{\eps}(u^k) , \widetilde{\eps}(z)\bigr) - \theta \bigl(\nabla w^k \otimes \nabla w^k , \widetilde{\eps}(z)\bigr),
\end{align*}
for all $(v,z)$ satisfying homogeneous boundary conditions. We show that the scheme is unconditionally stable and energy decreasing which implies convergence of a sequence $(u^k,w^k)$ to a stationary configuration $(u,w)$. To ensure well-posedness of the equations we employ the following adaptive time stepping scheme that includes an appropriate parameter $\tau_\text{max}>0$ to prevent numerical overflow:
\begin{itemize}
\setlength\itemsep{1pt}
\item decrease $\tau_k$ until the Newton scheme terminates within $N>0$ iterations
\item set $\tau_{k+1} = \min\bigl\{2\tau_k,\tau_\text{max}\bigr\}$ for the next gradient flow step
\end{itemize}
The dependence of the Föppl--von Kármán model on the in-plane and out-of-plane components allows for different spatial discretizations of the two respective variables. In particular, the displacement $u$ is discretized using $P_1$ finite elements whereas the deflection $w$ is discretized via the discrete Kirchhoff triangle \cite{bartels_dkt}. The latter allows for a practical realization of bending elements since the discrete gradient values $\nabla_h w_h$ belong to degrees of freedom of the finite element space. The discrete elastic energy functional reads
\begin{align*}
E^{\theta}_h(u_h,w_h) = & \ \frac{1}{2} \int_{\Omega} |\nabla\nabla_h w_h - \alpha I |^2  \dx \\
	& \hspace{10pt} + \frac{\theta}{2} \int_{\Omega} \widehat{\mathcal{I}}_h\bigl[|\nabla w_h \otimes \nabla w_h + \widetilde{\eps}(u_h)|^2\bigr] \dx - \int_\Omega \widehat{\mathcal{I}}_h\bigl[f w_h\bigr] \dx ,
\end{align*}
where $\widehat{\mathcal{I}}_h$ is the elementwise nodal interpolation operator introduced in Section \ref{sec:P1} and $\nabla_h$ denotes the discrete gradient operator defined in Section \ref{sec:dkt}. Under appropriate boundary conditions the discrete elastic energy $E_h$ $\Gamma$-converges to the continuous elastic energy $E$ in $H^1(\Omega,\mathbb{R}^2)\times W^{1,4}(\Omega)$ for a sequence of regular triangulations $(\Th)_{h>0}$ as $h\to0$.  

The practical realization of the model based on rigorously justified numerical methods is motivated by results from \cite{schmidt_multilayers}. The authors quantify configurations of spherical and cylindrical shape as energy minimizing configurations of the continuous limit models as $\theta\to0$ and $\theta\to\infty$, respectively. We investigate stationary low energy configurations of the discrete model for intermediate values $\theta\in(0,\infty)$ via numerical experiments. The aim is to possibly identify a critical range of values $\theta$ at which a rapid transition from spherical to cylindrical shape takes place. 
%Closely related to elastic processes in isometric materials 
%%that switch between two configurations 
%is the effect of wrinkling \cite{wrinkling1,wrinkling2,wrinkling3} and buckling \cite{buckling_shells1,buckling_shells2,buckling_shells3,buckling_shells4}. The accompanying formation of creases is often related to material failure induced by large forces that act on thin elastic objects.
%We perform numerical experiments on the transition of a simply supported elastic membrane from a concave to a convex spherical configuration by pushing it through its center via an external load. 
{\cb Such a stark change of material response in a critical parameter region has been indicated by first experiments in \cite{schmidt_multilayers}, which were based on an ad-hoc projected gradient descent using  nonconforming $P_1$ elements on the out-of-plane strain and enforcing this quantity to be curl-free eventually by a penalization. Our present numerical analysis and implementation detailed above substantiates these observations. Moreover,} we perform numerical experiments on the curvature inversion of the resulting spherical and cylindrical configurations and investigate the elastic responses of foldable cardboards for different spontaneous curvatures and crease geometries.

The outline of this article is as follows. Section 2 is devoted to the derivation of the continuous Föppl--von Kármán model along with discretization aspects. In Section 3 we show the energy decreasing property of the discrete gradient flow. The $\Gamma$-convergence of the discrete energy $E_h$ to the continuous energy $E$ is derived in Section 4. Section 5 contains various numerical experiments.

\section{Preliminaries}

\subsection{Model Derivation}\label{sec:model_derivation}

A rigorous derivation of the Föppl--von Kármán model for prestrained plates has been established in \cite{schmidt_hierarchy} via $\Gamma$-convergence. For a thin plate $\Omega_\gamma=\Omega\times(-\gamma/2,\gamma/2)$ of thickness $\gamma>0$ with $\Omega\subset\mathbb{R}^2$ {\cb a bounded Lipschitz domain} and some elastic energy density $W_\gamma:\mathbb{R}^{3\times3}\to\mathbb{R}$, the corresponding three-dimensional scaled hyperelastic energy reads
\begin{align*}
E_\text{3d}(y)=\frac{1}{\gamma^4}\int_{\Omega_1} W_\gamma(x_3,\partial_1y,\partial_2y,\gamma^{-1}\partial_3y){\cb \dx},
\end{align*}
where {\cb $y \in H^1(\Omega_1,\mathbb{R}^{3})$} is defined via rescaling to the plate $\Omega_1$ of thickness $\gamma=1$. In contrast to general singlelayer models, the energy considered here explicitly depends on the out-of-plane variable $x_3\in(-1/2,1/2)$. The elastic energy density $W_\gamma$ reads
\begin{align*}
W_\gamma(x_3,F) = W_0(x_3,F(I+\gamma^2\sqrt{\theta}B_\gamma(x_3))),\quad F\in\mathbb{R}^{3\times3},
\end{align*}
and is described by the stored energy density $W_0$ of the reference configuration that depends on some \textit{internal misfit} $B_\gamma:(-1/2,1/2)\to\mathbb{R}^{3\times3}$ weighted by {\cb $\gamma^2 \sqrt{\theta}$ with} a parameter $\theta\in(0,\infty)$. 
The energy density $W_0$ is assumed to satisfy the classical, physically motivated requirements such as \textit{smoothness} in a neighbourhood of $SO(3)$, \textit{frame indifference}, {\cb \textit{non-degeneracy}}, and \textit{quadratic growth}, {\cb so that it is minimal (with value $0$) precisely on $SO(3)$. 
As the strength of the misfit scales with $\gamma^2$, typical deformation gradients deviate from $SO(3)$ by a comparable amount and the effect of the prestrain is suitably described within the von Kármán energy scaling $\gamma^4$.} 
The variable $\theta$ specifies the amount of misfit {\cb on that scale} and serves as an interpolation between the linearized von Kármán ($\theta\to0$) and the linearized Kirchhoff ($\theta\to\infty$) models as $\gamma\to0$. 
{\cb By frame invariance, rigid body motions do not store elastic energy so that for each deformation $y$ and $R \in SO(3)$, $c \in \mathbb{R}^3$ one has 
\begin{align*}
  E_\text{3d}(y) = E_\text{3d}(Ry + c). 
\end{align*}
Considering the asymptotic behavior of a sequence of rescaled deformations as $\gamma \to 0$ we may and will thus renormalize by tacitly applying a rigid body motion to the plate so as to guarantee that 
\begin{align}\label{eq:y-normalize}
 \| y - \mathrm{id}_\gamma \|_{L^2} 
 = \min \big\{ \|Ry + c - \mathrm{id}_\gamma \|_{L^2} : R \in SO(3), \, c \in \mathbb{R}^3 \big\},   
\end{align}
where $\id_\gamma(x) = (x_1,x_2,\gamma x_3)$. 
In the von Kármán regime one then considers the rescaled in-plane and out-of-plane displacements 
%\begin{align*}
%u_i(x_1,x_2) &= \frac{1}{\theta \gamma^2} \int_{-1/2}^{1/2} y_i-x_i \dx_3, \quad i=1,2,  \\
%w(x_1,x_2) &= \frac{1}{\sqrt{\theta}\gamma} \int_{-1/2}^{1/2} y_3 \dx_3  
%\end{align*} 
\begin{align}\label{eq:full-displ}
\widehat u_i(x_1,x_2) = \frac{1}{\theta \gamma^2} ( y_i(x)-x_i ) \quad(i=1,2) \quad\text{and}\quad  
\widehat w(x_1,x_2) = \frac{1}{\sqrt{\theta}\gamma} y_3(x)
\end{align} 
and their averages along the small plate height 
\begin{align*}
u_i(x_1,x_2) = \int_{-1/2}^{1/2} \widehat u(x) \dx_3 
\quad(i=1,2) \quad\text{and}\quad  
w(x_1,x_2) = \int_{-1/2}^{1/2} \widehat w(x) \dx_3  
\end{align*} 
as $\gamma \to 0$. The $\Gamma$-limit of $E_\text{3d}$ together with a suitable compactness results for bounded energy sequences has been obtained in \cite{schmidt_hierarchy} with respect to convergence of the averaged displacement variable $(u,w)$. 
Since the corresponding deformation gradients are close to rotations, the resulting energy functional $(u,w) \mapsto E^{\theta}(u,w)$ only depends on the infinitesimal elastic moduli and is explicitly computable from $D_F^2 W_0(t,I)$. Omitting the forcing term, for the prototypical example of} an isotropic homogeneous material {\cb with $W_0(t,F) = \mathrm{dist}^2(F,SO(3))$ and} linear internal misfit $B(t)= tI$ we arrive at the following energy functional
\begin{align*}
	E^{\theta}(u,w) = \frac{1}{2} \int_{\Omega} |D^2w -  I |^2  \dx
	+ \frac{\theta}{2} \int_{\Omega} |\nabla w \otimes \nabla w + \widetilde{\eps}(u)|^2 \dx,
\end{align*}
which is a rescaled version of the general Föppl--von Kármán functional. The corresponding set of admissible pairs $(u,w)\in\mathcal{A}$ is defined as
\begin{align*}
\mathcal{A} = \mathcal{A}_0 + (u_D,w_D) \subset H^1(\Omega,\mathbb{R}^2) \times H^2(\Omega).
\end{align*}
Here, $\mathcal{A}$ and $\mathcal{A}_0$ are affine and linear subspaces of $H^1(\Omega,\mathbb{R}^2) \times H^2(\Omega)$, respectively, such that the following Korn--Poincaré inequality holds, i.e.,
\begin{align}\label{korn_poincare}
||u_0||_{H^1(\Omega)} + ||w_0||_{H^2(\Omega)} \leq c_P \Bigl(||\widetilde{\varepsilon}(u_0)|| + ||D^2w_0|| \Bigr)
\end{align}
for all $(u_0,w_0)\in\mathcal{A}_0$. For instance, this is the case if we have
\begin{align*}
u_0|_{\gamma_D}=0,\quad w_0|_{\gamma_D}=0,\quad \nabla w_0|_{\gamma_D} = 0,
\end{align*}
on a subset $\gamma_D\subset\partial\Omega$ of positive surface measure or
\begin{align*}
u_0|_{\partial\Omega}=0,\quad w_0|_{\partial\Omega}=0,
\end{align*}
on the whole boundary $\partial\Omega$ of $\Omega$. {\cc To guarantee well-posedness of the iterative scheme, we note that the Sobolev inequality
\begin{align}\label{sobolev}
||\nabla w_0||_{L^4(\Omega)} \leq c_S||D^2w_0||
\end{align}
holds for all $w_0\in\mathcal{A}_0$ due to the Sobolev embedding $H^2(\Omega) \hookrightarrow W^{1,4}(\Omega)$.}

{\cb Our results on the minimizers of $E^\theta$ yields a very precise asymptotic analysis of low energy configurations for the original three dimensional functional, which for definiteness we formulate here for the prototypical case 
\begin{align*}
 E_\mathrm{3d}^\gamma(y) 
 = \frac{12}{\gamma^4}\int_{\Omega_1} \mathrm{dist}^2\big((1+\gamma^2(\theta/3)^{1/2}\,x_3)(\partial_1y,\partial_2y,\gamma^{-1}\partial_3y),SO(3)\big) \dx.  
\end{align*} 
(We introduced irrelevant constants so as to arrive at the simplest limiting functional.) 

A first consequence of the $\Gamma$-convergence of $E_\text{3d}$ to $E^{\theta}$ is that for a sequence of (almost) minimizers $y_\gamma$ of $E_\text{3d}$ one has subsequential convergence of the rescaled averaged dispacements $u_\gamma \to u$ and $w_\gamma\to w$, where $(u,w)$ is a minimizer of $E^{\theta}$. Indeed, combining the $\Gamma$-convergence and compactness results in \cite[Theorem~3.1 and Lemma~4.1]{schmidt_hierarchy} with the observation in \cite[Proposition~2]{BraunSchmidt:22} yields a slightly stronger result for the limiting behavior of the full displacements in \eqref{eq:full-displ}
with limiting plate functional 
\begin{align*}
	E^{\theta}(u,w) = \frac{1}{2} \int_{\Omega} |D^2w -  I |^2  \dx
	+ \frac{\theta}{2} \int_{\Omega} |\nabla w \otimes \nabla w + \widetilde{\eps}(u)|^2 \dx.
\end{align*}

\begin{thm}
Suppose $(y_\gamma) \subset H^1(\Omega_1,\mathbb{R}^3)$ is a sequence of almost minimizers for $E_\mathrm{3d}$, i.e.,
\begin{align*}
 \lim_{\gamma \to 0} \Big( E_\mathrm{3d}^\gamma(y_\gamma) - \min\big\{ E_\mathrm{3d}(y) : y \in H^1(\Omega_1,\mathbb{R}^3) \big\} \Big) = 0,  
\end{align*}
such that \eqref{eq:y-normalize} is satisfied for every $y_\gamma$. 
Then there exists a subsequence (not relabeled) such that the corresponding displacements $(\widehat{u}_\gamma, \widehat{w}_\gamma)$ given in \eqref{eq:full-displ} satisfy 
\begin{align*}
 \widehat{u}_\gamma \rightharpoonup \widehat{u} 
 ~\text{in}~ H^1(\Omega_1, \mathbb{R}^2) 
 \quad\text{and}\quad 
 \widehat{w}_\gamma \rightharpoonup \widehat{w} 
 ~\text{in}~ H^1(\Omega_1),  
\end{align*}
where the limiting $(\widehat u, \widehat w)$ is given by 
\begin{align*}
 \widehat u(x) &= u(x_1,x_2) - \frac{1}{\sqrt{\theta}} x_3 \big( \partial_1 w(x_1,x_2),\partial_2 w(x_1,x_2) \big), \\ 
 \widehat w(x) &= w(x_1,x_2) + \frac{1}{\sqrt{\theta}} x_3    
\end{align*}
and $(u,w)$ is a minimizer of $E^{\theta}$. 
\end{thm}

We note that the in-plane scale $\theta\gamma^2$ is negligible with respect to the deflection scale $\sqrt{\theta} \gamma$.  
Thus, the descaled original displacement is to leading order determined by the limiting deflection $w$: 
\begin{align}\label{eq:y-gamma-descaled}
 y_\gamma(x) &= \big( x_1 + \theta \gamma^2 \widehat u_1(x) + o(\gamma^2),  x_2 + \theta \gamma^2 \widehat u_2(x) + o(\gamma^2), \sqrt{\theta} \gamma \widehat w(x) + o(\gamma) \big)\nonumber \\ 
 &= (x_1,x_2,\gamma x_3) + \sqrt{\theta} \gamma \, w(x_1,x_2) (0,0,1)\T + o(\gamma).   
\end{align}
More precisely, in terms of the scaled identity $\mathrm{id}_\gamma(x) = (x_1,x_2,\gamma x_3)$ we have: 
\begin{cor}
Suppose $(y_\gamma) \subset H^1(\Omega_1,\mathbb{R}^3)$ is a sequence of almost minimizers for $E_\mathrm{3d}$
such that \eqref{eq:y-normalize} is satisfied for every $y_\gamma$. 
Then for a subsequence (not relabeled) we have 
\begin{align*}
 \frac{1}{\sqrt{\theta}\gamma} ( y_\gamma - \mathrm{id}_\gamma ) \rightharpoonup w\,(0,0,1)\T  
 ~\text{in}~ H^1(\Omega_1, \mathbb{R}^3),  
\end{align*}
where $w \in H^1(\Omega)$ is such that $(u,w)$ is a minimizer of $E^{\theta}$ for a suitable $u$, i.e., a minimizer for $\min\{ E^{\theta}(u, \cdot) : u \in H^1(\Omega,\mathbb{R}^2)\}$. 
\end{cor}
}

{\cv
\subsection{Remarks on plates with folds}
\newcommand{\R}{\mathbb{R}}

While the above results apply to elastic plates, in Section~\ref{subsec:FoldableCardboard} below we will also report on numerical experiments for plates which can be folded along some special curves. From an analytical point of view, the derivation of suitable effective plate theories for such foldable plates is rather challenging. We include here a short discussion of the difficulties that occur when modeling infinitesimal deflections in foldable thin structures. 

For homogeneous Kirchhoff plates which are subject to finite (nonlinear) bending and whose energy scales like $\gamma^2$, folds and even cracks and voids have been sucessfully analyzed in \cite{BBH,SantilliSchmidt:23}. Following the approach in \cite{BBH} one may consider a plate whose stored energy function is weakened in a tubular neighborhood of some crease line. More precisely, let $\Sigma \subset \Omega$ be a Jordan arc with both endpoints on the same connected component of $\partial \Omega$ and such that the two connected components of $\Omega \setminus \Sigma$ are themselves Lipschitz domains. Let $\Sigma_r = \{x \in \R : \mathrm{dist}(x,\Sigma) < r\}$ and introduce the damage indicator $f^\gamma : \Omega_1 \to [0,1]$ by 
%\[ f^\gamma = \eps_\gamma \chi_{\Sigma_{r_\gamma}} + 1 - \chi_{\Sigma_{r_\gamma}}, \]
%where $0 < r_\gamma, \eps_\gamma$ are parameters that measure the width and, respectively, level of degradation within the damaged region $\Sigma_{r_\gamma}$. 
\[ f^\gamma(x) = \eps_\gamma \chi_{\Sigma_{\gamma}}(x_1,x_2) + 1 - \chi_{\Sigma_{\gamma}}(x_1,x_2),\quad\mbox{where } \gamma^2 \lesssim \eps_\gamma \ll \gamma. \]
(More general models can be considered.) 
The main result of \cite{BBH} implies that, as $\gamma \to 0$ 
%and $\eps_\gamma, r_\gamma \to 0$ such that 
%\[ \gamma^2 \lesssim \eps_\gamma,\quad 
%   \gamma \lesssim r_\gamma \quad\mbox{and}\quad 
%   \eps_\gamma r_\gamma \ll \gamma^2
% \] 
the prototypical energy functionals 
\begin{align*}
y \mapsto \frac{1}{\gamma^2}\int_{\Omega_1} f^\gamma \, \mathrm{dist}^2\big((\partial_1y,\partial_2y,\gamma^{-1}\partial_3y), SO(3)) \dx,
\end{align*}
$\Gamma$-converge to the limiting functional 
\[ E_{K}^{\rm hom}(y) 
   = \begin{cases}
       \frac{1}{24} \int_{\Omega\setminus\Sigma} |A|^2 \, \dx &\mbox{if } y \in \mathcal{A}, \\
       + \infty &\mbox{otherwise}. 
     \end{cases} \]
Here the admissible class $\mathcal{A}$ consists of piecewise isometric immersions: 
\[ \mathcal{A} 
   = \big\{ y \in H^1(\Omega;\R^3) \cap H^2(\Omega\setminus \Sigma;\R^3) : \nabla y\T \nabla y = I \mbox{ a.e.\ on } \Omega \big\}, \]
and $A$ is the second fundamental form of the surface $y(\Omega)$. 
A straightforward combination of these results with the analysis of multilayer Kirchhoff plates, cf.\ \cite{Schmidt:07a,Schmidt:07b,Padilla-Garza:22}, then shows that the functionals
\begin{align*}
 y \mapsto  
 \frac{1}{\gamma^2}\int_{\Omega_1} f^\gamma \, \mathrm{dist}^2\big((1+\gamma(a + o_\gamma(1))\,x_3)(\partial_1y,\partial_2y,\gamma^{-1}\partial_3y),SO(3)\big) \dx  
\end{align*} 
for given $a \in \R$ $\Gamma$-converge to 
\[ E_{K}(y) 
   = \begin{cases}
       \frac{1}{24} \int_{\Omega\setminus\Sigma} |A - a I|^2 \, \dx &\mbox{if } y \in \mathcal{A}, \\
       + \infty &\mbox{otherwise}. 
     \end{cases} \]
Moreover, bounded energy sequences are precompact. We remark that the scaling assumption $\gamma^2 \lesssim \eps_\gamma \ll \gamma$ for the degradation strength guarantees that near the crease the material is so weak that arbitrarily large folding angles are possible at zero energy while still being strong enough not to break. 

In the von Kármán regime, the interplay of possible folding angles and their energetic costs is more complicated. One natural choice is to consider a very weak regime where the damage indicator now satsifies  
\[ f^\gamma = \eps_\gamma \chi_{\Sigma_{\gamma}} + 1 - \chi_{\Sigma_{\gamma}} \quad \mbox{with } \gamma^4 \lesssim \eps_\gamma \ll \gamma^3. \]
Then still arbitrarily large folding angles are possible at zero energy. If we adapt the energy functional of the previous section by setting 
\begin{align*}
 E_\mathrm{3d}^\gamma(y) 
 = \frac{12}{\gamma^4}\int_{\Omega_1} f^\gamma \, \mathrm{dist}^2\big((1+\gamma^2(\theta/3)^{1/2}\,x_3)(\partial_1y,\partial_2y,\gamma^{-1}\partial_3y),SO(3)\big) \dx,   
\end{align*} 
we can apply the above results for Kirchhoff plates (with $a = 0$) to the functionals $\gamma^2 E_\mathrm{3d}^\gamma$ to see that any sequence $(y_\gamma)$ with $E_\mathrm{3d}^\gamma(y_\gamma) \le C$ and hence $\gamma^2 E_\mathrm{3d}^\gamma(y_\gamma) \to 0$ converges -- up to subsequences -- to a limiting deformation $y$ whose second fundamental form vanishes on the two components $\Omega_1,\Omega_2$ of $\Omega \setminus \Sigma$. But then $y$ is a rigid motion on these sets, so there are $R_1,R_2\in SO(3)$ and $c_1,c_2 \in\R^3$ such that 
\begin{align}\label{eq:y-R1-R2}
 y(x) = 
   \begin{cases} 
    R_1 x + c_1 &\mbox{if } x \in \Omega_1, \\ 
    R_2 x + c_2 &\mbox{if } x \in \Omega_2.
   \end{cases} 
\end{align}
Moreover, as $y \in H^1(\Omega,\R^3)$ cannot jump along $\Sigma$, $R_1$ and $R_2$ are rank-1-connected and $\Sigma$ must be a straight line unless $R_1 = R_2$ and $c_1 = c_2$ in which cases there is no fold at all. If $R_1 \ne R_2$ with $R_2 - R_1 = a \otimes n$, then $n$ is a normal to $\Sigma$. 

With these observation one is led to consider renormalizations by substracting from $y_\gamma$ limiting deformations of the form \eqref{eq:y-R1-R2} rather than a single rigid motion as in \eqref{eq:y-normalize}. Since there is no energy cost caused by the crease, the problem completely decouples and one can apply the results above on both components separately. After descaling (cf.\  \eqref{eq:y-gamma-descaled} in the homogenous case) we see that deformations $y_\gamma$ are to leading order of the form 
\begin{align*}
 y_\gamma(x) = 
  R_i \Big( (x_1,x_2,\gamma x_3) + \sqrt{\theta} \gamma \, w(x_1,x_2) (0,0,1)\T \Big) + c_i + o(\gamma) \quad \mbox{for } x \in \Omega_i, \ i = 1,2.
\end{align*}
In particular, there is no energetical coupling between $w|_{\Omega_1}$ and $w|_{\Omega_2}$. 

Other natural scaling regimes for the damage parameter are obtained by choosing larger values for $\eps_\gamma$ such as in the bending regime above, i.e., $\gamma^2 \lesssim \eps_\gamma \ll \gamma$. While we expect that a $\Gamma$-convergence result can be proved under the assumption that there is a unique rigid body motion as in \eqref{eq:y-normalize} such that the rescaled relative in-plane and out-of-plane displacements \eqref{eq:full-displ} are convergent, we must observe now that such an assumption cannot be inferred from energy bounds. Indeed, exploratory computations indicate that if $R_i \in SO(3)$ and $c_i \in \R^3$ are optimal choices in \eqref{eq:y-normalize} on $\Omega_i$ $i = 1,2$, for $y = y_\gamma$, a bounded energy sequence, then 
\[ |R_2 - R_1|^2 \lesssim \eps_\gamma^{-1} \gamma^3 \]
and that configuartions with a folding angle scaling with $ \eps_\gamma^{-1/2} \gamma^{3/2}$ are possible but will lead to an extra `folding energy' contribution in the limit. A full analysis of this regime appears challenging and is beyond the scope of this contribution. 
}

\subsection{$P_1$-Finite Elements}\label{sec:P1}
In this section we introduce the finite element space used for the discretization of the in-plane component $u$. For a regular triangulation $\Th$ of the polygonal domain $\Omega\subset\mathbb{R}^2$, the standard $P1$ finite element space is defined as
\begin{align*}
\mathcal{S}^1(\Th) = \Bigl\{v_h\in C(\overline{\Omega})\,\Big|\, v_h|_T\text{ affine for all }T\in\Th\Bigr\}.
\end{align*}
If $\mathscr{N}_h$ denotes the set of nodes of the triangulation, the nodal basis functions $(\varphi_z)_{z\in\mathscr{N}_h}$ associated with $\mathcal{S}^1(\Th)$ satisfy the Kronecker delta property $\varphi_z(y)=\delta_{zy}$ for all $z,y\in\mathscr{N}_h$. We introduce the space of discontinuous $P1$ functions which is defined as
\begin{align*}
\widehat{\mathcal{S}}^1(\Th) = \Bigl\{v_h\in L^\infty(\Omega)\,\Big|\, v_h|_T\text{ affine for all }T\in\Th\Bigr\},
\end{align*}
along with the elementwise nodal interpolant $\widehat{\mathcal{I}}_hv\in\widehat{\mathcal{S}}^1(\Th)$ of piecewise continuous functions $v\in L^\infty(\Omega)$ to account for gradient jumps of discrete functions across element sides $S\in\mathscr{S}_h$. In particular, if $\varphi_z^T\in L^\infty(\Omega)$ is the discontinuous function such that $\varphi_z^T(x)=\chi_T(x)\varphi_z(x)$ for all $x\in\Omega$, we define
\begin{align*}
\widehat{\mathcal{I}}_hv = \sum_{T\in\Th}\sum_{z\in\mathscr{N}_h\cap\,T} v|_T(z)\varphi_z^T.
\end{align*}
The interpolator is used to approximate the $L^2$ inner product of piecewise continuous functions or vector fields $v,w\in L^\infty(\overline{\Omega},\mathbb{R}^\ell)$ for $\ell=1,2$ and coefficients $\beta_z^T = \int_T \varphi_z\dx$ via
\begin{align}\label{discrete_product}
\bigl(v,w\bigr)_h = \int_\Omega \widehat{\mathcal{I}}_h\bigl[v\cdot w\bigr]\dx = \sum_{T\in\Th}\sum_{z\in\mathscr{N}_h\cap\,T} \beta_z^T v|_T(z)\cdot w|_T(z) .
\end{align}
Furthermore, we denote by $V_h = \mathcal{S}^1(\Th)\times \mathcal{S}^1(\Th)$ the set of continuous $P1$ vector fields.

\subsection{Discrete Kirchhoff Elements}\label{sec:dkt}
To approximate the vertical component $w$ of the Föppl-von Kármán model we employ an $H^2$-nonconforming finite element discretization. For a given regular triangulation $\Th$ of $\Omega$ the $H^1$-conforming finite element spaces $W_h\subseteq H^1(\Omega)$ and $\Theta_h\subseteq H^1(\Omega,\mathbb{R}^2)$ are defined as
\begin{align*}
	W_h&:=\Bigl\{w_h\in C(\overline{\Omega})\,\Big|\,w_h|_T\in P^{\text{red}}_3(T)\,\,\,\forall T\in\Th,\,\nabla w_h\text{ is continuous at all }z\in\mathscr{N}_h\Bigr\},\\
	\Theta_h&:=\Bigl\{\theta_h\in C(\overline{\Omega},\mathbb{R}^2)\,\Big|\,\theta_h|_T\in P_2(T)^2\,\,\,\forall T\in\Th\Bigr\}.
\end{align*}
The space $P_k(T)$ denotes the set of polynomials with total degree smaller than or equal to $k\geq0$ restricted to the element $T\in\Th$. Furthermore, in the definition of
\begin{align*}
	P^{\text{red}}_3(T):=\Bigl\{p\in P_3(T)\,\Big|\,p(x_T)=\frac{1}{3}\sum_{z\in\mathscr{N}_h\cap \,T}\bigl(p(z)+\nabla p(z)\cdot(x_T-z)\bigr)\Bigr\},
\end{align*}
the degree of freedom in the center of mass $x_T=(1/3)\,\sum_{z\in\mathscr{N}_h\cap\,T}z$ of the triangle $T$ is eliminated. The remaining degrees of freedom of the space $W_h$ are given by the function values and the two partial derivatives at each vertex of the elements $T\in\Th$. This property is especially useful for the practical realization of the discrete model.

To relate the finite element spaces $W_h$ and $\Theta_h$ we introduce the operator $\nabla_h:W_h\to\Theta_h$ which defines a discretization of the deformation gradient. The property $\nabla_h w_h\in\Theta_h\subseteq H^1(\Omega,\mathbb{R}^2)$ for functions $w_h\in W_h$ ensures that second order derivatives can be discretized via the operation $\nabla\nabla_h w_h$.

\begin{defn}[See {\cite[Def. 8.6]{nonlinear_bartels}}]\thlabel{gradop}
For every $w_h\in W_h$ the \textit{discrete gradient operator} $\nabla_h:W_h\to\Theta_h$ is the uniquely defined function $\theta_h=\nabla_hw_h$ such that
\begin{align*}
\theta_h(z)&=\nabla w_h(z)&&\quad\forall z\in\mathscr{N}_h,\\
\theta_h(z_S)\cdot n_S&=\frac{1}{2}\,\bigl(\nabla w_h(z_S^1)+\nabla w_h(z_S^2)\bigr)\cdot n_S&&\quad\forall S\in\mathscr{S}_h,\\[2pt]
\theta_h(z_S)\cdot t_S&=\nabla w_h(z_S)\cdot t_S&&\quad\forall S\in\mathscr{S}_h.
\end{align*}
For every side $S\in\mathscr{S}_h$ we denote by $z_S^1,\,z_S^2\in\mathscr{N}_h$ its endpoints, by $z_S=(z_S^1+z_S^2)/2$ its midpoint and by $n_S,\,t_S\in\mathbb{R}^2$ the orthonormal vectors such that $n_S$ is normal to $S$.
\end{defn}

\begin{rmk}\thlabel{dkt_interpol}
The discrete gradient operator can be applied to functions $w\in H^3(\Omega)$ by defining a nodal interpolant $\mathcal{I}_h^\mathrm{dkt}w\in W_h$ via the conditions $\mathcal{I}_h^\mathrm{dkt}w(z)=w(z)$ and $\nabla \mathcal{I}_h^\mathrm{dkt} w(z)=\nabla w(z)$ for every node $z\in\mathscr{N}_h$. In particular, we have that $\nabla_hw=\nabla_h\mathcal{I}_h^\mathrm{dkt}w$.
\end{rmk}

As a direct consequence of the above definition we have for every $S\in\mathscr{S}_h$ that
\begin{align}
	\begin{split}\label{nablahw}
	\nabla_h w_h(z_S)&=\vphantom{\frac{1}{2}}\bigl(\nabla_h w_h(z_S)\cdot n_S\bigr)\,n_S+\bigl(\nabla_h w_h(z_S)\cdot t_S\bigr)\,t_S\\
	&=\bigl(\theta_h(z_S)\cdot n_S\bigr)\,n_S+\bigl(\theta_h(z_S)\cdot t_S\bigr)\,t_S\\
	&=\frac{1}{2}\,\bigl((\nabla w_h(z_S^1)+\nabla w_h(z_S^2))\cdot n_S\bigr)\,n_S+\bigl(\nabla w_h(z_S)\cdot t_S\bigr)\,t_S,
	\end{split}
\end{align}
where we split $\nabla_h w_h(z_S)$ in its orthogonal parts, namely the parts in the directions $n_S$ and $t_S$.
The discrete gradient operator possesses the following useful properties, cf.~\cite{nonlinear_bartels,braess}.

\begin{lem}[See {\cite[Lem. 8.1]{nonlinear_bartels}}]\thlabel{dkt_properties}(i) There exists a constant $c>0$ such that for every $w_h\in W_h$ and $T\in\Th$ we have for $\ell=0,1$ with $\nabla^0=I$ that
\begin{align*}
	c^{-1}\,||\nabla^{\ell+1}w_h||_{L^2(T)}\leq||\nabla^{\ell}\nabla_hw_h||_{L^2(T)}\leq c\,||\nabla^{\ell+1}w_h||_{L^2(T)}.
\end{align*}
(ii) There exists a constant $c>0$ such that for every $w\in H^3(\Omega)$ and $T\in\Th$ we have
\[
||\nabla_hw-\nabla w||_{L^2(T)}+h_T||\nabla\nabla_hw-D^2 w||_{L^2(T)}\leq ch_T^2||D^3w||_{L^2(T)}.
\]
(iii) There exists a constant $c>0$ such that for $w_h\in W_h$ and $2\leq p\leq \infty$ there holds
\begin{align*}
	||\nabla_hw_h-\nabla w_h||_{L^{^p}(T)}\leq c h_T ||D^2w_h||_{L^{^p}(T)},
\end{align*}
for every $T\in\Th$ with $h_T:=\text{diam}(T)$.
\item[(iv)] The mapping $w_h\rightarrow||\nabla\nabla_hw_h||$ defines a norm on the sets 
\begin{align*}
W_{0,h}^\mathrm{clamped}&:=\Bigl\{w_h\in W_h\,\Big|\,w_h(z)=0\,\text{ and }\,\nabla w_h(z)=0\,\,\,\forall z\in\mathscr{N}_h\cap\gamma_D\Bigr\},\\
W_{0,h}^\mathrm{simple}&:=\Bigl\{w_h\in W_h\,\Big|\,w_h(z)=0\,\,\,\forall z\in\mathscr{N}_h\cap\partial\Omega\Bigr\}.
\end{align*}
For $w_h\in W_{0,h}^\mathrm{clamped}$ we have $w_h|_{\gamma_D}=\nabla w_h|_{\gamma_D}=0$ and $\widetilde{w}_h\in W_{0,h}^\mathrm{simple}$ satisfies $\widetilde{w}_h|_{\partial\Omega}=0$.
\end{lem}

\section{Energy Decreasing Iteration}
This section deals with the analysis of the discrete gradient flow with regard to well-posedness, unconditional stability and monotone energy decay of the resulting iterates.

\begin{alg}[Decoupled gradient flow]\label{decoupled_grad_flow}
Specify an initial configuration $(u^0,w^0)\in\mathcal{A}$, an initial step size $\tau_1>0$, a stopping tolerance $\varepsilon_\mathrm{stop}>0$ and set $k=1$.\\[3pt]
(1) Compute $(u^k,w^k)\in\mathcal{A}$ such that
\begin{align*}
	\bigl(d_t w^k,v\bigr)_\mathrm{ver} = & -\bigl(D^2w^k - \alpha I , D^2v\bigr) - 2 \theta \bigl(|\nabla w^k|^2 \nabla w^k + \widetilde{\eps}(u^{k-1}) \nabla w^{k-1/2} , \nabla v\bigr) + \bigl(f,v\bigr),\\
	\bigl(d_t u^k,z\bigr)_\mathrm{hor}	= & -\theta \bigl(\widetilde{\eps}(u^k) , \widetilde{\eps}(z)\bigr) - \theta \bigl(\nabla w^k \otimes \nabla w^k , \widetilde{\eps}(z)\bigr),
\end{align*}
for all $(z,v)\in\mathcal{A}_0$.\\[3pt]
(2) Terminate the algorithm if $||d_t u^k||_\mathrm{hor}+||d_t w^k||_\mathrm{ver}\leq\varepsilon_\mathrm{stop}\min\{1,\tau_k\}$. Otherwise, define $\tau_{k+1}$, set $k\to k+1$ and continue with Step (1).
\end{alg}

The algorithm admits a sequence $(u^k,w^k)_{k\geq0}\subset\mathcal{A}$ that decreases the elastic energy.
For simplicity we assume from now on that $f=0$.

\begin{lem}[Energy decay]\label{lem:energy}
Iterates $(u^k,w^k)_{k\geq0}\subset\mathcal{A}$ of Algorithm \ref{decoupled_grad_flow} satisfy 
\[
E^\theta(u^k,w^k)+\sum_{k=1}^K\tau_k\Bigl(||d_tw^k||^2_\mathrm{ver}+||d_tu^k||^2_\mathrm{hor}\Bigr)\leq E^\theta(u^0,w^0).
\]
This implies that the updates $d_tu^k$ and $d_tw^k$ vanish as $k\to\infty$. Cluster points of the sequence $(u^k,w^k)$ thus become stationary configurations for the elastic energy $E^\theta$. 
\end{lem}
\begin{proof}
Testing the first and second equation of the iterative scheme of Algorithm \ref{decoupled_grad_flow} with $v=d_tw^k$ and $z=d_tu^k$, respectively, shows that
\begin{align*}
	||d_t w^k||_\text{ver}^2 = & -\bigl(D^2w^k - \alpha I , D^2 d_t w^k\bigr) - 2 \theta \bigl(|\nabla w^k|^2 \nabla w^k + \widetilde{\eps}(u^{k-1}) \nabla w^{k-1/2} , \nabla d_t w^k\bigr),\\ %+ \bigl(f,d_t w^k\bigr)\,,\\
	||d_t u^k||_\text{hor}^2	= & -\theta \bigl(\widetilde{\eps}(u^k) , \widetilde{\eps}(d_t u^k)\bigr) - \theta \bigl(\nabla w^k \otimes \nabla w^k , \widetilde{\eps}(d_t u^k)\bigr).
\end{align*}
The discrete product rule $d_t(a^kb^k)=(d_ta^k)b^k+a^{k-1}(d_tb^k)$ implies that we have
\[
\theta d_t \bigl(\widetilde{\eps}(u^k),\nabla w^k \otimes \nabla w^k\bigr) = \theta \bigl(\widetilde{\eps}(d_t u^k),\nabla w^k \otimes \nabla w^k\bigr) + \theta \bigl(\widetilde{\eps}(u^{k-1}), d_t(\nabla w^k \otimes \nabla w^k)\bigr),
\]
whereas the binomial formula $d_t|a^k|^2=2d_ta^k\cdot a^{k-1/2}$ yields 
\begin{align*}
\theta \bigl(\widetilde{\eps}(u^{k-1}),d_t(\nabla w^k \otimes \nabla w^k)\bigr) = 2 \theta \bigl(\widetilde{\eps}(u^{k-1}) \nabla w^{k-1/2} , \nabla d_t w^k\bigr).
\end{align*}
Combining the two equations, we arrive at
\[
\theta d_t \bigl(\widetilde{\eps}(u^k),\nabla w^k \otimes \nabla w^k\bigr) = \theta \bigl(\widetilde{\eps}(d_t u^k),\nabla w^k \otimes \nabla w^k\bigr) + 2 \theta \bigl(\widetilde{\eps}(u^{k-1}) \nabla w^{k-1/2} , \nabla d_t w^k\bigr).
\]
Since the backwards difference quotient is invariant under additive constants, we may incorporate the scaled identity matrix $\alpha I$ into the sequence $(D^2d_tw^k)_{k\geq0}$, in the sense that
\[
D^2 d_t w^k = d_t D^2 w^k = d_t(D^2w^k-\alpha I).
\]
By combining the above equation with the binomial formula $2d_ta^k\cdot a^k=d_t|a^k|^2+\tau_k|d_ta^k|^2$ we find that
\begin{align*}
\bigl(D^2w^k-\alpha I,D^2d_tw^k\bigr) &=  \bigl(D^2w^k-\alpha I,d_t(D^2w^k-\alpha I)\bigr)\\
&=\frac{1}{2}\bigl(d_t||D^2w^k-\alpha I||^2+\tau_k||d_t(D^2w^k-\alpha I)||^2\bigr).
\end{align*}
A summation of the discrete evolution equations and the convexity of $|\nabla w|^4$ show that
\begin{align*}
||d_tw^k||_\text{ver}^2&+\frac{1}{2}\bigl(d_t||D^2w^k-\alpha I||^2+\tau_k||d_t(D^2w^k-\alpha I)||^2\bigr)\\[-5pt]
&\hspace{2cm}+||d_tu^k||_\text{hor}^2+\frac{\theta}{2}\bigl(d_t||\widetilde{\eps}(u^k)||^2+\tau_k||\widetilde{\eps}(d_tu^k)||^2\bigr)\\[2pt]
&= -2 \theta \bigl(\widetilde{\eps}(u^{k-1}) \nabla w^{k-1/2} , \nabla d_t w^k\bigr) - \theta \bigl(\widetilde{\eps}(d_t u^k),\nabla w^k \otimes \nabla w^k\bigr)\\[2pt]
&\hspace{2cm} - 2\theta\bigl(|\nabla w^k|^2\nabla w^k,\nabla d_tw^k\bigr)\\[-2pt]
&\leq - \theta d_t \bigl(\widetilde{\eps}(u^k),\nabla w^k \otimes \nabla w^k\bigr) -\frac{\theta d_t}{2} \int_\Omega |\nabla w^k|^4\dx.
\end{align*}
Using that $\tau_k||d_t(D^2w^k-\alpha I)||^2+\tau_k||d_t\widetilde{\eps}(u^k)||^2\geq0$ and multiplying the above equation by $\tau_k$, a summation over $k=1,2,\ldots,K$ yields the desired estimate. 
\end{proof}

The iteration is unconditionally stable, but chosing large step sizes might lead to non-uniqueness of solutions. A mild condition on $\tau_k$ ensures well-posedness of Algorithm \ref{decoupled_grad_flow}.

\begin{prop}[Uniqueness of minimizers]\label{prop:uniqueness}
The iterates $(u^k,w^k)_{k\geq1}\subset\mathcal{A}$ of Algorithm \ref{decoupled_grad_flow} are unique, provided that $||D^2\cdot||\leq c_\text{eq}||\cdot||_\mathrm{ver}$ and
\[
\tau_k\leq\frac{1}{2c_0c_\text{eq}^2c_S^2}
\]
for every $k\geq1$ with a constant $c_0$ depening on $\theta,c_S,w_D$ and $E_0=E^\theta(u_0,w_0)$.
\end{prop}
\begin{proof}
The linear equation defining $u^k$ admits a unique solution by the Lax--Milgram lemma. The nonlinear equation with respect to $w^k$ defines an optimality condition for the minimization problem $\min_{w\in W}G(w)$, where the functional $G$ is defined as
\begin{align*}
G(w) &= \frac{1}{2\tau_k}||w-w^{k-1}||_\text{ver}^2+\frac{1}{2}||D^2w-\alpha I||^2+\frac{\theta}{2}\int_\Omega|\nabla w|^4\dx\\
&\hspace{2cm}+\frac{\theta}{2}\int_\Omega\widetilde{\varepsilon}(u^{k-1}):\bigl[\nabla(w+w^{k-1})\otimes\nabla(w+w^{k-1})\bigr]\dx.
\end{align*}
The existence of a minimizer $w^k$ follows by lower semicontinuity and coercivity on $H^2(\Omega)$. Lemma \ref{lem:energy} implies that $\theta||\widetilde{\varepsilon}(u^k)+\nabla w^k\otimes \nabla w^k||^2\leq 2E^\theta(u^k,w^k)\leq 2E_0$. We show that $||\widetilde{\varepsilon}(u^k)||\leq c_0$ with a constant $c_0$ defined below. Adding and subtracting $\nabla w^k\otimes \nabla w^k$ and $\nabla w_D$, the Sobolev inequality \eqref{sobolev} and the formula $(x+y)^2\leq 2x^2+2y^2$ yield
\begin{align*}
||\widetilde{\varepsilon}(u^k)|| &\leq ||\widetilde{\varepsilon}(u^k)+\nabla w^k\otimes \nabla w^k||+||\nabla w^k||^2_{L^4(\Omega)}\\
&\leq \sqrt{2E_0/\theta}+2||\nabla(w^k-w_D)||^2_{L^4(\Omega)}+2||\nabla w_D||^2_{L^4(\Omega)}\\
&\leq \sqrt{2E_0/\theta}+2c_S^2||D^2(w^k-w_D)||^2+2||\nabla w_D||^2_{L^4(\Omega)}=: c_0.
\end{align*}
Due to the above estimate we find with similar arguments and the Hölder inequality that
\begin{align*}
\frac{\theta}{2}\int_\Omega \widetilde{\varepsilon}(u^{k-1}):\nabla w&\otimes\nabla w\dx\geq-\frac{\theta}{2}||\widetilde{\varepsilon}(u^{k-1})||\,||\nabla w||^2_{L^4(\Omega)}\\
&\geq -||\widetilde{\varepsilon}(u^{k-1})||\,||\nabla (w-w^{k-1})||^2_{L^4(\Omega)}-||\widetilde{\varepsilon}(u^{k-1})||\,||\nabla w^{k-1}||^2_{L^4(\Omega)}\\[2pt]
&\geq -c_0c_S^2||D^2(w-w^{k-1})||^2-c_0||\nabla w^{k-1}||^2_{L^4(\Omega)}.
\end{align*}
The second term on the right-hand side is independent of $w$ whereas the first term can be absorbed using $||D^2\cdot||\leq c_{eq}||\cdot||_\text{ver}$, provided that $\tau_k\leq1/(2c_0c_{eq}^2c_S^2)$. In particular, the functional $G$ is strongly convex which shows the uniqueness of the minimizer $w^k$.
\end{proof}

To deal with the nonlinear system of equations in Algorithm \ref{decoupled_grad_flow} we employ a Newton scheme. In the following we assume for simplicity that $||\cdot||_\text{ver}=||D^2\cdot||$.

\begin{prop}[Newton scheme]\thlabel{newton_scheme}
Solving the first equation of Algorithm \ref{decoupled_grad_flow} is equivalent to seeking $w^k\in W$ such that $F_k(w^k)[v]=0$ for all $v\in W$ with
\begin{align*}
F_k(w)[v]&=\bigl(D^2[w-w^{k-1}],D^2v\bigr)+\tau_k\bigl(D^2w-\alpha I,D^2v\bigr)\\
&\hspace{1cm}+\tau_k\theta\bigl(\bigl[2|\nabla w|^2+\widetilde{\varepsilon}(u^{k-1})\bigr]\nabla w+\widetilde{\varepsilon}(u^{k-1})\nabla w^{k-1},\nabla v \bigr).
\end{align*}
The Newton scheme converges quadratically for sufficiently small step sizes.
\end{prop}
\begin{proof}
The equivalence of the equation $F_k(w)[v]=0$ to the first step of Algorithm \ref{decoupled_grad_flow} follows by the definition of the average $w^{k-1/2}=(w^k+w^{k-1})/2$. Proposition \ref{prop:uniqueness} implies that there exists a unique solution $w^k\in W$. The Fréchet derivative of $F_k$ along $z\in W$ reads
\begin{align*}
F_k'(w)[v,z]&=(1+\tau_k)\bigl(D^2z,D^2v\bigr)+\tau_k\theta\bigl([4\nabla w\otimes\nabla w+2|\nabla w|^2+\widetilde{\varepsilon}(u^{k-1})]\nabla z,\nabla v\bigr).
\end{align*}
The coercivity of $F_k'(w)$ follows by noting that with Hölder's inequality we arrive at
\begin{align*}
F_k'(w)[v,v]\geq(1+\tau_k)||D^2v||^2-\tau_k\theta||4\nabla w\otimes\nabla w+2|\nabla w|^2I_2+\widetilde{\varepsilon}(u^{k-1})||\,||\nabla v||^2_{L^4(\Omega)}
\end{align*}
for every $v\in W$ provided that $\tau_k$ is sufficiently small. The quadratic convergence of the Newton scheme then follows by continuity properties of $F_k$ and its derivatives.
\end{proof}

\section{$\Gamma$-Convergence}

The validity of the discrete approximation of the dimensionally reduced elastic energy for the singlelayer model has been rigorously justified in \cite{bartels_fvk} by means of $\Gamma$-convergence. Additional arguments are required to apply the results to the bilayer model considered in this work. We omit dealing with the forcing term whose convergence analysis is standard and recap that the discrete energy reads
\begin{align*}
E^{\theta}_h(u_h,w_h) =  \frac{1}{2} \int_{\Omega} |\nabla\nabla_h w_h - \alpha I |^2  \dx + \frac{\theta}{2} \int_{\Omega} \widehat{\mathcal{I}}_h\bigl[|\nabla w_h \otimes \nabla w_h + \widetilde{\eps}(u_h)|^2\bigr] \dx 
\end{align*}
for functions $(u_h,w_h)\in V_h\times W_h$. We start by introducing a compactness result that holds for appropriate assumptions on the boundary conditions stated in Section \ref{sec:model_derivation}.

\begin{lem}[Compactness]\thlabel{lem:compactness}
Let $(u_h,w_h)_{h>0}\subset V_h\times W_h$ be a sequence such that
\[
u_h = u_{D,h}+u_{0,h},\quad w_h = w_{D,h}+w_{0,h},
\]
where $(u_{D,h},w_{D,h})\in V_h\times W_h$ satisfy $u_{D,h}\to u_D$ in $H^1(\Omega;\mathbb{R}^2)$ and $w_{D,h}\to w_D$ in $W^{1,4}(\Omega)$ as $h\to0$, and $(u_{0,h},w_{0,h})\in V_h\times W_h$ admit homogeneous boundary conditions, i.e.,
\begin{align*}
u_{0,h}|_{\gamma_D}=0,\quad w_{0,h}|_{\gamma_D}=0,\quad \nabla w_{0,h}|_{\gamma_D} = 0,
\end{align*}
on a subset $\gamma_D\subset\partial\Omega$ of positive surface measure or
\begin{align*}
u_{0,h}|_{\partial\Omega}=0,\quad w_{0,h}|_{\partial\Omega}=0,
\end{align*}
on the whole boundary $\partial\Omega$ of $\Omega$. If $E^\theta_h(u_h,w_h)\leq c_E$ for all $h>0$ we have that
\[
||u_h||_{H^1(\Omega;\mathbb{R}^2)} + ||w_h||_{W^{1,4}(\Omega)} \leq c.
\]
\end{lem}
\begin{proof}
From the bound on the discrete energy and a binomial formula we find that
$\nabla_h w_h$ is bounded in $H^1(\Omega;\mathbb{R}^2)$. \thref{dkt_properties} implies that the mapping $w_h\to||\nabla\nabla_h w_h||$ defines a norm on the set of functions $w_h\in W_h$ that satisfy the assumed boundary conditions. In particular, the sequence $(w_h)_{h>0}$ is bounded in $W^{1,4}(\Omega)$. Since $\widetilde{\varepsilon}(u_h)$ is bounded in $L^2(\Omega,\mathbb{R}^2)$, Korn's inequality \eqref{korn_poincare} implies that the sequence $(u_h)_{h>0}$ is bounded in $H^1(\Omega;\mathbb{R}^2)$. 
\end{proof}

The compactness result ensures $\Gamma$-convergence of the discrete energy functionals $E_h^\theta$ to the continuous energy $E^\theta$ as $h\to0$. 

\begin{thm}[$\Gamma$-convergence]
(i) Let $(u_h,w_h)_{h>0}\subset V_h\times W_h$ be a sequence of functions with bounded energy, i.e., $E_h^\theta(u_h,w_h)\leq c_E$ for all $h>0$. Then there exists a pair $(u,w)\in H^1(\Omega;\mathbb{R}^2)\times W^{1,4}(\Omega)$ with $w\in H^2(\Omega)$ and a subsequence (not relabeled) such that
\[
(u_h,w_h)\to(u,w)\,\text{ in }\,H^1(\Omega;\mathbb{R}^2)\times W^{1,4}(\Omega)
\]
as $h\to0$ and
\[
E^\theta(u,w)\leq\liminf_{h\to0}E_h^\theta(u_h,w_h).
\]
(ii) For every $(u,w)\in H^1(\Omega;\mathbb{R}^2)\times H^2(\Omega)$ there exists a recovery sequence $(u_h,w_h)_{h>0}\subset V_h\times W_h$ such that
\[
(u_h,w_h)\to(u,w)\,\text{ in }\,H^1(\Omega;\mathbb{R}^2)\times W^{1,4}(\Omega)
\]
as $h\to0$ and
\[
\lim_{h\to0}E_h^\theta(u_h,w_h)=E^\theta(u,w).
\]
\end{thm}
\begin{proof}
We show that the error of numerical integration vanishes as $h\to0$, i.e.,
\begin{align*}
\int_{\Omega} \widehat{\mathcal{I}}_h\bigl[|\nabla w_h \otimes \nabla w_h + \widetilde{\eps}(u_h)|^2\bigr] \dx \to \int_{\Omega} |\nabla w_h \otimes \nabla w_h + \widetilde{\eps}(u_h)|^2 \dx.
\end{align*} 
Upon defining $\varphi_h=\nabla w_h \otimes \nabla w_h + \widetilde{\eps}(u_h)$ and applying an elementwise interpolation estimate, the chain rule, Hölder's inequality and the triangle inequality, we find that
\begin{align*}
\left|\int_{\Omega} \widehat{\mathcal{I}}_h|\varphi_h|^2-|\varphi_h|^2 \dx\right|
&\leq c \sum_{T\in\Th}h_T\left|\left|D|\varphi_h|^2\right|\right|_{L^1(T)}\\
&\leq 2c \sum_{T\in\Th}h_T^{1/2}\left|\left|\varphi_h\right|\right|_{L^2(T)}h_T^2\left|\left|D\varphi_h\right|\right|_{L^2(T)}\\
&\leq 2c h^{1/2}\left|\left|\varphi_h\right|\right|\sum_{T\in\Th}h_T^2\left|\left|D\varphi_h\right|\right|_{L^2(T)}\\
&\leq 4ch^{1/2}\left|\left|\varphi_h\right|\right|\Bigl(\sum_{T\in\Th}h_T||\nabla w_h||_{L^\infty(T)}^2||D^2w_h||_{L^2(T)}^2\Bigr)^{1/2}.
%&\leq 4ch^{1/2}\left|\left|\varphi_h\right|\right|\sum_{T\in\Th}h_T^2\Bigl(||\nabla w_h||_{L^\infty(T)}||D^2w_h||_{L^2(T)}+||D^2 u_h||_{L^2(T)}\Bigr).
\end{align*}
We used the fact that $u_h\in\mathcal{S}^1(\Th)^2$ implies $||D^2u_h||_{L^2(T)}=0$ on every $T\in\Th$. 
%Since $u_h\in\mathcal{S}^1(\Th)^2$ we have that $||D^2u_h||_{L^2(T)}=0$ for every $T\in\Th$.
The bound $E_h^\theta(u_h,w_h)\leq c_E$ for all $h>0$ yields that the terms including $\varphi_h$ and $D^2w_h$ are uniformly bounded by a constant. In addition, we are able to apply \thref{lem:compactness} which gives
\[
||u_h||_{H^1(\Omega;\mathbb{R}^2)} + ||w_h||_{W^{1,4}(\Omega)} \leq c.
\]
Hence, a final application of the inverse estimates $||\nabla w_h||_{L^\infty(T)}\leq ch_T^{-1/2}||\nabla w_h||_{L^4(T)}$ and $||D^2 u_h||_{L^2(T)}\leq ch_T^{-1}||\nabla u_h||_{L^2(T)}$ shows that the right-hand side tends to zero as $h\to0$.

(i) Let $(u_h,w_h)_{h>0}\subset V_h\times W_h$ be a sequence of functions satisfying $E_h^\theta(u_h,w_h)\leq c_E$ for all $h>0$. \thref{lem:compactness} implies that $||u_h||_{H^1(\Omega;\mathbb{R}^2)} + ||w_h||_{W^{1,4}(\Omega)} \leq c$. In particular, for some $Y\in H^1(\Omega;\mathbb{R}^2)$ and $w\in W^{1,4}(\Omega)$ we have for a subsequence (not relabeled) that
\begin{alignat*}{2}
\nabla_hw_h&\rightharpoonup Y&&\text{ in }H^1(\Omega;\mathbb{R}^2),\\
w_h&\rightharpoonup w&&\text{ in } W^{1,4}(\Omega),\\
\nabla_hw_h&\to \nabla w_h\quad&&\text{ in }L^4(\Omega;\mathbb{R}^2).
\end{alignat*}
We used items (i) and (iii) of \thref{dkt_properties} and an inverse estimate to find that
\[
||\nabla_hw_h-\nabla w_h||_{L^4(\Omega)}\leq ch^{1/2}||\nabla\nabla_hw_h||_{L^2(\Omega)}\to0
\]
as $h\to0$. The uniqueness of weak limits implies that $\nabla w =Y$, hence $w\in H^2(\Omega)$. Since $\nabla_hw_h\to Y=\nabla w$ in $L^4(\Omega;\mathbb{R}^2)$ we also have that $\nabla w_h\to\nabla w$ in $L^4(\Omega;\mathbb{R}^2)$. Similar arguments show that there exists $u\in H^1(\Omega;\mathbb{R}^2)$ such that $u_h\to u$ in $H^1(\Omega;\mathbb{R}^2)$. Since the error of numerical integration vanishes, strong convergence of $\nabla w_h\to\nabla w$ in $L^4(\Omega;\mathbb{R}^2)$ and $u_h\to u$ in $H^1(\Omega;\mathbb{R}^2)$ together with weak lower semicontinuity of the convex terms shows the desired estimate.

(ii) We argue by density of smooth functions to assume that $(u,w)\in H^2(\Omega;\mathbb{R}^2)\times H^3(\Omega)$ and construct a recovery sequence $(u_h,w_h)_{h>0}$ via interpolation of $(u,w)$ by defining $u_h=\mathcal{I}_hu\in V_h$ and $w_h=\mathcal{I}_h^\text{dkt}w\in W_h$. Here, $\mathcal{I}_h$ denotes the standard nodal interpolation operator while $\mathcal{I}_h^\text{dkt}$ is defined in \thref{dkt_interpol}. \thref{dkt_properties} implies the interpolation estimate
\[
||\nabla_hw-\nabla w||_{L^2(T)}+h_T||\nabla\nabla_hw-D^2 w||_{L^2(T)}\leq ch_T^2||D^3w||_{L^2(T)}.
\]
Upon recalling that $\nabla_hw=\nabla_h\mathcal{I}_h^\text{dkt}w$ for $w\in H^3(\Omega)$, a summation over all elements $T\in\Th$ yields that $\nabla\nabla_h w_h\to D^2 w$ in $L^2(\Omega)$ as $h\to0$. In particular, this shows that
\[
\int_{\Omega} |\nabla\nabla_h w_h - \alpha I |^2  \dx \to \int_{\Omega} |D^2 w - \alpha I |^2  \dx
\] 
as $h\to0$. From standard interpolation estimates for $w_h$ and $u_h$ we also obtain that %\marginpar{$\nabla w_h$ to $\nabla w$ in $L^4$?}
\[
\int_{\Omega} |\nabla w_h \otimes \nabla w_h + \widetilde{\eps}(u_h)|^2 \dx \to \int_{\Omega} |\nabla w \otimes \nabla w + \widetilde{\eps}(u)|^2 \dx
\]
as $h\to0$. Combining the above results shows the asserted convergence property.
\end{proof}

\section{Numerical Experiments} 
\subsection{Implementation}
In this section we introduce the discrete version of the proposed gradient flow and analyse the performance of the corresponding \textsc{Matlab} implementation. If $D^2_h=\nabla\nabla_h$ denotes the discrete Hessian, the iterative scheme is defined via inner products
\[
(w_h,v_h)_\text{ver} = \bigl(D_h^2w_h,D_h^2v_h\bigr),\quad (u_h,z_h)_\text{hor} = \bigl(\widetilde{\varepsilon}(u_h),\widetilde{\varepsilon}(z_h)\bigr)
\]
for discrete functions $w_h,v_h\in W_h$ and  $u_h,z_h\in V_h$ and the $L^2$ inner product $(\cdot,\cdot)_h$ specified in \eqref{discrete_product}. For approximations $u_{D,h}\in V_h$ and $w_{D,h}\in W_h$ of some boundary data, we define the discrete set of admissible functions as
\[
\mathcal{A}_h = \mathcal{A}_{0,h}+(u_{D,h},w_{D,h}),
\]
where $\mathcal{A}_{0,h}=V_{0,h}\times W_{0,h}$ is a linear subspace subject to homogeneous boundary conditions. This leads to the following scheme for the discrete decoupled gradient flow.

\begin{alg}[Discrete decoupled gradient flow]\thlabel{discr_decoupled_grad_flow}
Specify an initial configuration $(u^0_h,w^0_h)\in\mathcal{A}_h$, an initial step size $\tau_1>0$, stopping tolerances $\varepsilon_\mathrm{Newton},\varepsilon_\mathrm{stop}>0$ and set $k=1$.\\[3pt]
(1a) Decrease $\tau_k$ until the Newton scheme computes $w^k_h\in [W_{0,h}+w_{D,h}]$ within a tolerance $||D_h^2d_tw_h^k||\leq\varepsilon_\mathrm{Newton}$ and a maximum number $N>0$ of iterations such that
\begin{align*}
	\bigl(D_h^2d_t w_h^k,D_h^2v_h\bigr) &= -\bigl(D_h^2w_h^k - \alpha I , D_h^2v_h\bigr)\\
	&\hspace{1cm} - 2 \theta \bigl(|\nabla w_h^k|^2 \nabla w_h^k + \widetilde{\eps}(u_h^{k-1}) \nabla w_h^{k-1/2} , \nabla v_h\bigr)_h + \bigl(f,v_h\bigr)_h
\end{align*}
for all $v_h\in W_{0,h}$.\\[3pt]
(1b) Compute $u^k_h\in [V_{0,h}+u_{D,h}]$ such that
\begin{align*}
	\bigl(d_t \widetilde{\varepsilon}(u_h^k),\widetilde{\varepsilon}(z_h)\bigr)	= -\theta \bigl(\widetilde{\eps}(u_h^k) , \widetilde{\eps}(z_h)\bigr)_h - \theta \bigl(\nabla w_h^k \otimes \nabla w_h^k , \widetilde{\eps}(z_h)\bigr)_h
\end{align*}
for all $z_h\in V_{0,h}$.\\[3pt]
(2) Terminate the algorithm if $||d_t \widetilde{\varepsilon}(u_h^k)||+||d_t D^2_hw_h^k||\leq\varepsilon_\mathrm{stop}\min\{1,\tau_k\}$. Otherwise, define 
\[
\tau_{k+1}=\min\bigl\{2\tau_k,\tau_\mathrm{max}\bigr\},
\]
set $k\to k+1$ and continue with Step (1a).
\end{alg}

For the following numerical experiments we use a maximum number of Newton iterations of $N=5$, a maximum step size $\tau_\text{max}=10^5$, stopping tolerances $\varepsilon_\text{Newton}=10^{-5}$, $\varepsilon_\text{stop}=10^{-12}$ and $\alpha = 1$.

We analyse the performance of the iterative scheme by running the discrete algorithm for an initially flat configuration on a uniform grid of meshsize $h=0.05$ and a prestrain parameter $\theta=1000$. The left plot in Figure \ref{fig:flat_circle_energy} shows the development of the energies $E_h^\theta(w_h^k,u_h^k)$ for the first fifty iterations $1\leq k\leq50$. We observe that the scheme effectively reduces the elastic energy. On the right-hand side of Figure \ref{fig:flat_circle_energy} the automatically chosen step sizes are visualized. The adaptive time stepping scheme decreases the step size for the first iteration to $\tau_1=0.125$. The step sizes then gradually increase until they reach the maximum value $\tau_\text{max}=10^5$ after 23 iterations.

\begin{figure}[H]
\centering
\begin{minipage}{0.49\textwidth}
\centering
\begin{tikzpicture}[scale=0.8]
\begin{axis}[
    xlabel={Iterations},
    xtick = {1,20,40,60,80,100},
    xticklabels = {$1$,$20$,$40$,$60$,$80$},
    legend pos=north east,
    ymajorgrids=false,
    xmajorgrids=false,
    grid style=dashed,
    legend style={nodes={scale=0.9, transform shape}},
    legend cell align={left},
    xmax = 50
]
\addplot[
    color=black
    ]
    table [x=iters,y = ener] {material/circ005_flat_theta1000_eps14_new_short.txt};
    \legend{$E_h^\theta(w_h^k{,}u_h^k)$}
\end{axis}
\end{tikzpicture}
\end{minipage}
\begin{minipage}{0.49\textwidth}
\centering
%\vspace*{-0.35cm}
\begin{tikzpicture}[scale=0.8]
\begin{axis}[
		every y tick scale label/.style={at={(yticklabel* cs:0.96,0cm)},anchor=near yticklabel},
    xlabel={Iterations},
    xtick = {1,20,40,60,80,100},
    xticklabels = {$1$,$20$,$40$,$60$,$80$},
    legend pos=south east,
    ymajorgrids=false,
    xmajorgrids=false,
    grid style=dashed,
    legend style={nodes={scale=1, transform shape}},
    legend cell align={left},
    xmax = 50,
    ymax = 115000
]
\addplot[const plot,
    color=black
    ]
    table [x=iters,y = taus] {material/circ005_flat_theta1000_eps14_new_short.txt};
    \legend{$\tau_k$}
\addplot[color=black,
    dashed
    ]
    coordinates {(1,100000) (23,100000)};
    \addlegendentry{$\tau_\text{max}$}
\end{axis}
\end{tikzpicture}
\end{minipage}
\caption{Energy development of the sequence $(w_h^k,u_h^k)_{k\geq1}$ and step sizes $(\tau_k)_{k\geq1}$ of the iterative scheme for the first fifty iterations with $\theta=1000$ and $h=0.05$.}
\label{fig:flat_circle_energy}
\end{figure}
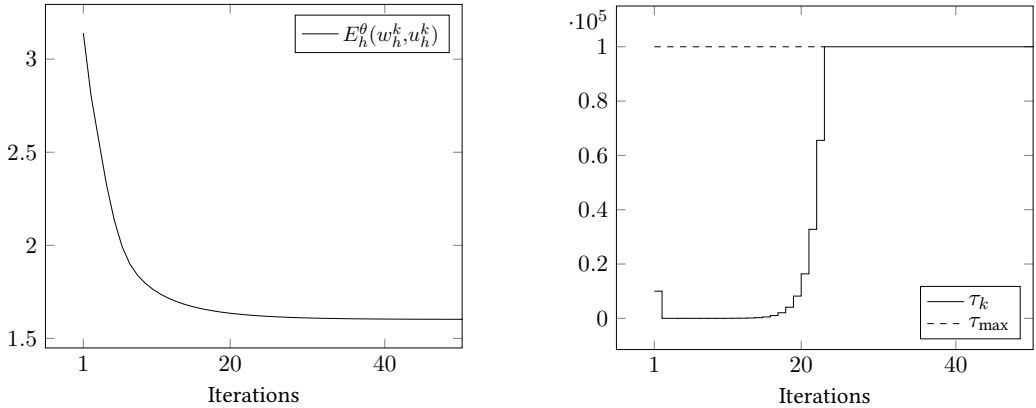

\subsection{Controlled Sphere-Cylinder Transition}
The amount of prestrain specified by the parameter $\theta$ dictates the shape of minimizers for the elastic energy functional. In theory, minimizers approach spherical configurations as in the linearized von Kármán model for small values of $\theta$. On the other hand, energy optimal configurations are expected to obtain cylindrical shapes as in the linearized Kirchhoff model for large values of $\theta$. To investigate the various shapes an initially flat approximation of the unit disc $\Omega=B_1(0)$ is considered. It is defined by the initial configurations
\[
u^0(x)=\begin{bmatrix}
0\\0
\end{bmatrix},\quad w^0(x)=0,\quad x\in B_1(0).
\]
We run the iterative scheme (\thref{discr_decoupled_grad_flow}) for $\alpha=1$ without applying any body force or boundary conditions, i.e., we set $f=0$, $\gamma_D=\emptyset$, and  $u_D = 0$ and $w_D=0$. To ensure uniqueness of discrete solutions an $L^2$-contribution is included in the discrete gradient flow. In particular, we define
\begin{align}\label{l2_gradflow}
(w_h,v_h)_\text{ver} = \bigl(D_h^2w_h,D_h^2v_h\bigr)+\bigl(w_h,v_h\bigr),\quad (u_h,z_h)_\text{hor} = \bigl(\widetilde{\varepsilon}(u_h),\widetilde{\varepsilon}(z_h)\bigr)+\bigl(u_h,z_h\bigr).
\end{align}
The corresponding numerical solutions for a grid of meshsize $h = 0.05$ and different values $\theta\in\{1,300,350,1000\}$ are visualized in Figure \ref{fig:flat_circle}.

\begin{figure}[H]
\centering
\begin{minipage}{0.49\textwidth}
\includegraphics[width=1\textwidth]{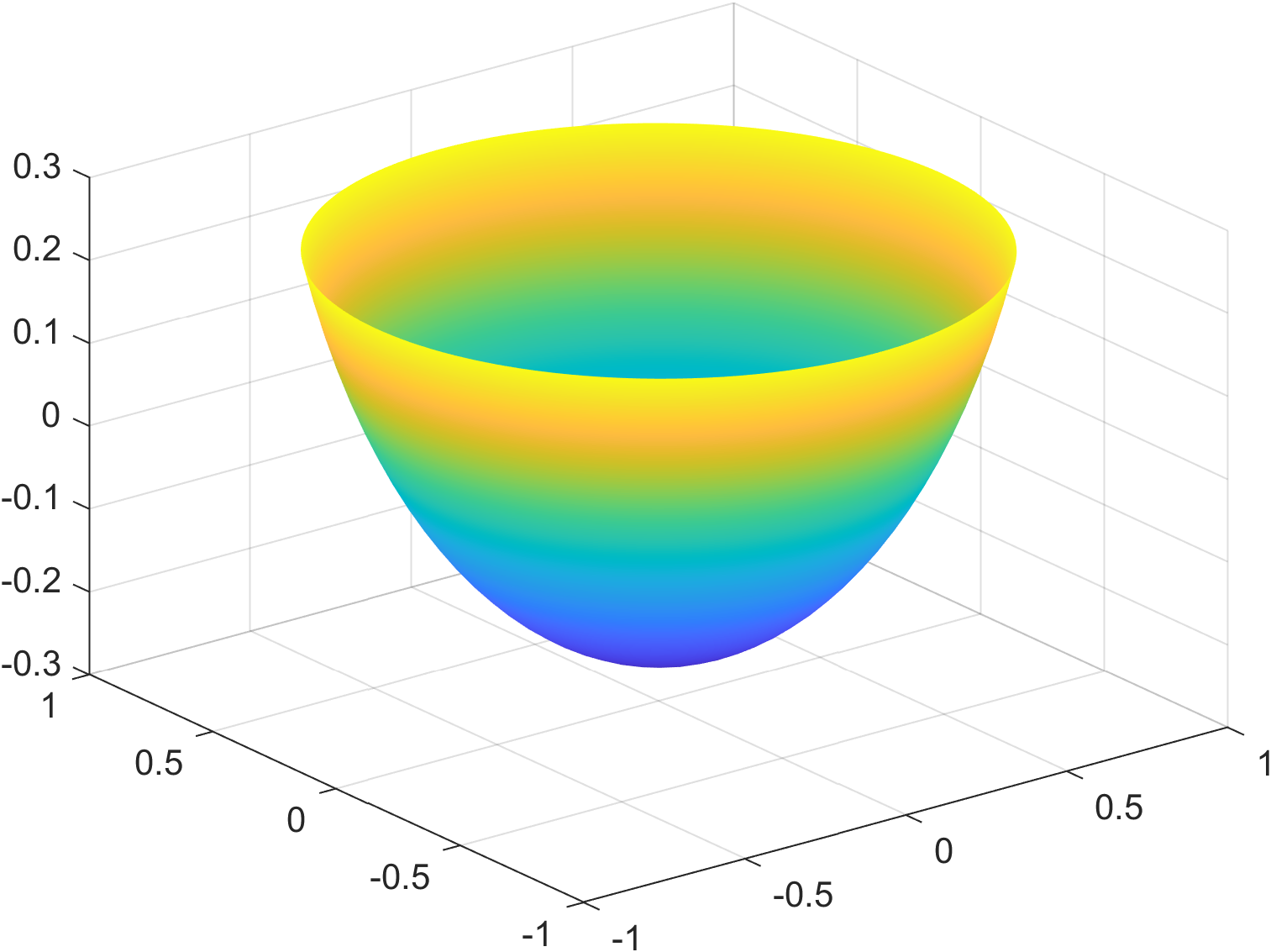}
\end{minipage}
\begin{minipage}{0.49\textwidth}
\includegraphics[width=1\textwidth]{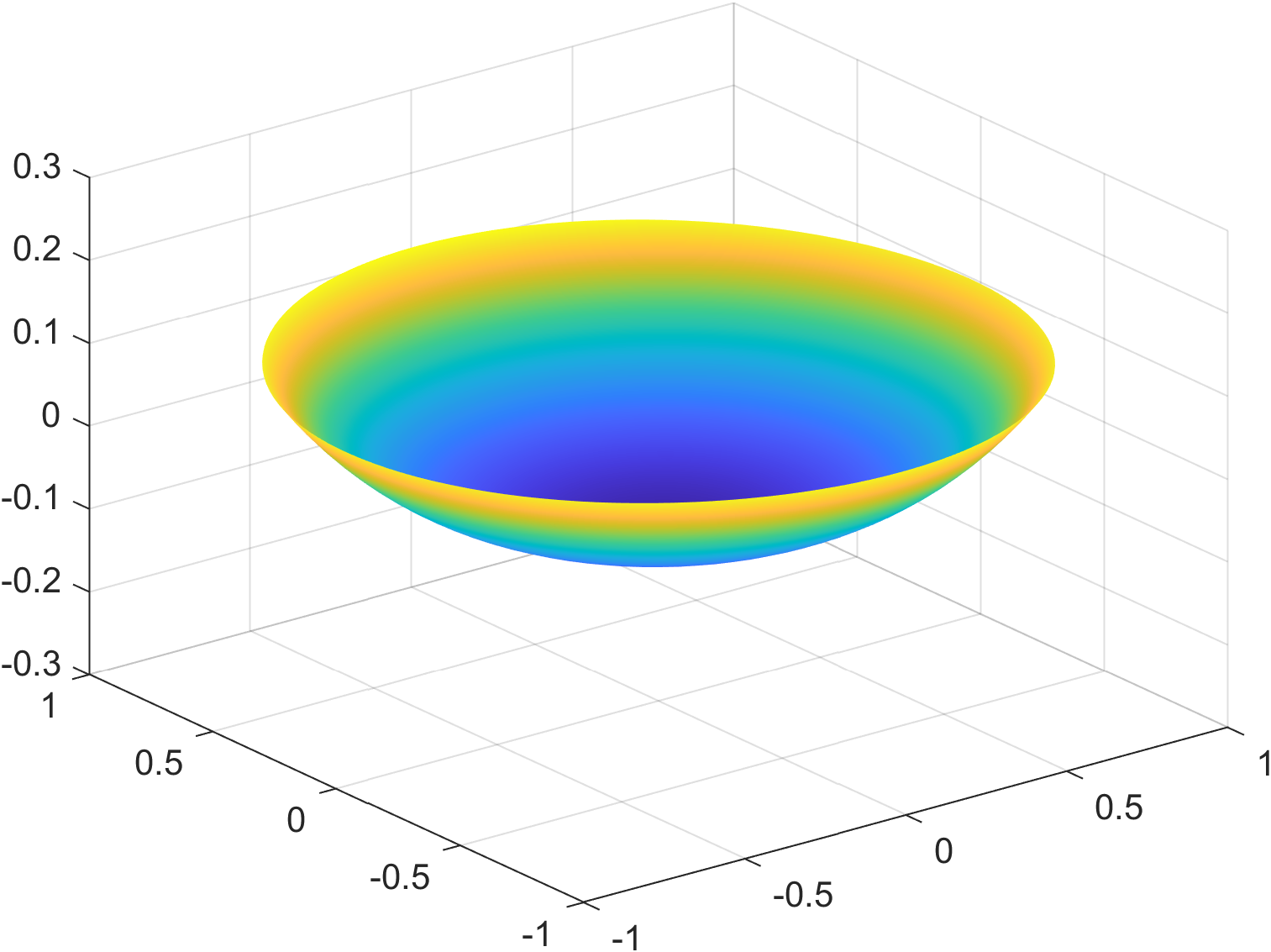}
\end{minipage}
\begin{minipage}{0.49\textwidth}
\includegraphics[width=1\textwidth]{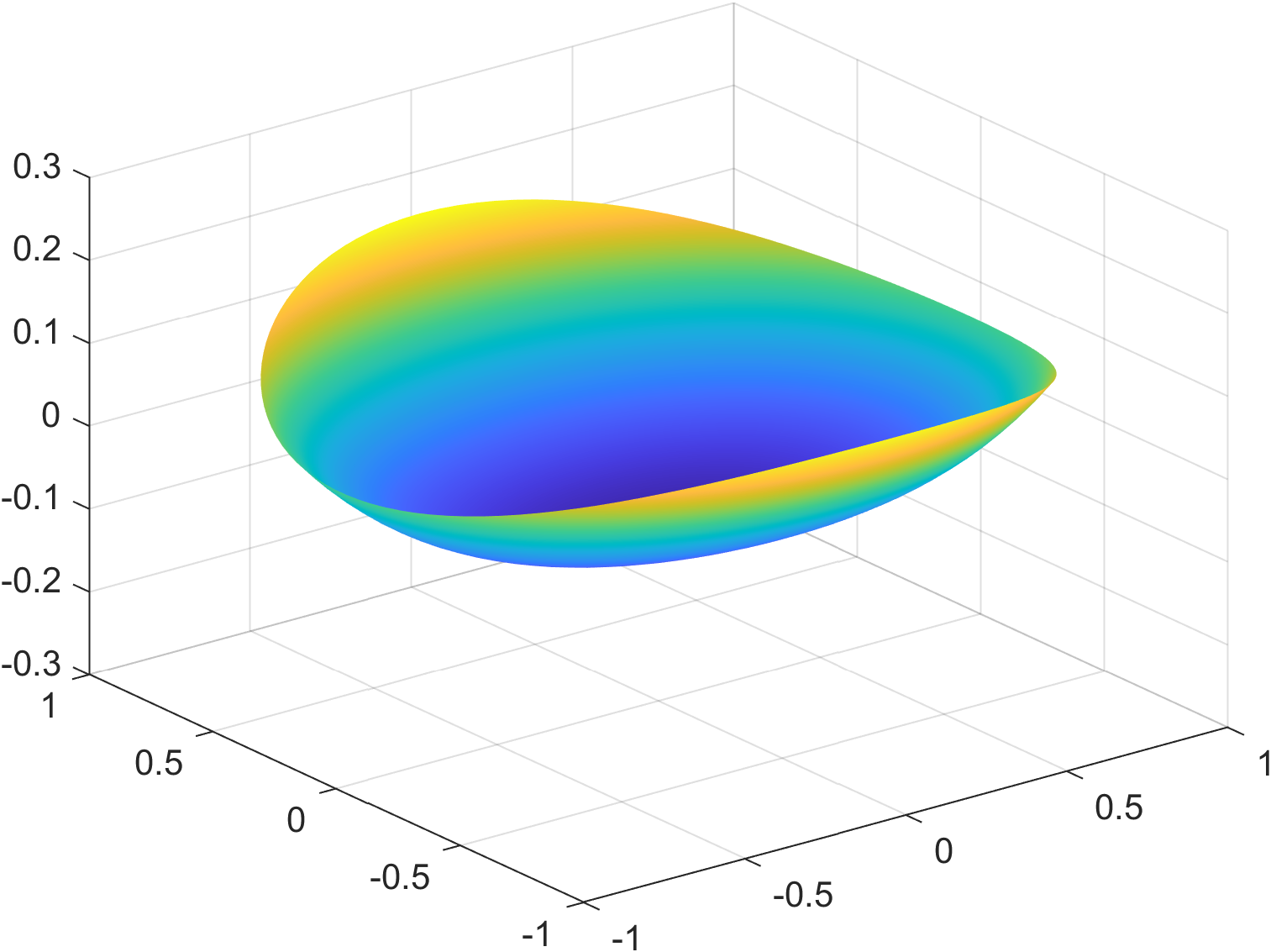}
\end{minipage}
\begin{minipage}{0.49\textwidth}
\includegraphics[width=1\textwidth]{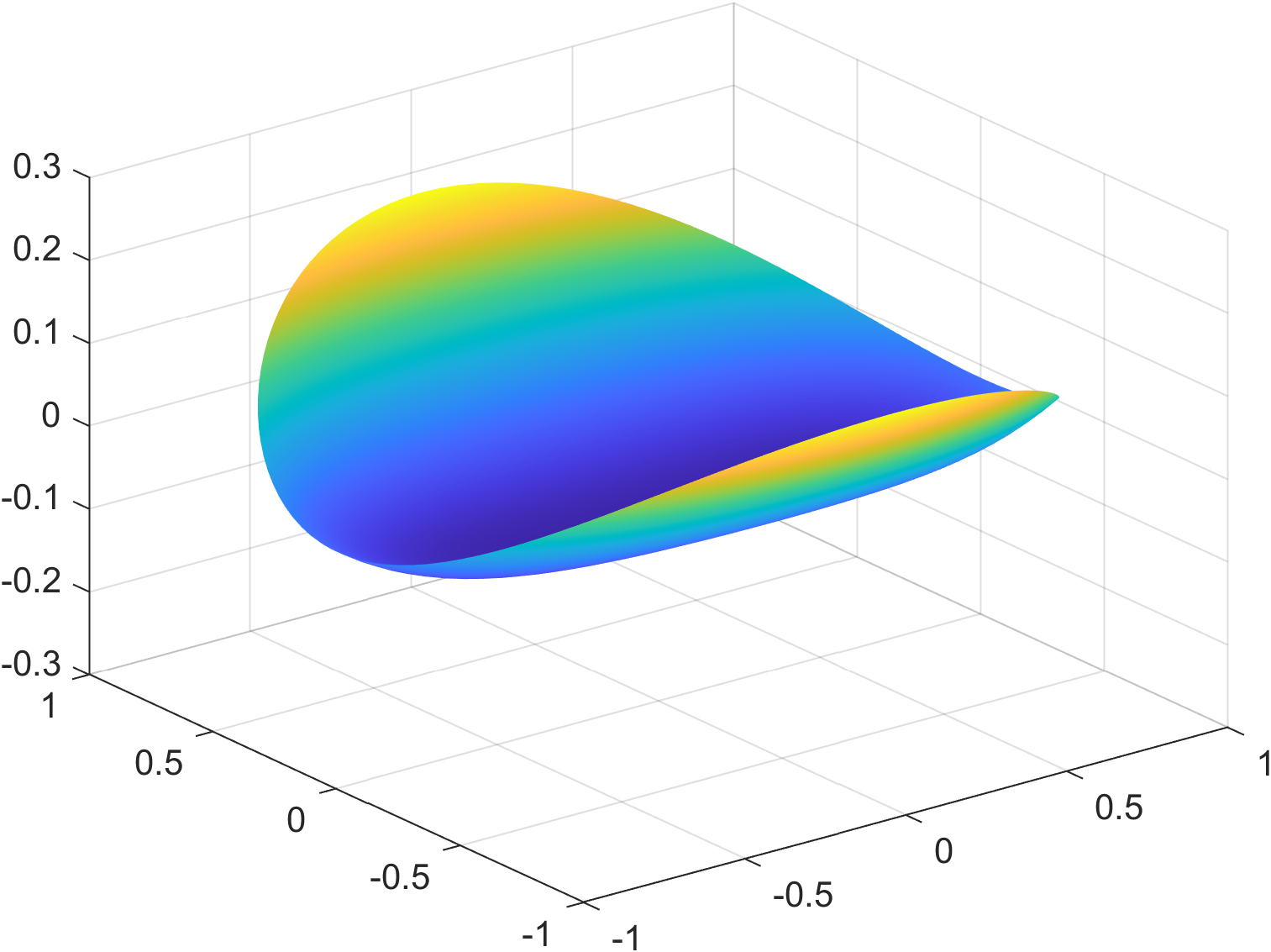}
\end{minipage}
\caption{Numerical solutions of \thref{discr_decoupled_grad_flow} for an initially flat disk with $\theta=1$ (top left), $\theta=300$ (top right), $\theta=350$ (bottom left) and $\theta=1000$ (bottom right). The colors represent the vertical magnitude from dark (lowest) to bright (highest). {\cc A transition from spherical to cylindrical configurations for increasing values of $\theta$ can be observed.}}
\label{fig:flat_circle}
\end{figure}

As can be seen in Figure \ref{fig:flat_circle} we obtain a spherical shape for small values of $\theta$ which coincides with the theoretical conjectures. The configuration flattens as $\theta$ increases until some critical range of values $\theta\in[250,350]$, at which a transition to a cylindrical shape takes place. The transformation occurs abruptly and sooner than theoretical observations might suggest, as minimizers of the Föppl--von Kármán model are known to approximate cylindrical minimizers of the linearized Kirchhoff model only in the limit $\theta\to\infty$. 

So far we identified low-energy solutions of the discrete model for some fixed values $\theta$. The remainder of this section is devoted to analyse the transition between the spherical and the cylindrical shape for intermediate values of the parameter $\theta$. We aim at identifying some critical point $\theta_\text{crit}$ at which a drastic deviation takes place. Our previous experiments suggest that $\theta_\text{crit} \approx 300$. The graphs in Figure \ref{fig:flat_circle_curvature} display the directional mean curvatures and the quotient 
\begin{align}\label{qsym}\cc
q_{\text{sym}}=\frac{\max u_1 -\min u_1}{\max u_2-\min u_2}
\end{align}
of the directional in-plane deflections of the final configurations for $\theta=1,2,\ldots,400$. {\cc The value $q_{\text{sym}}$ quantifies symmetric properties of the in-plane deformation $u$, where larger deviations from the value one indicate more asymmetric, and in our case more cylindrical, states.} We also include results using larger increments $\theta=400,450,\ldots,600$ to reduce the overall computation times.
%% Curvatures
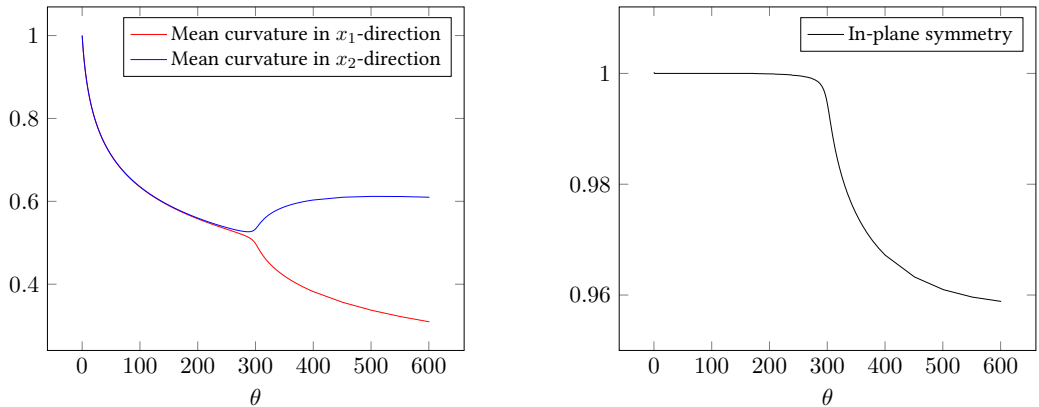
\begin{figure}
\centering
\begin{minipage}{0.49\textwidth}
\centering
\begin{tikzpicture}[scale=0.8]
\begin{axis}[
    xlabel={$\theta$},
    legend pos=north east,
    ymajorgrids=false,
    xmajorgrids=false,
    grid style=dashed,
    legend style={nodes={scale=0.9, transform shape}},
    legend cell align={left}
]
\addplot[
    color=red
    ]
    table [x = thetas, y expr=\thisrow{mean_curvatures_1_thetas}] {material/circ005thetas0_1_400_50_600.txt};
    \legend{Mean curvature in $x_1$-direction}
\addplot[
    color=blue
    ]
    table [x = thetas, y expr=\thisrow{mean_curvatures_2_thetas}] {material/circ005thetas0_1_400_50_600.txt};
    \addlegendentry{Mean curvature in $x_2$-direction}
\end{axis}
\end{tikzpicture}
\end{minipage}
%% Change this to axis deflection
\begin{minipage}{0.49\textwidth}
\centering
\begin{tikzpicture}[scale=0.8]
\begin{axis}[
    xlabel={$\theta$},
    ymin = 0.95, ymax = 1.012,
    legend pos=north east,
    ymajorgrids=false,
    xmajorgrids=false,
    grid style=dashed,
    legend style={nodes={scale=0.9, transform shape}},
    legend cell align={left}
]
\addplot[
    color=black
    ]
    table [x = thetas, y expr=1/\thisrow{symmetry}] {material/circ005thetas0_1_400_50_600.txt};
    \legend{In-plane symmetry}
\end{axis}
\end{tikzpicture}
\end{minipage}
\caption{Development of the directional mean curvatures (left) and the quotient $q_{\text{sym}}$, see \eqref{qsym}, of the directional in-plane deflections (right) of the final configurations for $\theta=1,2,...,400$ and $\theta=400,450,...,600$. {\cc A critical regime at around $\theta=300$ appears from which on a stark break of symmetry occurs.}}
\label{fig:flat_circle_curvature}
\end{figure}
We observe, that the directional curvatures remain nearly equal and the in-plane symmetry is almost constant for $\theta\lesssim200$. For intermediate values $200\lesssim\theta\lesssim300$ slight deviations appear, whereas, for values greater than $300$, the directional mean curvatures split apart and the symmetric properties decrease at a high rate. This is coherent with the observations from Figure \ref{fig:flat_circle}, in which the spherical deformations flatten and stay symmetric at first but rapidly switch to a cylindrical shape at some point.

Similar experiments have been carried out in \cite{schmidt_multilayers} with a different scaling of the elastic energy. In particular, the authors considered the energy $\widetilde{E}^\lambda$ defined by
\begin{align*}
	\widetilde{E}^{\lambda}(u,w) = \frac{1}{24} \int\limits_{\Omega} |D^2w - I |^2  \dx
	+ \frac{\lambda}{8} \int\limits_{\Omega} |\nabla w \otimes \nabla w + \widetilde{\eps}(u)|^2 \dx.
\end{align*}
To draw a comparison between our experiments and their numerical results, we note that rescaling the elastic energy $\widetilde{E}$ by a constant $C$ (in this case $C = 12$) has no effect on its minimizers. Matching the parameters $\theta$ and $\lambda$ of the membrane energies in $E$ and $\widetilde{E}$, respectively, yields the relation $\theta = 3 \lambda$. The critical value the authors achieved, at which the transition between spherical and cylindrical shape takes place, reads $\lambda_{\text{crit}} \approx 86$, hence $3\lambda_{\text{crit}} \approx 258$. It is difficult to quantify such a critical value in our numerical experiments since the transition is much smoother than the one observed in the first experiment of \cite{schmidt_multilayers} for an initial flat configuration. However, the results observed therein suggest a very similar range of critical values. The developments of the directional mean curvatures and the symmetric properties agree with the results of the authors' second test in which an initial configuration with the shape of a potato chip is used instead. 

\subsection{Curvature Inversion} 
The identity matrix $I$ in the first part of the energy enforces a specific bending behaviour of energy optimal configurations. The strength of the prescribed curvature can be modified by scaling the matrix $I$ by a factor $\alpha\in\mathbb{R}$. In particular, we consider the energy
\begin{align*}
	E^{\alpha}(u,w) = \frac{1}{2} \int_{\Omega} |D^2w - \alpha I |^2  \dx
	+ \frac{\theta}{2} \int_{\Omega} |\nabla w \otimes \nabla w + \widetilde{\eps}(u)|^2 \dx
\end{align*}
in which the forcing term is neglected. The sign of the factor $\alpha$ correlates to the sign of the specified surface curvature. The case $\alpha=0$ is related to the corresponding singlelayer model which has been studied in \cite{bartels_fvk}. We address in this section the inversion of the curvature which is modeled by letting $\alpha=1$ go to $\alpha=-1$. Similar to the previous experiment neither an external force nor boundary conditions are applied, i.e., $f=0$ and $\gamma_D=\emptyset$. In addition to the $L^2$-contribution in the discrete gradient flow \eqref{l2_gradflow}, we fix the vertical deflection of the midpoint $x_0=[0,0]\T$ of the mesh to zero by prescribing $w(x_0)=0$. \thref{discr_decoupled_grad_flow} is run with an initial flat configuration and $\alpha=1$. Next, the parameter $\alpha$ is decreased by 0.05 and the algorithm is reapplied to the resulting configuration of the previous step. This process is repeated until $\alpha=-1$ is reached. Figure \ref{fig:curvature_inversion} shows the numerical solutions for $\theta\in\{0,1000\}$ and $\alpha\in\{1,0.7,0.3,-1\}$. %The colors represent the strain density from dark (lowest) to bright (highest).\newline
It is apparent from Figure \ref{fig:curvature_inversion} that for $\theta=0$ and any $\alpha\in[-1,1]$ the strain density is distributed radially symmetric across the plate. The center always contains the least amount of strain while the values progressively increase towards the boundary of the object. When $\alpha=0$ we obtain a completely flat configuration. For $\theta=1000$ and $\alpha=1$ the algorithm produces a cylindrical shaped object which is coherent with the observations of the previous experiments. For decreasing $\alpha$ the configurations flatten along paths with the largest directional curvature until all corner points of the plate reach an equal height for $\alpha\approx0.3$. The numerical solutions then flatten evenly to a planar state for $\alpha\to0$. In both cases $\theta\in\{0,1000\}$ the processes are reversed for negative values of the parameter $\alpha$ with a change in the sign of the surface curvature.

\begin{figure}
\centering
\begin{minipage}{0.49\textwidth}
\includegraphics[width=0.95\textwidth]{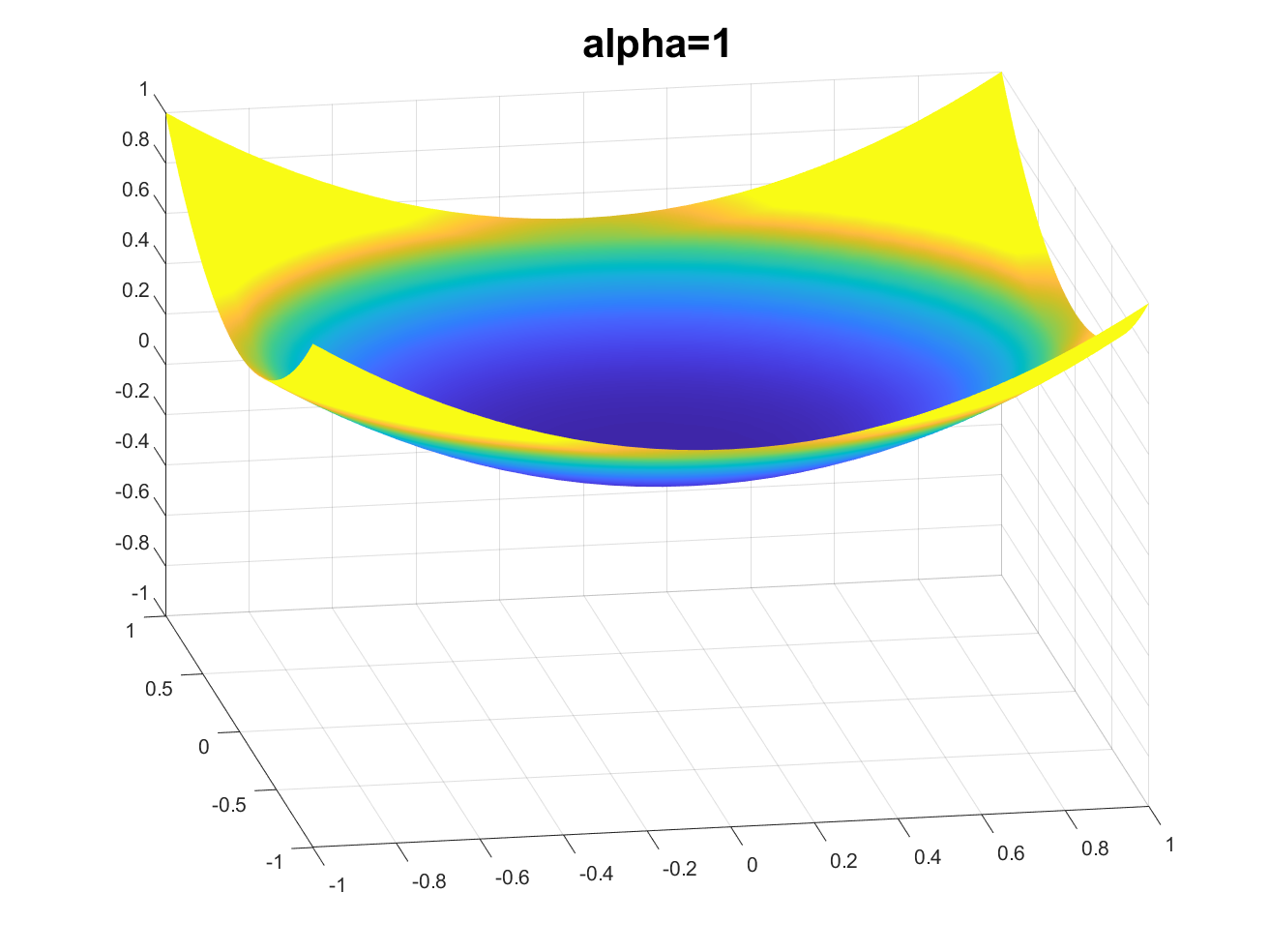}
\end{minipage}
\begin{minipage}{0.49\textwidth}
\includegraphics[width=0.95\textwidth]{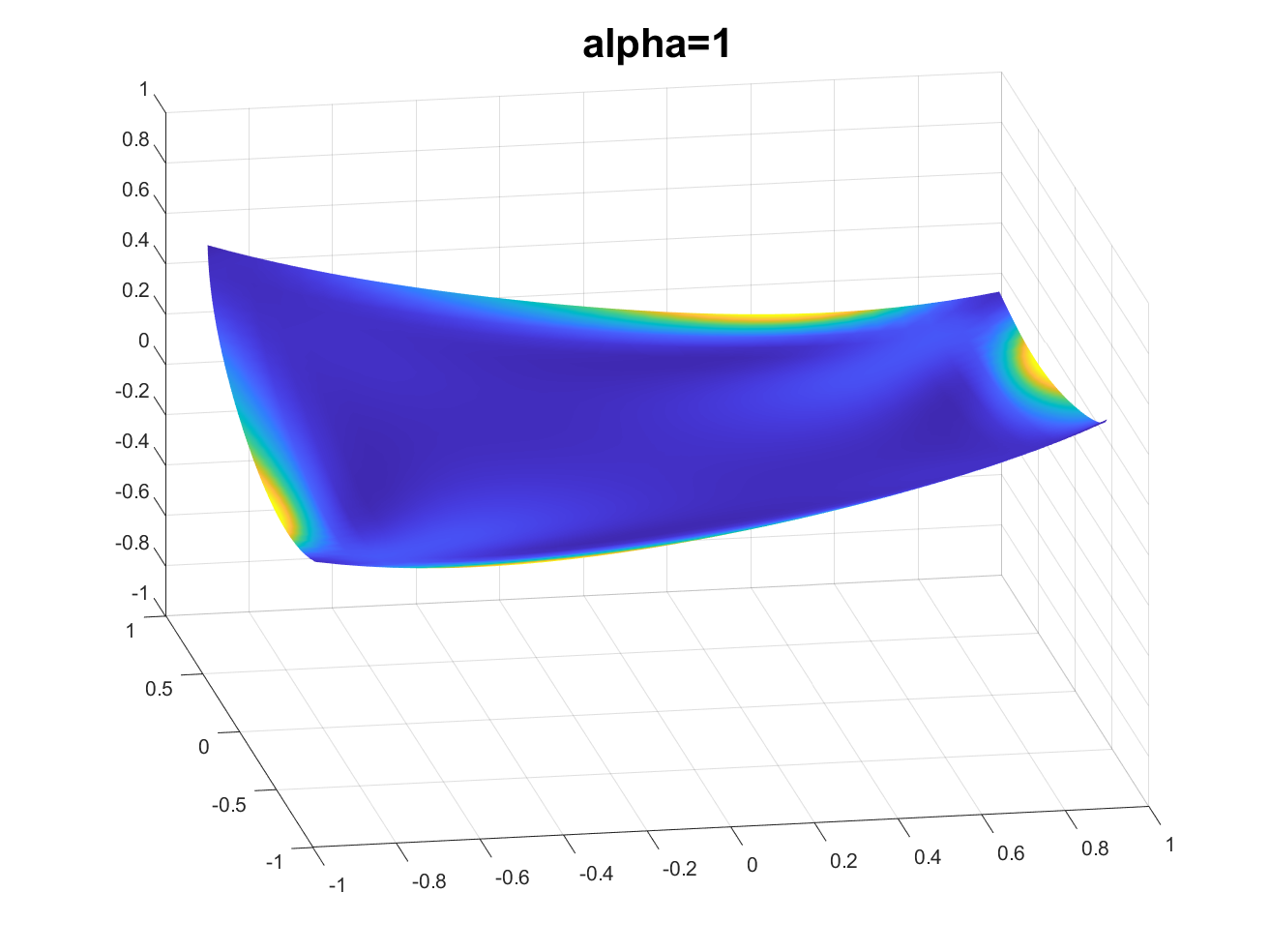}
\end{minipage}
\begin{minipage}{0.49\textwidth}
\includegraphics[width=0.95\textwidth]{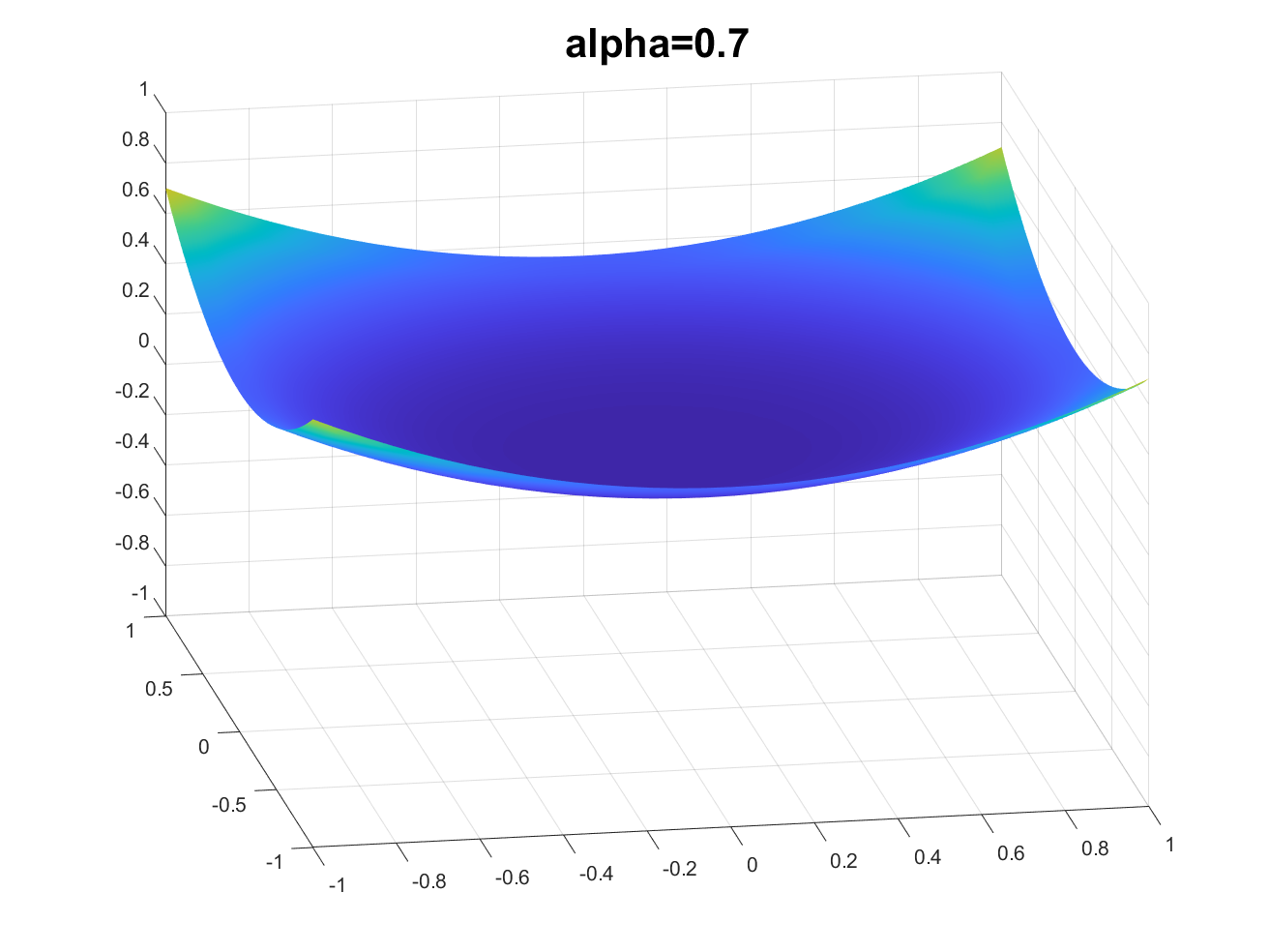}
\end{minipage}
\begin{minipage}{0.49\textwidth}
\includegraphics[width=0.95\textwidth]{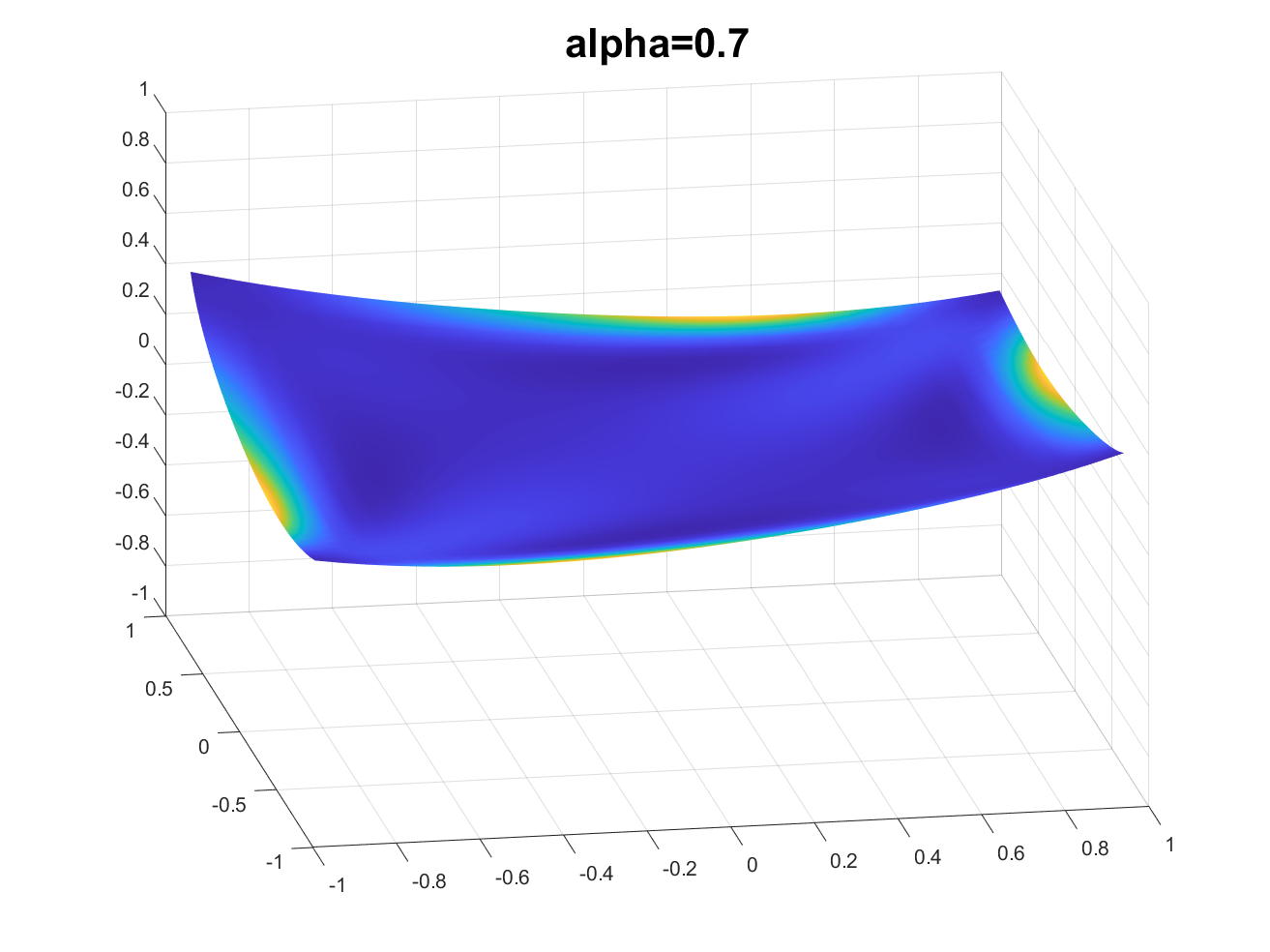}
\end{minipage}
\begin{minipage}{0.49\textwidth}
\includegraphics[width=0.95\textwidth]{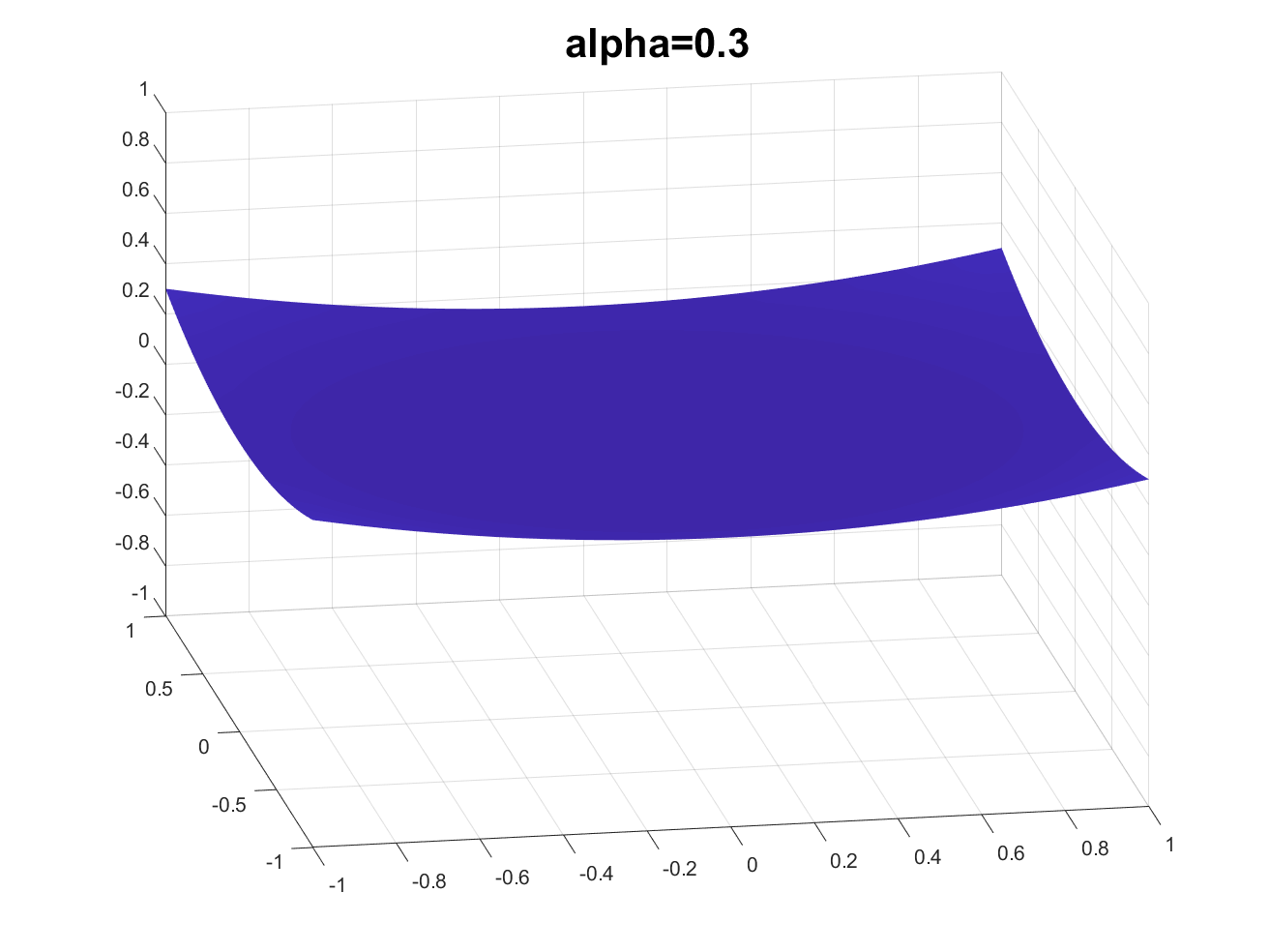}
\end{minipage}
\begin{minipage}{0.49\textwidth}
\includegraphics[width=0.95\textwidth]{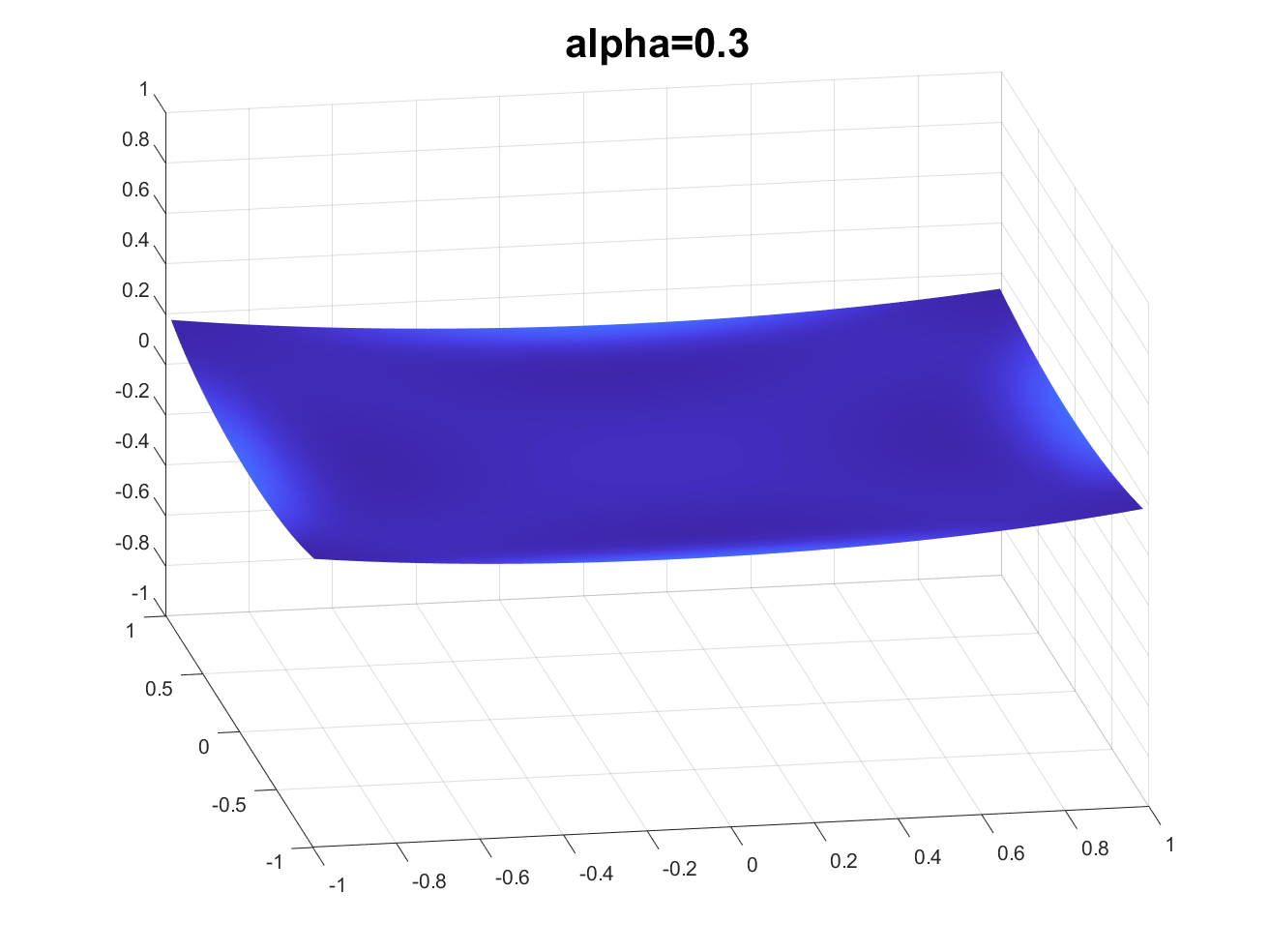}
\end{minipage}
\begin{minipage}{0.49\textwidth}
\includegraphics[width=0.95\textwidth]{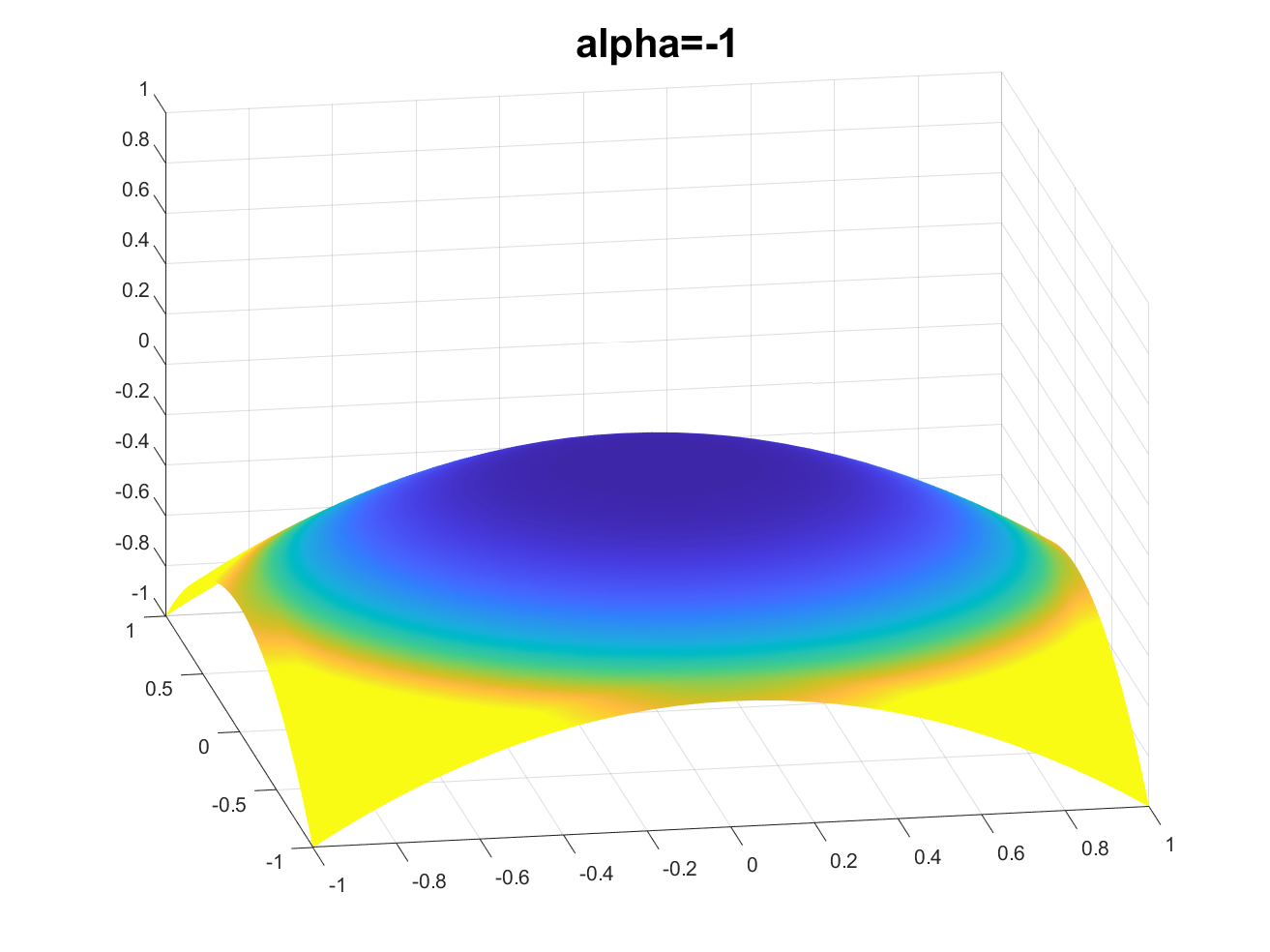}
\end{minipage}
\begin{minipage}{0.49\textwidth}
\includegraphics[width=0.95\textwidth]{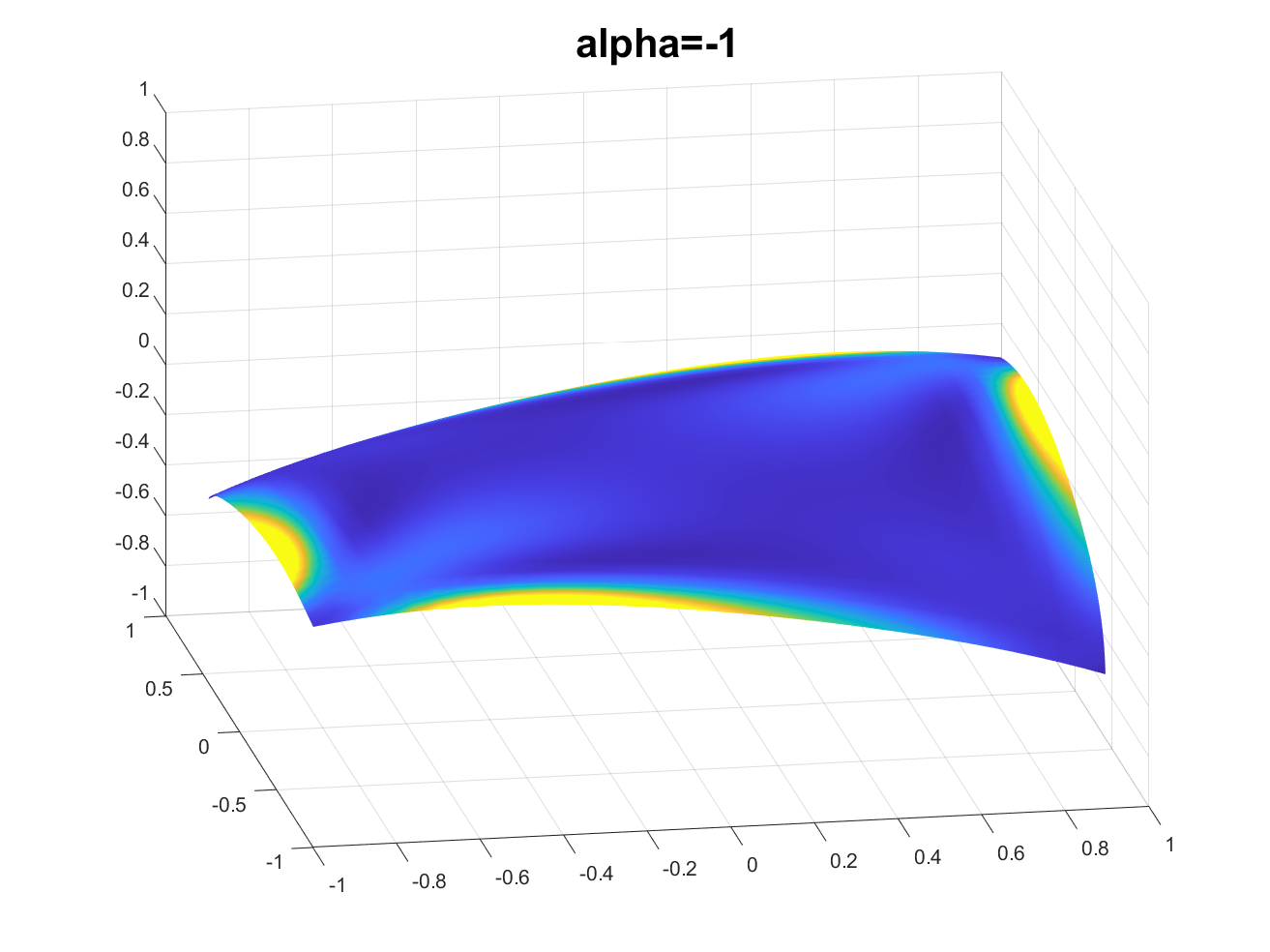}
\end{minipage}
\caption{Numerical solutions of \thref{discr_decoupled_grad_flow} for $\alpha\in\{1,0.7,0.3,-1\}$ (top to bottom) with $\theta=0$ (left) and $\theta=1000$ (right). The colors represent the strain density from dark (lowest) to bright (highest). {\cc The transition from $\alpha=1$ to $\alpha=-1$ induces a curvature inversion, where the value $\alpha=0$ leads to flat configurations for arbitrary $\theta$.}}
\label{fig:curvature_inversion}
\end{figure}

\subsection{Foldable Cardboard}\label{subsec:FoldableCardboard}

{\cc Next, we aim at analyzing the behavior of the bottom part of a deployed cardboard (compare with the right picture of Figure \ref{fig:photos_cardboard}), when the bistable mechanism is initiated by application of a force in a small radial area at the center of the plate.} We consider a singlelayer plate, i.e. $\alpha=0$, whose vertical component $w$ is clamped in a cylindrical shape via simple support boundary conditions on two opposing sides. No boundary conditions are enforced on the in-plane displacement $u$. The initial deformation is described by the data
\[
u_D(x)=\begin{bmatrix}
0\\0
\end{bmatrix},\quad w_D(x)=-\frac{1}{2}\left(x_2^2-1\right),\quad x\in [-1,1]^2.
\]
We set $\theta=10^{6}$ and induce a gradually increasing vertical force with a maximal value $f=-0.6\cdot10^6$ in the center $B_{1/10}(0)$ of the object. \thref{discr_decoupled_grad_flow} is applied to a regular cardboard without containing a crease line. The same experiment is carried out to a cardboard that is predamaged on a straight line $\If=\{0\}\times[-1,1]$ across its center. In the simulation this is achieved by assigning to nodes $z\in\Nh\cap\Ifh$ on the crease line two seperate gradient values $\nabla w_1(z)$ and $\nabla w_2(z)$ corresponding to each side $\Omega_1,\Omega_2$ of the crease. This is equivalent to solving the discrete decoupled gradient flow separately on each subdomain subject to the continuity constraint $w_1(z)=w_2(z)$ for all $z\in\Nh\cap\Ifh$. 
%A partitioning of the mesh is visualized in Figure \ref{fig:foldmesh}. 
If $DF_i$ and $F_i$ denote the respective system matrices for the Newton scheme (\thref{newton_scheme}) restricted to $\Omega_i$, the continuity condition is enforced via Lagrange multipliers by solving the modified linear system of equations
\[
\begin{bmatrix}
DF_1 & 0 & I_1\T \\
0 & DF_2 & -I_2\T \\
I_1 & -I_2 & 0
\end{bmatrix}
\begin{bmatrix}
\widetilde{w}_1\\
\widetilde{w}_2\\
\Lambda
\end{bmatrix}
=
\begin{bmatrix}
F_1\\
F_2\\
0
\end{bmatrix}.
\]
Here, $\widetilde{w}_i=w^{k}_i-w^{k-1}_i$ denotes the $k$-th update on each subdomain $\Omega_i$ containing the values $\widetilde{w}_i(z)$ and $\nabla \widetilde{w}_i(z)$ for all $z\in\Nh\cap\,\overline{\Omega}_i$, while the vector $\Lambda$ includes the Lagrange multipliers. The matrices $I_1,I_2$ contain the values $0$ or $1$, encoding the condition $\widetilde{w}_1(z)=\widetilde{w}_2(z)$ for all $z\in\Nh\cap\,\Ifh$. If the previous iterates $w^{k-1}_1$ and $w^{k-1}_2$ are continuous on $z\in\Nh\cap\Ifh$ the condition guarantees the pointwise continuity $w^k_1(z)=w^k_2(z)$ of the next iterates for $z\in\Nh\cap\Ifh$.
%\begin{figure}[H]
%\centering
%\def\svgwidth{360pt}
%\import{material/}{foldmeshnew.pdf_tex}
%\caption{Partitioning of the mesh followed by uniform refinements (left to right). The dashed line represents the straight crease while boundary conditions are applied parallel to it.}
%\label{fig:foldmesh}
%\end{figure}
Numerical solutions for the cardboard with and without crease line are visualized in Figure \ref{fig:cardboard_radial}. The numerical simulations show that both plates are pushed downwards in the center. Initially, the cardboards are indented in a radially symmetric area. The indented area of the undamaged cardboard transforms to an oval shape and the natural formation of a crease line across the plate center can be observed. In the deformation profile of the predamaged cardboard a rhombic shape is visible instead. The indented areas are surrounded by lines of large curvature that emenate from the endpoints of the middle crease to the boundary of the plate. This is especially apparent from Figure \ref{fig:cardboard_bending_radial} which contains top views of the numerical solutions shaded with respect to the elementwise bending energy densities. 
Figure \ref{fig:energy_cap_fold} shows the developments of the total elastic energies of the respective numerical solutions. 

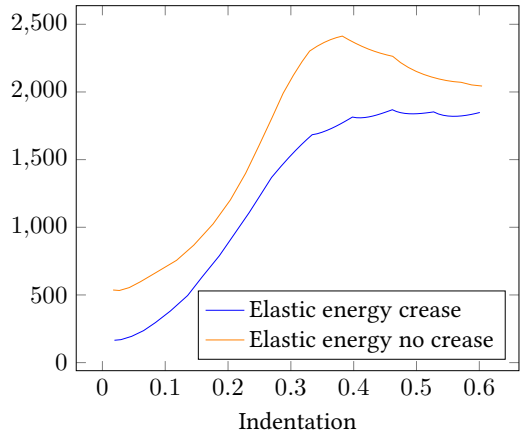
\begin{figure}[H]
\centering
\begin{tikzpicture}[scale=0.85]
\begin{axis}[
		%every y tick scale label/.style={at={(yticklabel* cs:0.96,0cm)},anchor=near yticklabel},
    xlabel={Indentation},
    legend pos=south east,
    ymajorgrids=false,
    xmajorgrids=false,
    grid style=dashed,
    legend style={nodes={scale=1, transform shape}},
    %ymax = 1.7*1e-3,
    legend cell align={left},
    %xmax = 4500
]
\addplot[
    color=blue
    ]
    table [x expr=-\thisrow{indentation}+0.5,y expr=\thisrow{ener_pcw}*1e6] {material/cardboard_fold/energy_fold_incforce.txt};
    \legend{Elastic energy crease}
\addplot[
    color=orange
    ]
    table [x expr=-\thisrow{indentation}+0.5,y expr=\thisrow{ener}*1e6] {material/cardboard_fold/energy_nofold_incforce.txt};
    \addlegendentry{Elastic energy no crease}   
\end{axis}
\end{tikzpicture}

\caption{Energy development of the sequence $(w_h^k,u_h^k)_{k\geq0}$ with and without crease line. In both cases we observe an energy barrier but the total elastic energy of the model with crease is persistently lower.}
\label{fig:energy_cap_fold}
\end{figure}

In both cases we observe an energy barrier between the unindented, cylindrical initial configuration and the indented configurations that contain large flat areas. The energy barrier of the plate without predamage is much larger due to an increased elastic resistance compared to the plate with crease. 
Interestingly, the distance of the wrinkles perpendicular to the formed line match the length of a crack that appeared in a real cardboard after repeated actuation, cf. Figure \ref{fig:photos_cardboard}. This suggests a connection between the two phenomena.

\begin{figure}[p]
\centering
\begin{minipage}{0.49\textwidth}
\includegraphics[width=0.97\textwidth]{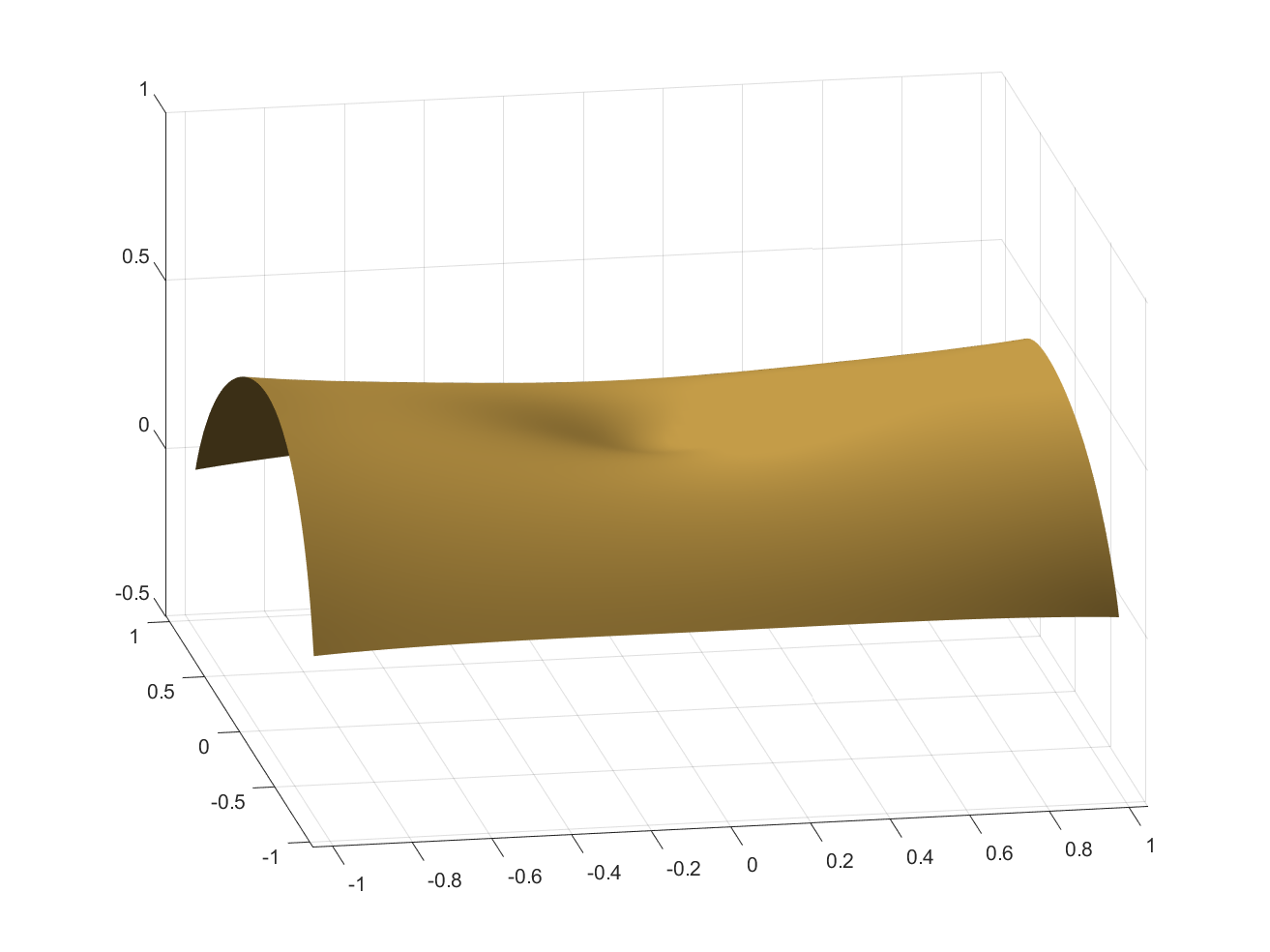}
\end{minipage}
\begin{minipage}{0.49\textwidth}
\includegraphics[width=0.97\textwidth]{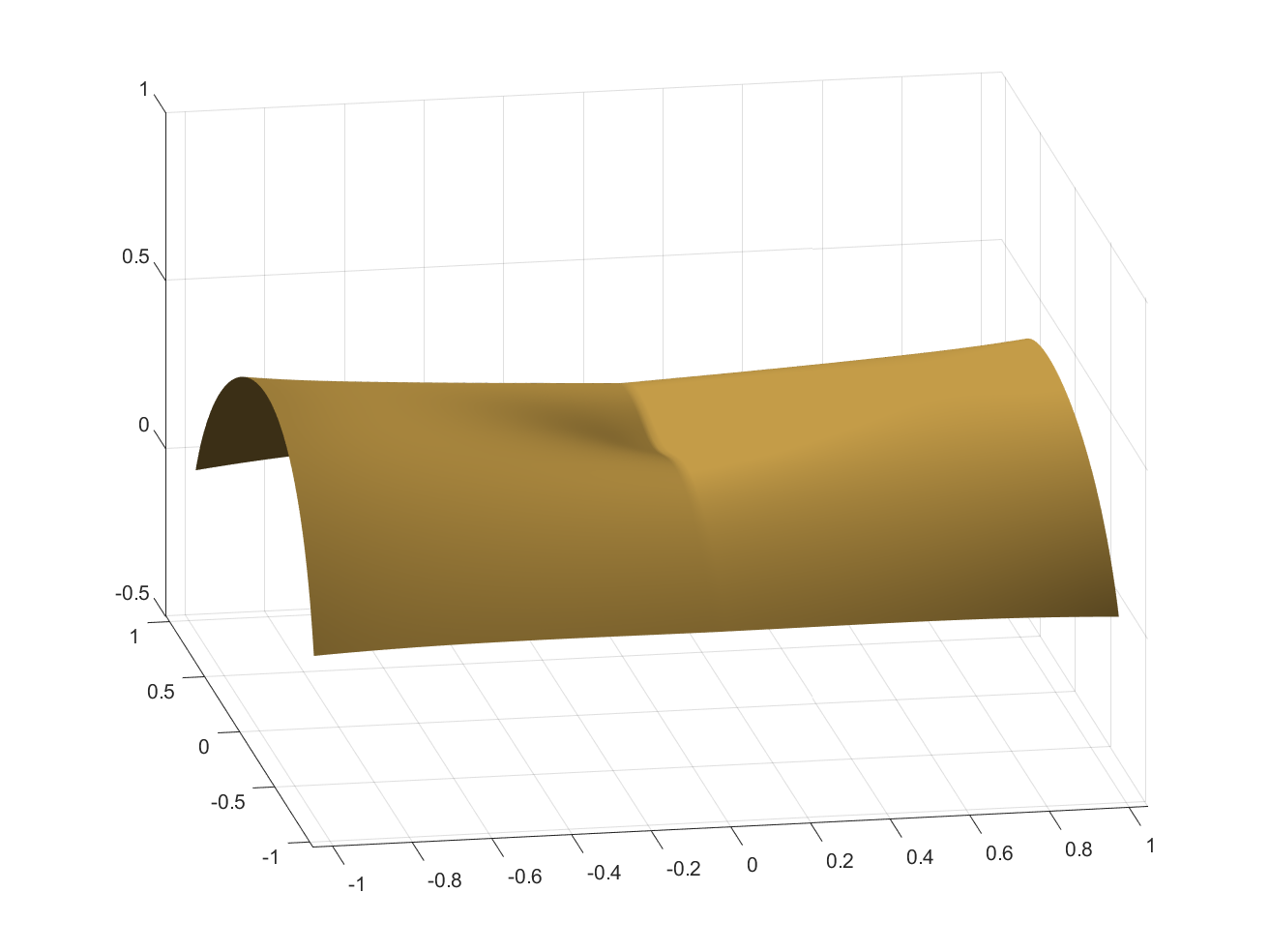}
\end{minipage}
\begin{minipage}{0.49\textwidth}
\includegraphics[width=0.97\textwidth]{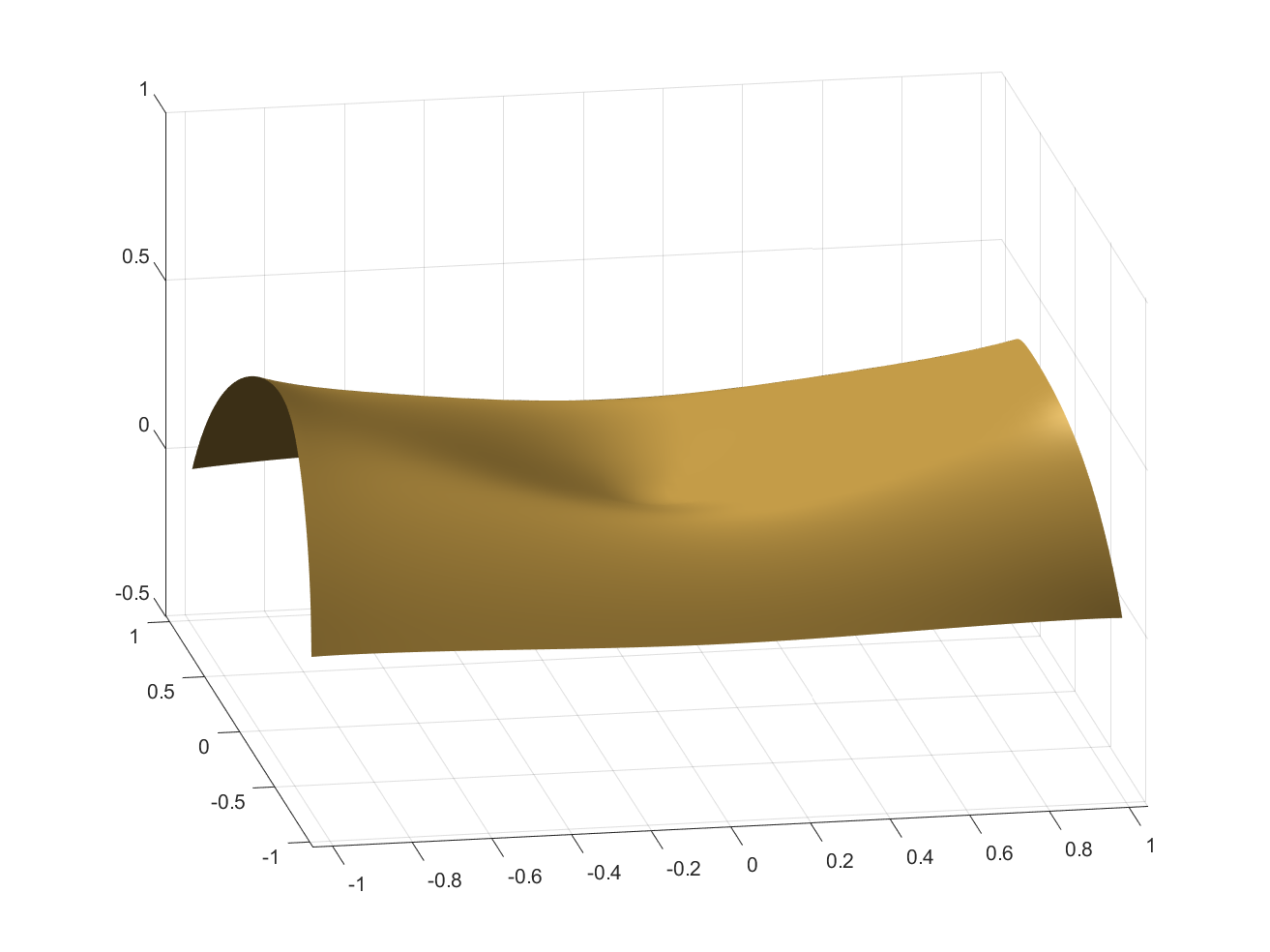}
\end{minipage}
\begin{minipage}{0.49\textwidth}
\includegraphics[width=0.97\textwidth]{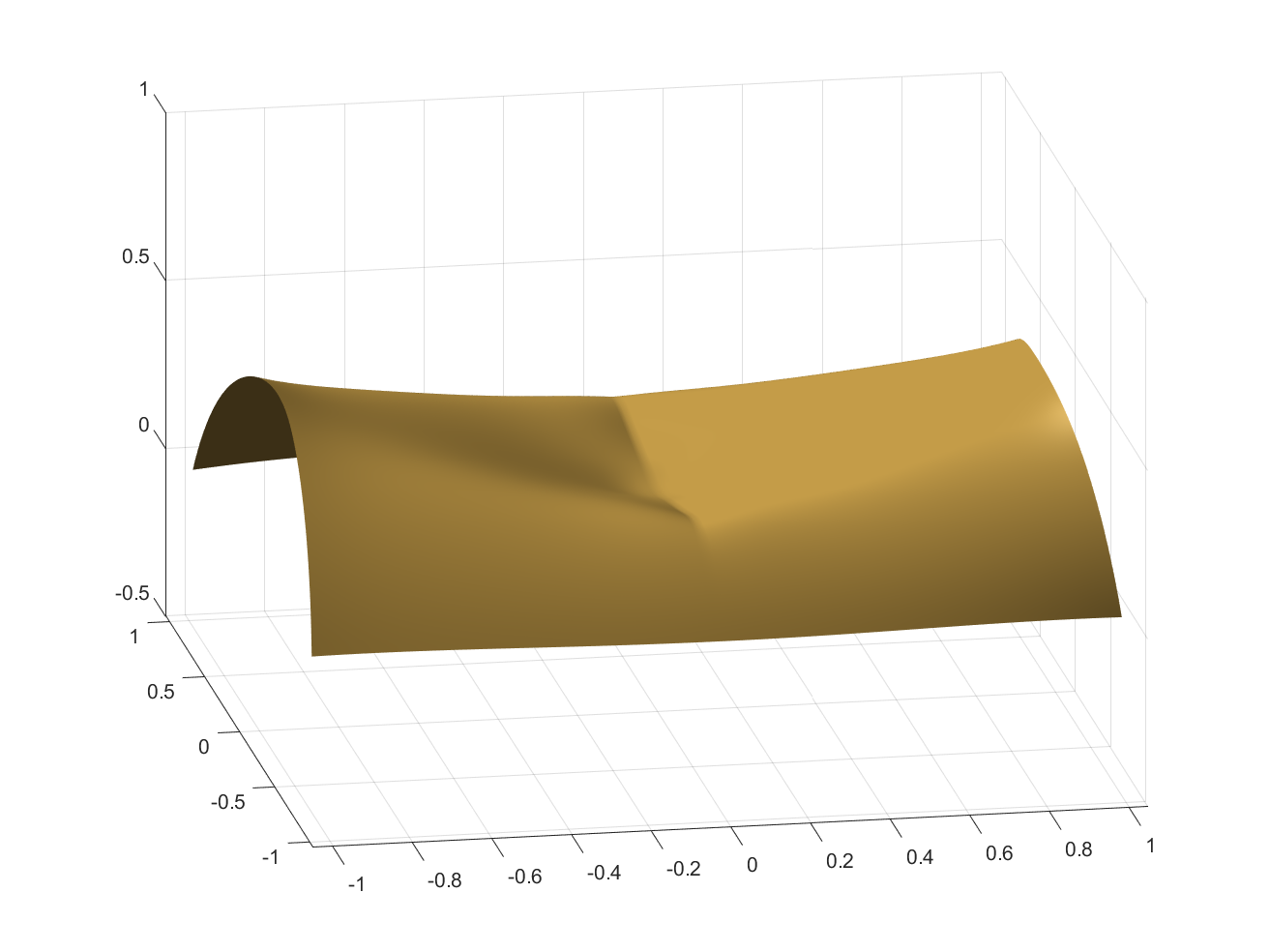}
\end{minipage}
\begin{minipage}{0.49\textwidth}
\includegraphics[width=0.97\textwidth]{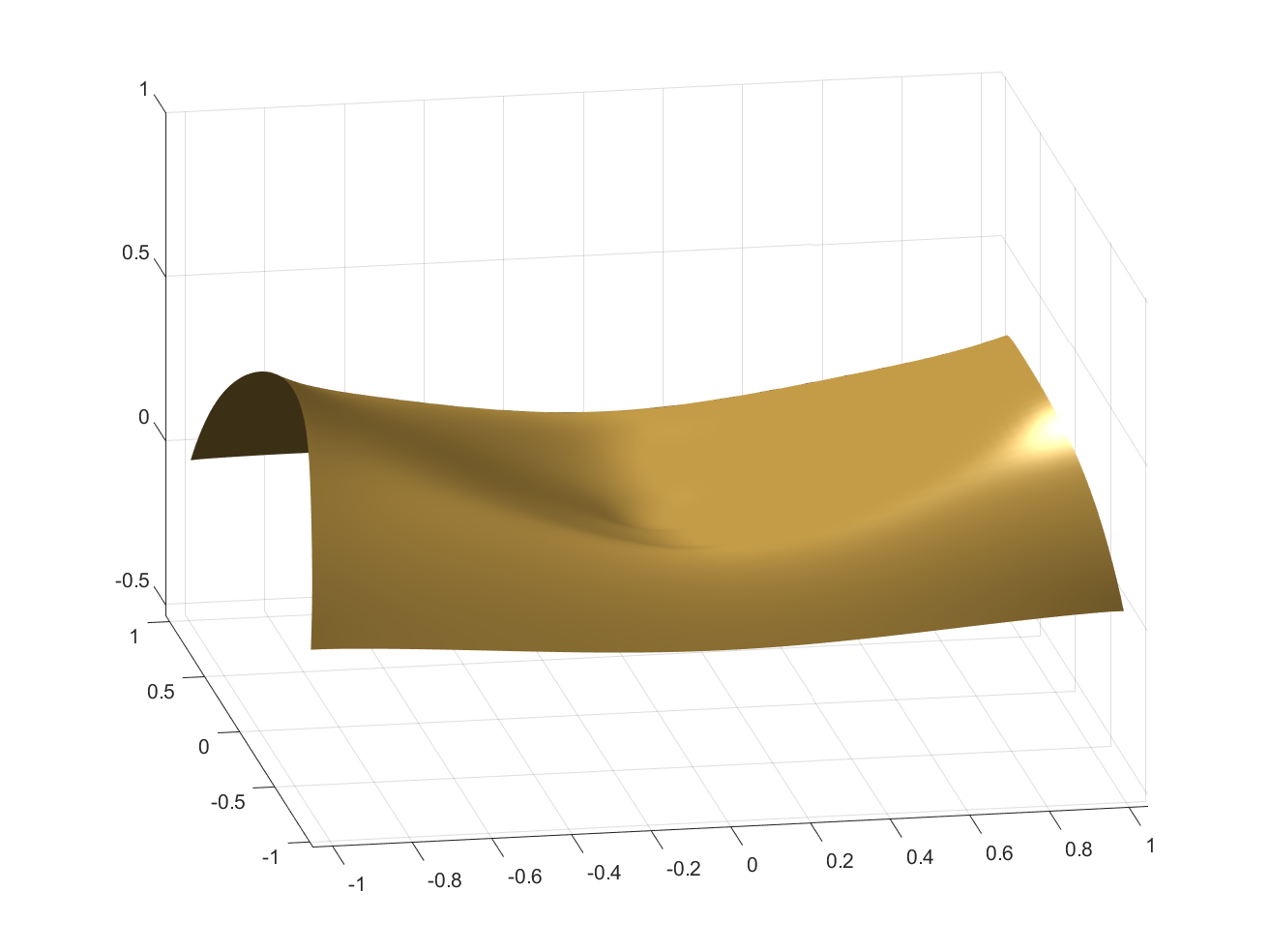}
\end{minipage}
\begin{minipage}{0.49\textwidth}
\includegraphics[width=0.97\textwidth]{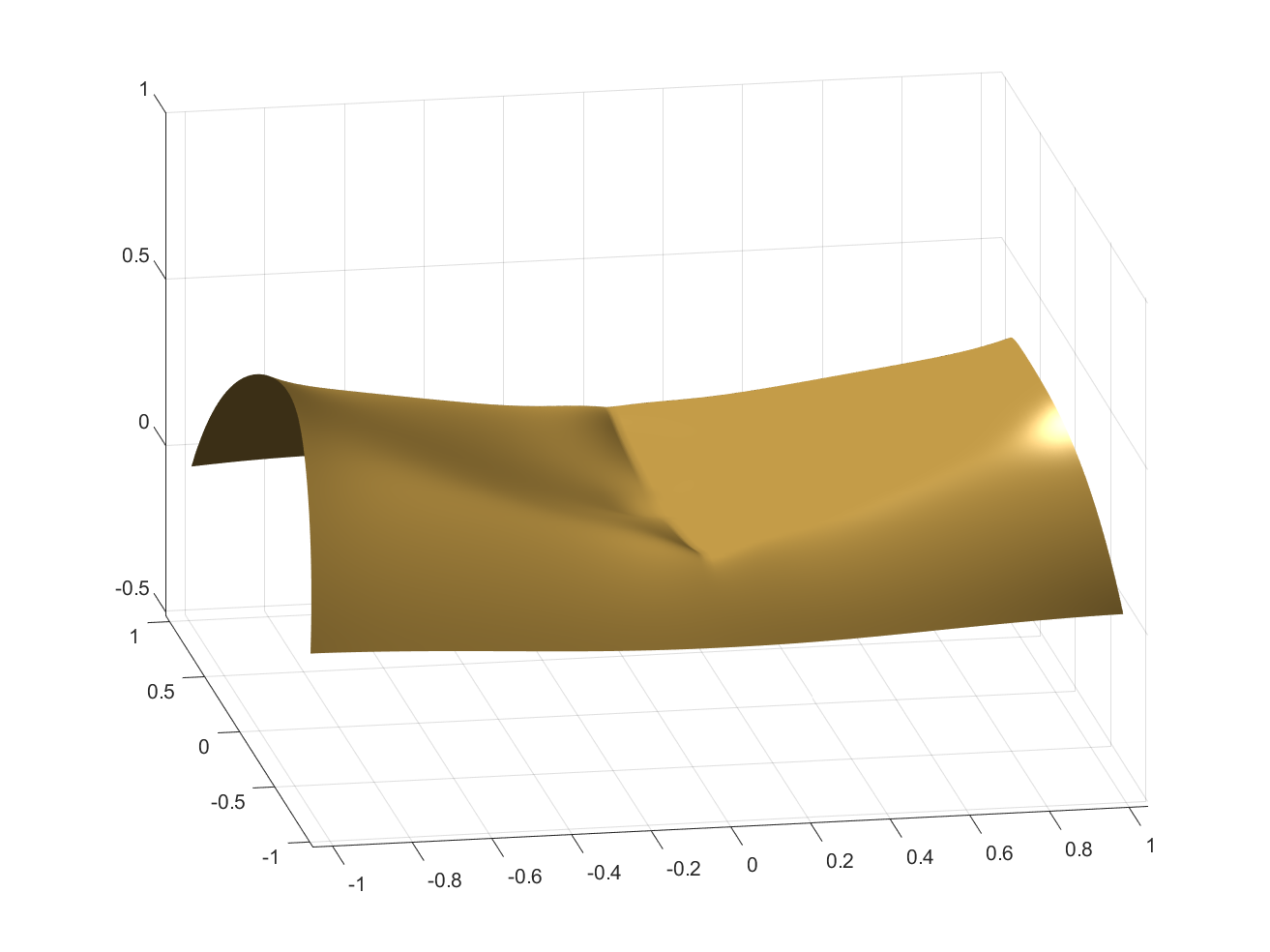}
\end{minipage}
\begin{minipage}{0.49\textwidth}
\includegraphics[width=0.97\textwidth]{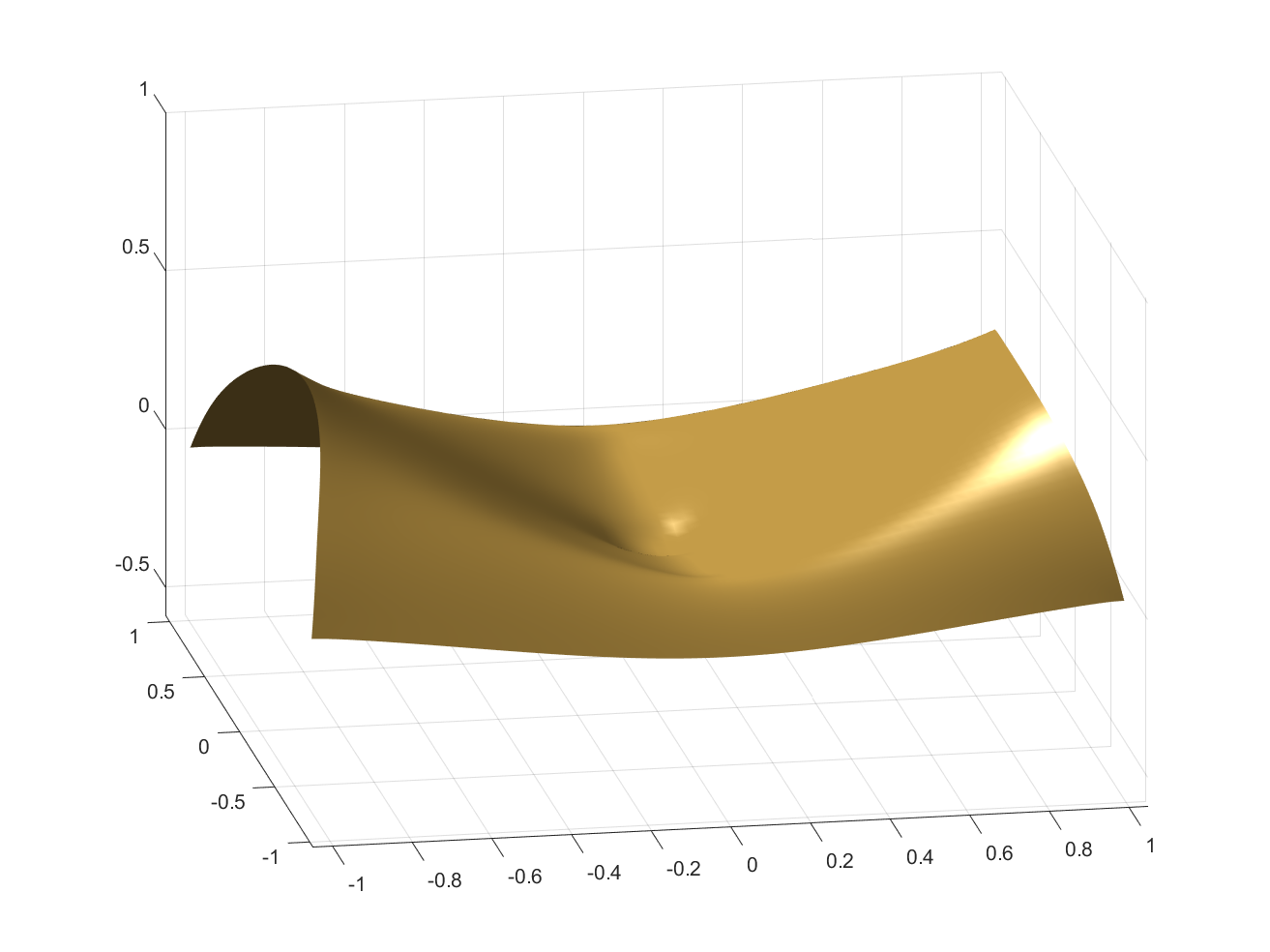}
\end{minipage}
\begin{minipage}{0.49\textwidth}
\includegraphics[width=0.97\textwidth]{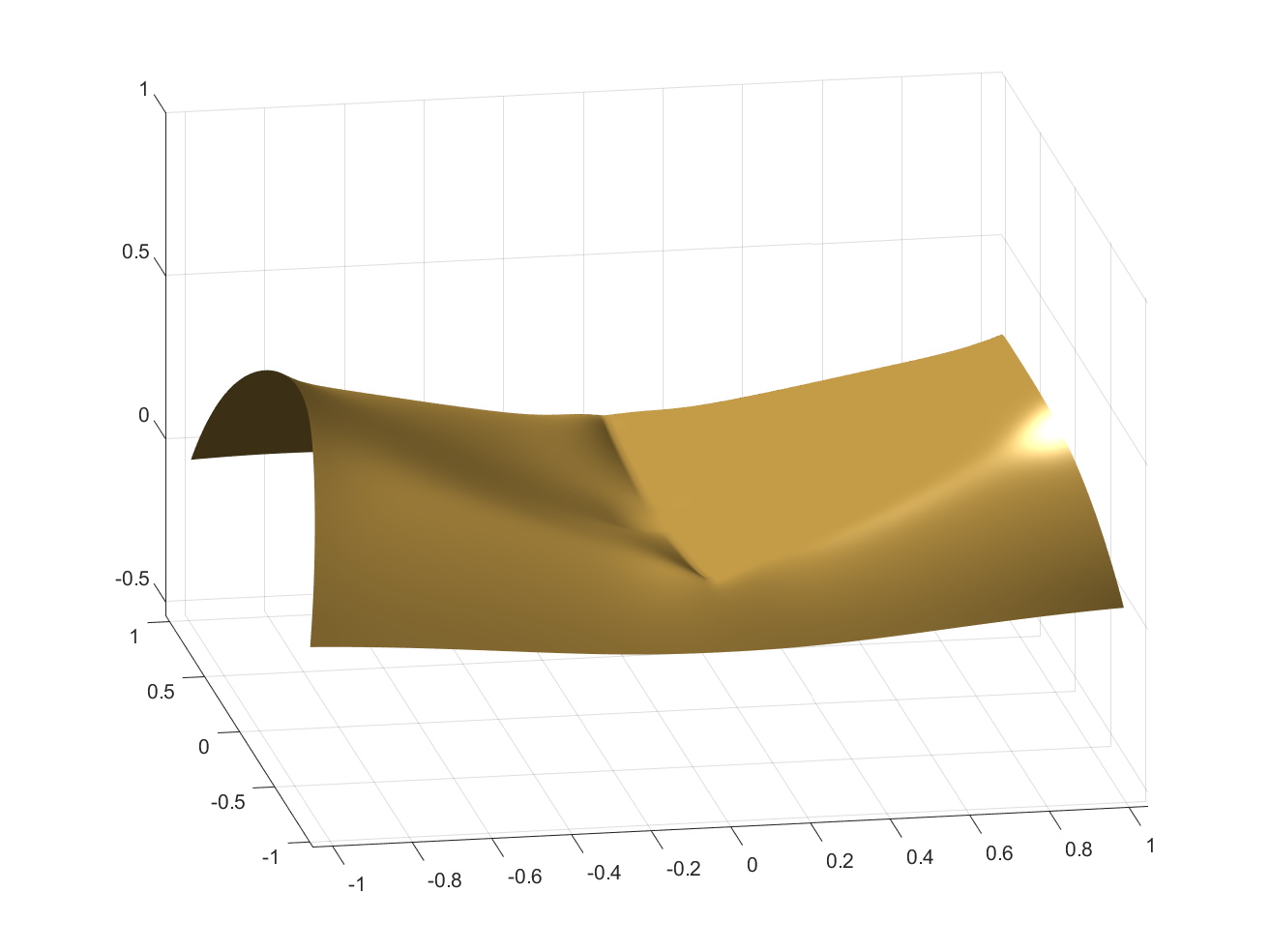}
\end{minipage}
\caption{Numerical solutions for a cardboard without crease line (left) and with crease line (right) for a radial force acting in the plate center after 20, 30, 40 and 50 iterations of the algorithm. {\cc The deformation of the undamaged cardboard leads to an oval indentation profile whereas the predamaged cardboard obtains a rhombic indentation profile.}}
\label{fig:cardboard_radial}
\end{figure}

\begin{figure}[p]
\centering
\begin{minipage}{0.49\textwidth}
\includegraphics[width=0.9\textwidth]{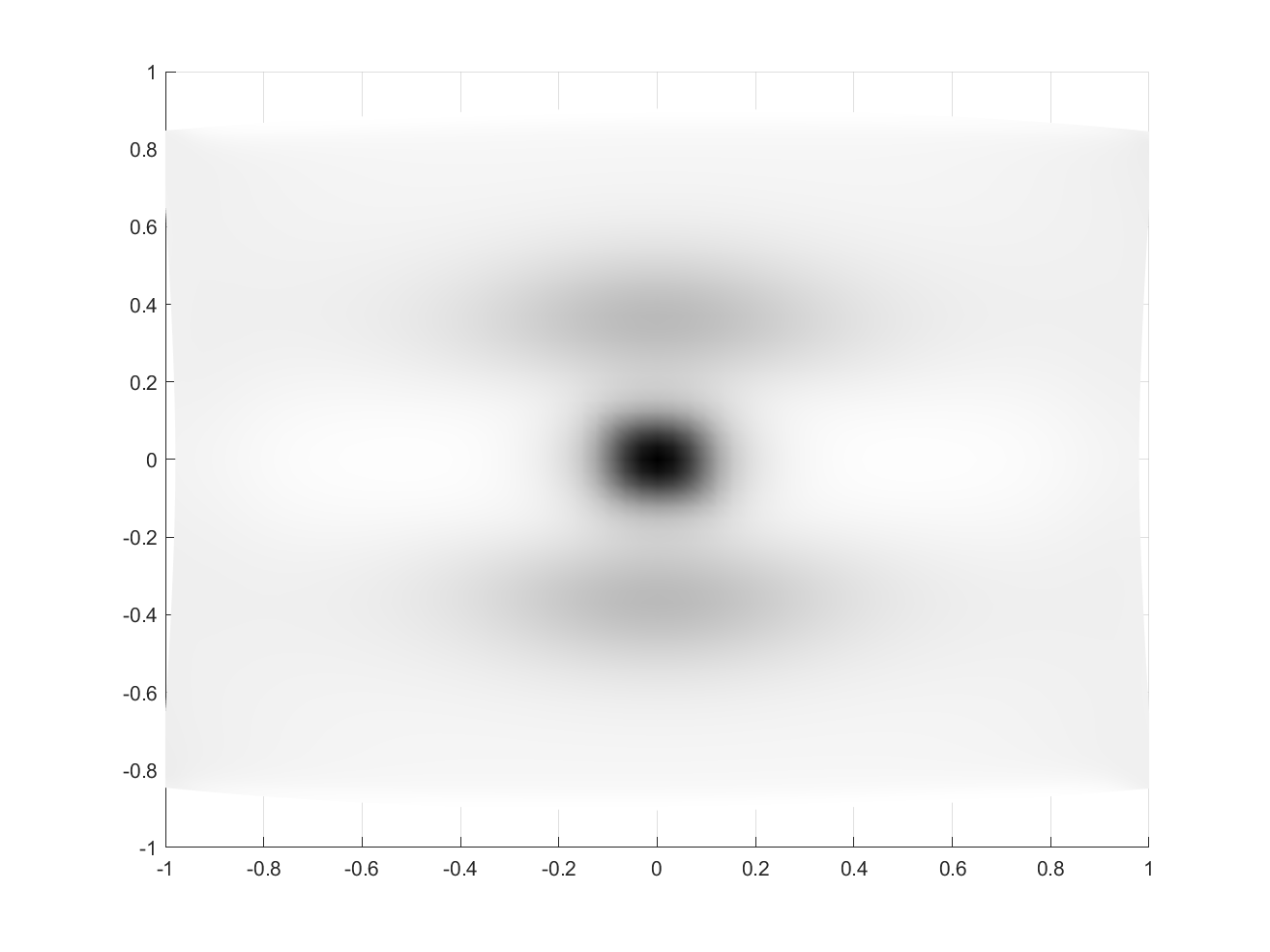}
\end{minipage}
\begin{minipage}{0.49\textwidth}
\includegraphics[width=0.9\textwidth]{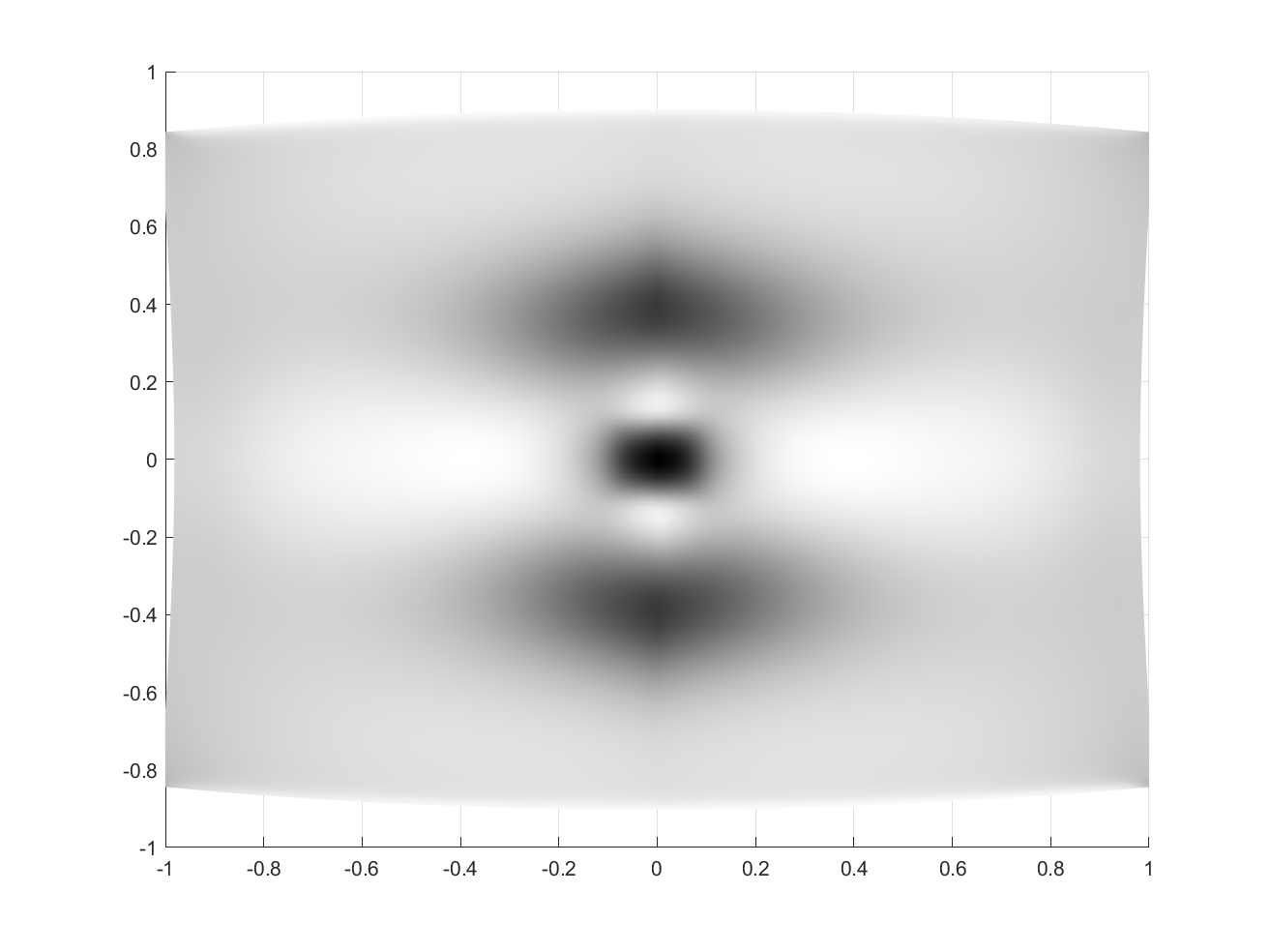}
\end{minipage}
\begin{minipage}{0.49\textwidth}
\includegraphics[width=0.9\textwidth]{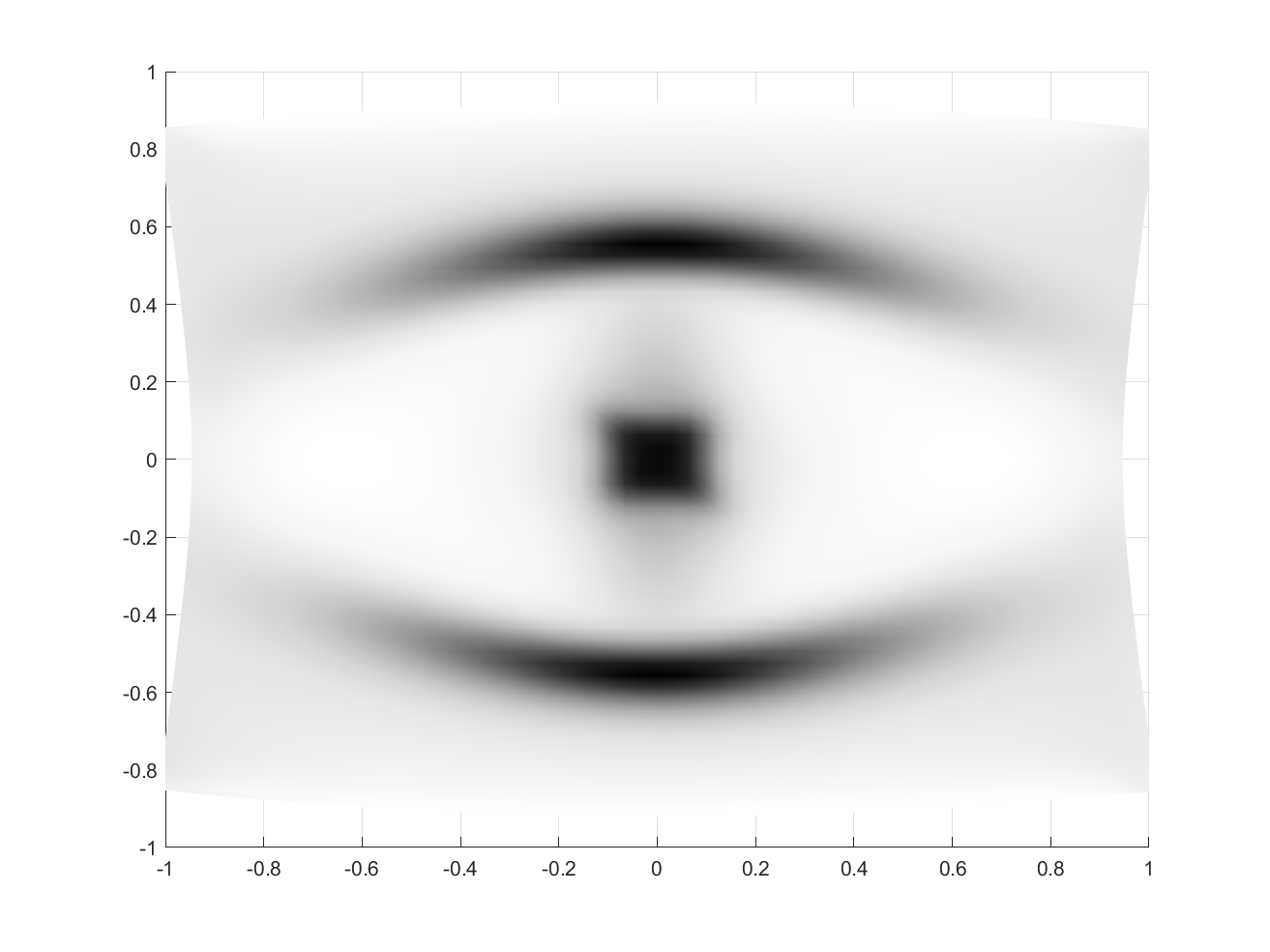}
\end{minipage}
\begin{minipage}{0.49\textwidth}
\includegraphics[width=0.9\textwidth]{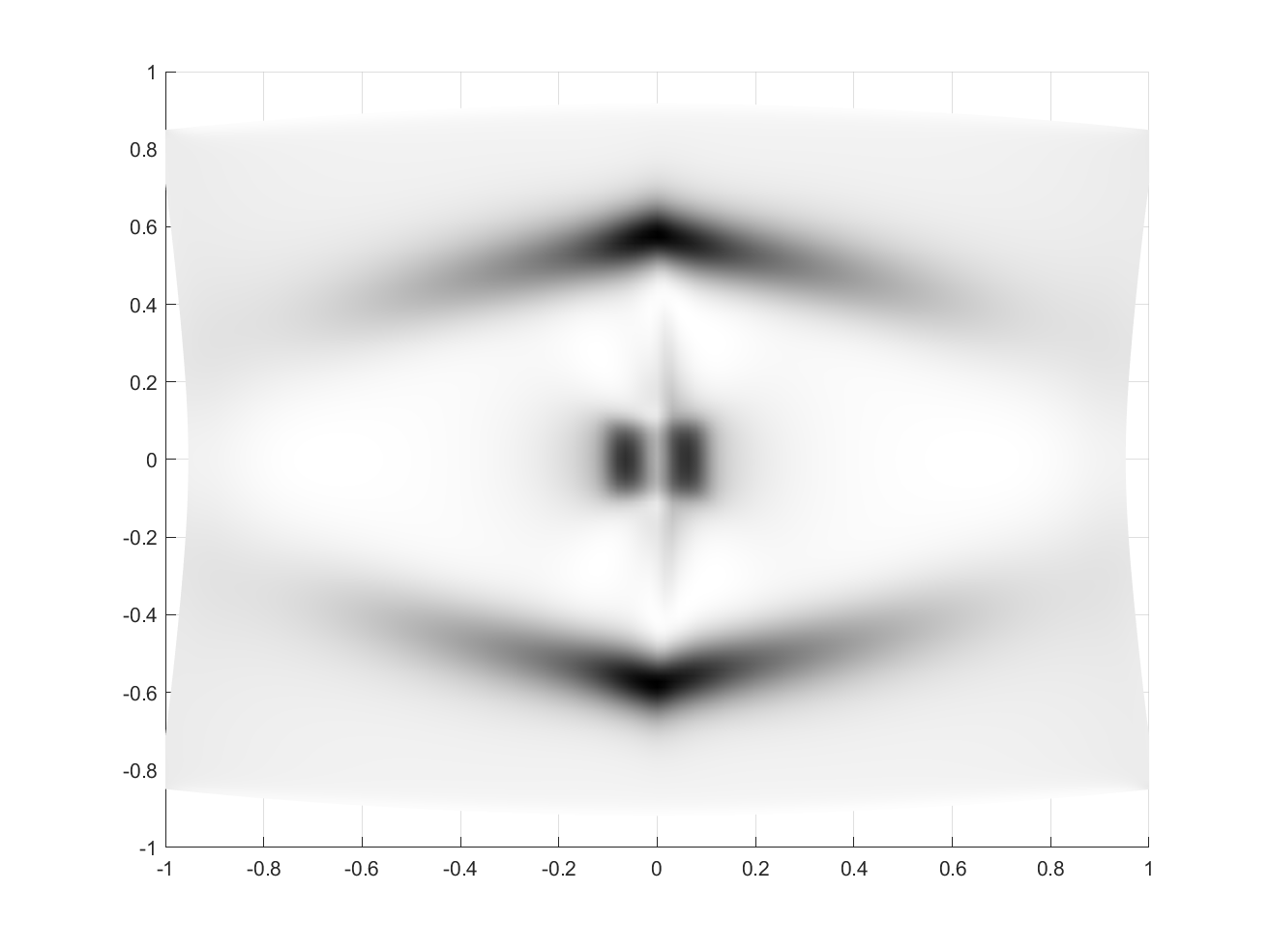}
\end{minipage}
\begin{minipage}{0.49\textwidth}
\includegraphics[width=0.9\textwidth]{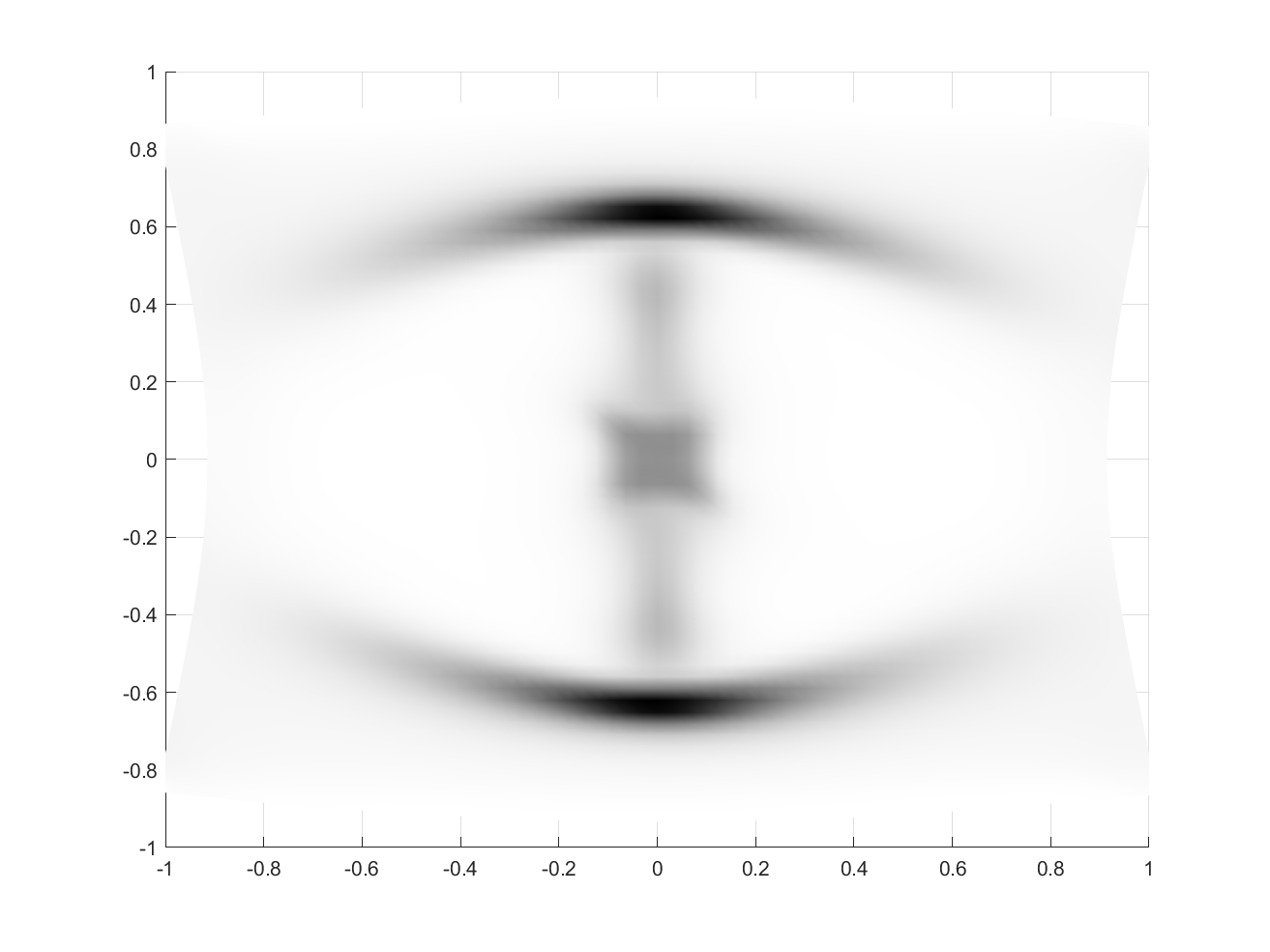}
\end{minipage}
\begin{minipage}{0.49\textwidth}
\includegraphics[width=0.9\textwidth]{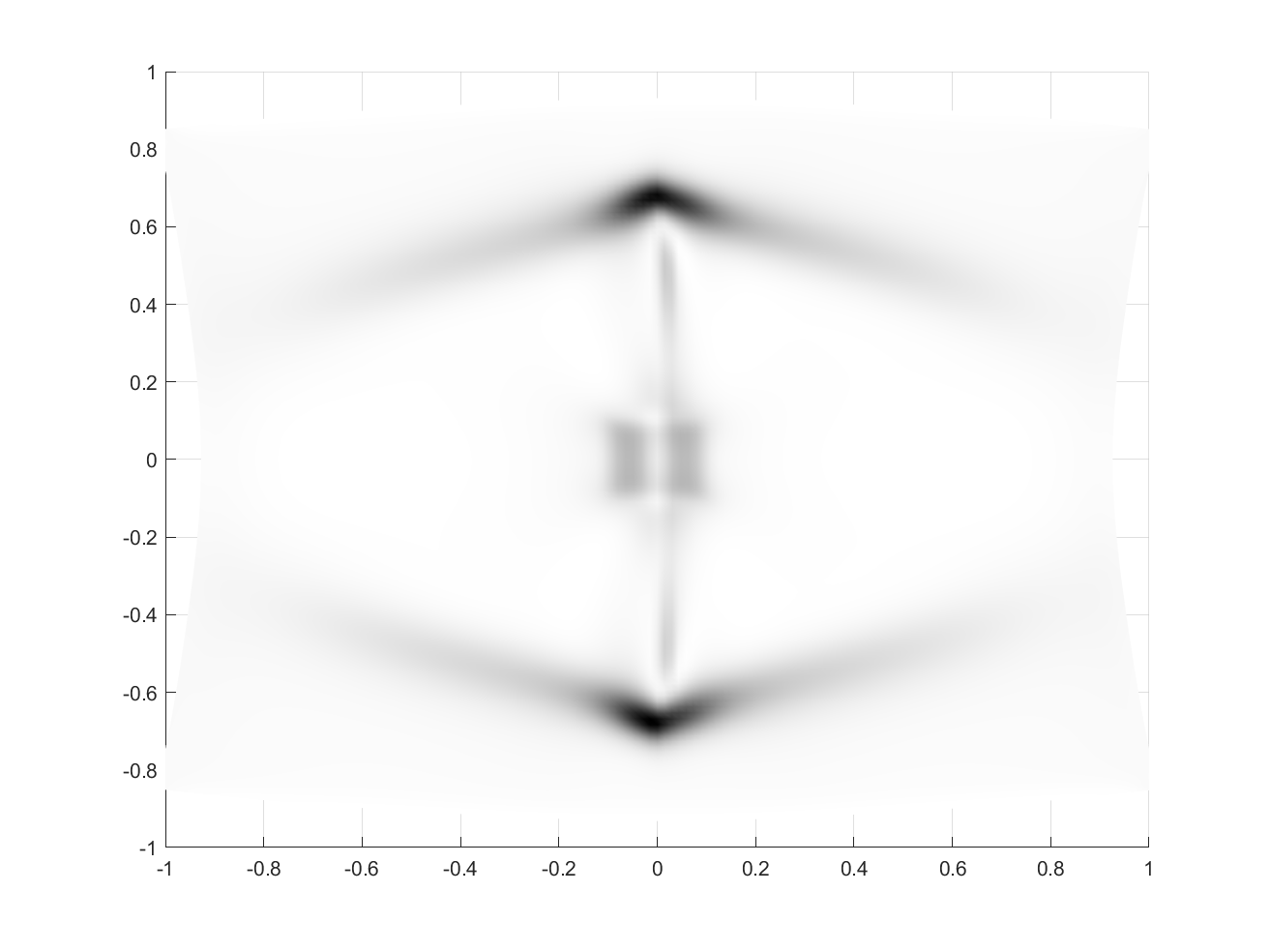}
\end{minipage}
\begin{minipage}{0.49\textwidth}
\includegraphics[width=0.9\textwidth]{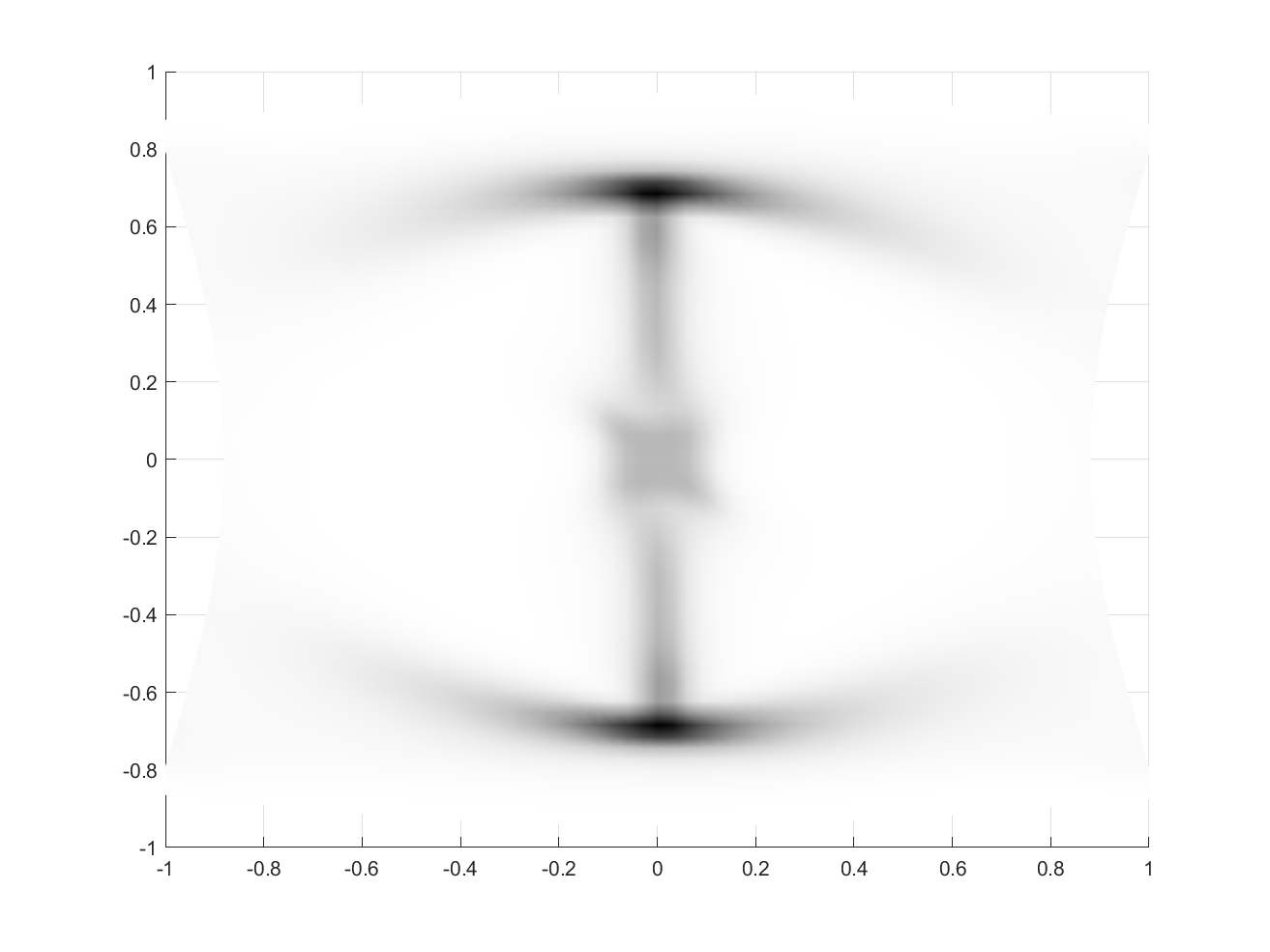}
\end{minipage}
\begin{minipage}{0.49\textwidth}
\includegraphics[width=0.9\textwidth]{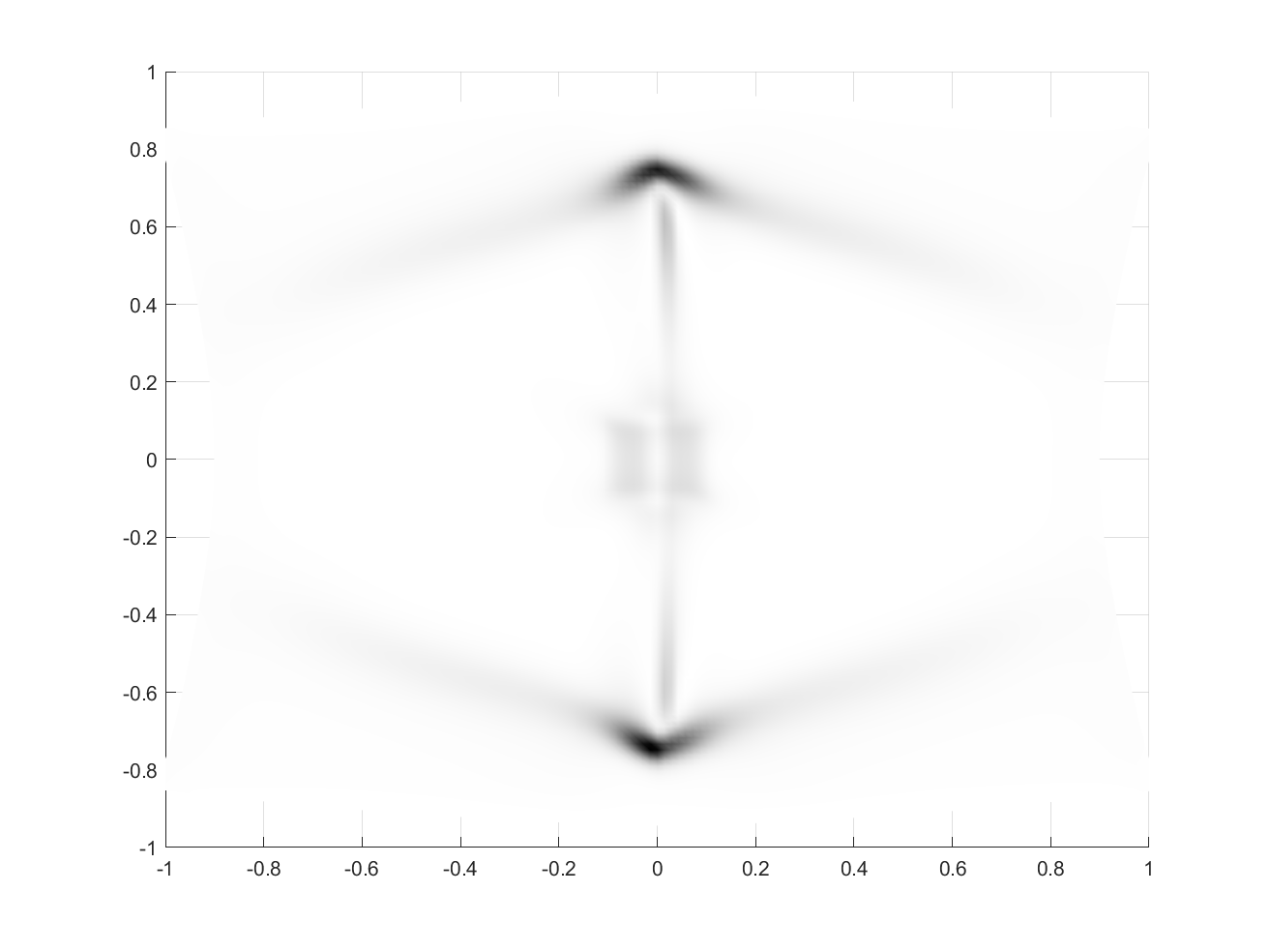}
\end{minipage}
\caption{Top views of the numerical solutions for a cardboard without crease line (left) and with crease line (right) for a radial force acting in the plate center after 20, 30, 40 and 50 iterations of the algorithm with a shading of the bending energy densities $|D^2_hw_h|$. {\cc In the undamaged cardboard a natural formation of a folding line across the plate center can be observed. In both the undamaged and predamaged models, lines of large curvature appear across the (formed) folding line close to the boundary of the plates. The distance between these creases and the boundary seems to match the crack length observed in the left picture of Figure \ref{fig:photos_cardboard}.}}
\label{fig:cardboard_bending_radial}
\end{figure}

\subsection{Bilayer Plate Folding}
Finally, we combine elastic effects of bilayer plates with physical phenomena related to folding. The motivation stems from the fact that simulations of mechanisms like Venus flytraps via bilayer models are presumably more realistic than simulations obtained from singlelayer models which depend on the artificial application of boundary conditions or forcing terms. For instance, bilayer models are capable of describing swelling effects that enable the actuation of carnivorous plants. We give a brief overview over different configurations for various parameters $\alpha_1,\alpha_2\in\mathbb{R}$ corresponding to the subdomains $\Omega_1,\Omega_2$ when combined with different geometries of given crease curves $\If$ that separate the two regions. The discrete crease is either assumed to be a straight line $\Ifh=\{0\}\times[-1,1]$ 
%across the plate $\Omega=(-1,1)^2$ 
or a piecewise linear approximation of the arc 
\[
\If(t)=\left(\frac{1}{6}\sin(\pi t)+\frac{1}{3},t\right),\quad t\in[-1,1].
\]
%We restrict ourselves to the case $\theta=1$ and initialize the algorithm with a flat configuration.
Figure \ref{fig:bilayer_folding} contains numerical solutions for straight and piecewise linear crease geometries paired with various combinations of parameters $\alpha_1,\alpha_2\in\{-1,0,1\}$ and $\theta=1$. 
We observe drastic differences in the deformation profiles between the two crease geometries. Interestingly, the straight crease with $\alpha_1,\alpha_2=1$ leads to a kink while the curved crease leads to a nearly smooth deformation without significant jumps in the deflection gradient. In the last two tests a spontaneous curvature is only applied on the left part of the plate. The curved crease geometry leads to a snapping mechanism while the deflection on the right side remains nearly flat when a straight crease is included. 

% \newpage

\begin{figure}[p]
\begin{minipage}{0.49\textwidth}
\centering
\includegraphics[width=.90\textwidth]{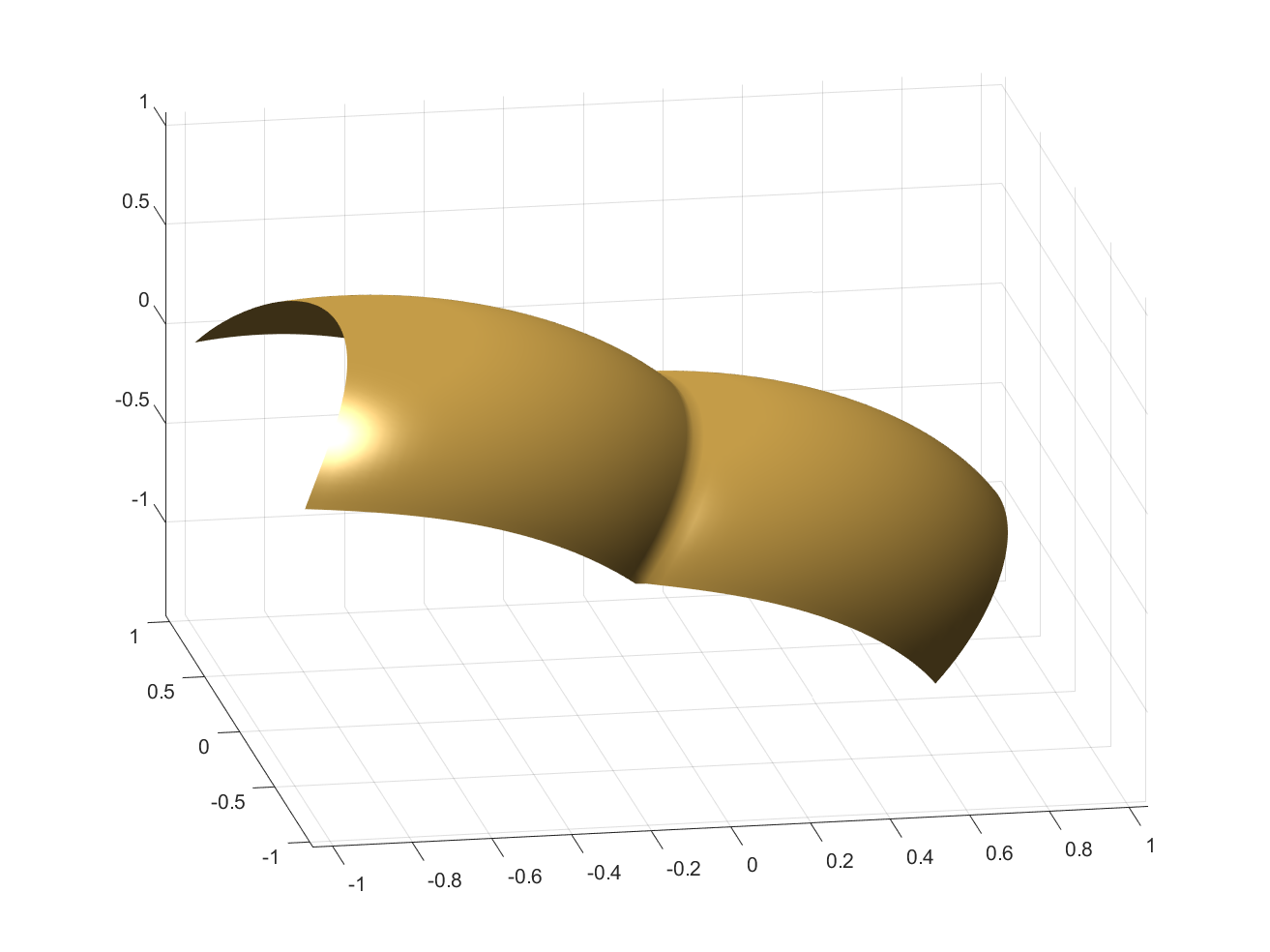}
\end{minipage}
\begin{minipage}{0.49\textwidth}
\centering
\includegraphics[width=.90\textwidth]{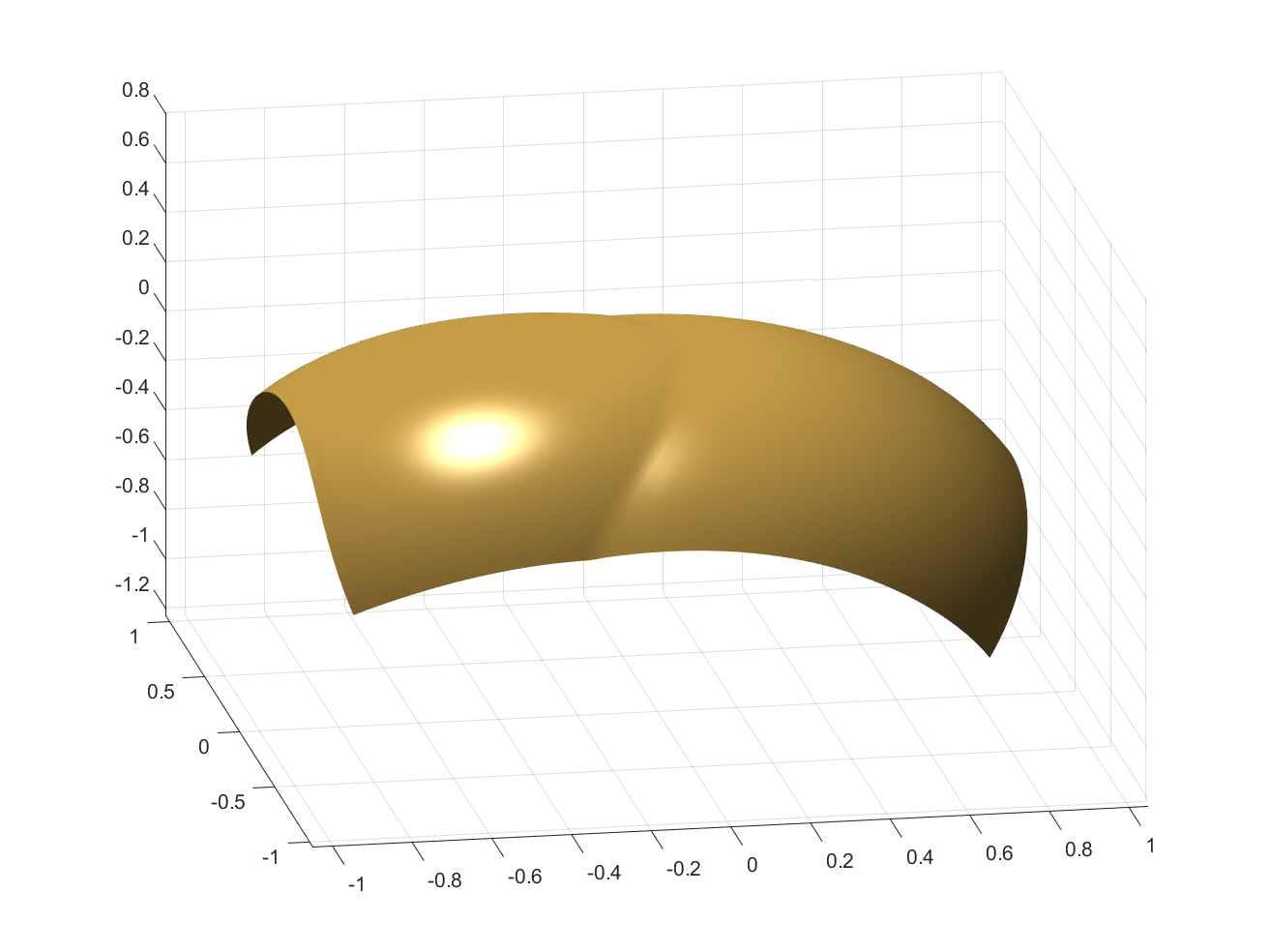}
\end{minipage}
\begin{minipage}{0.49\textwidth}
\centering
\includegraphics[width=.90\textwidth]{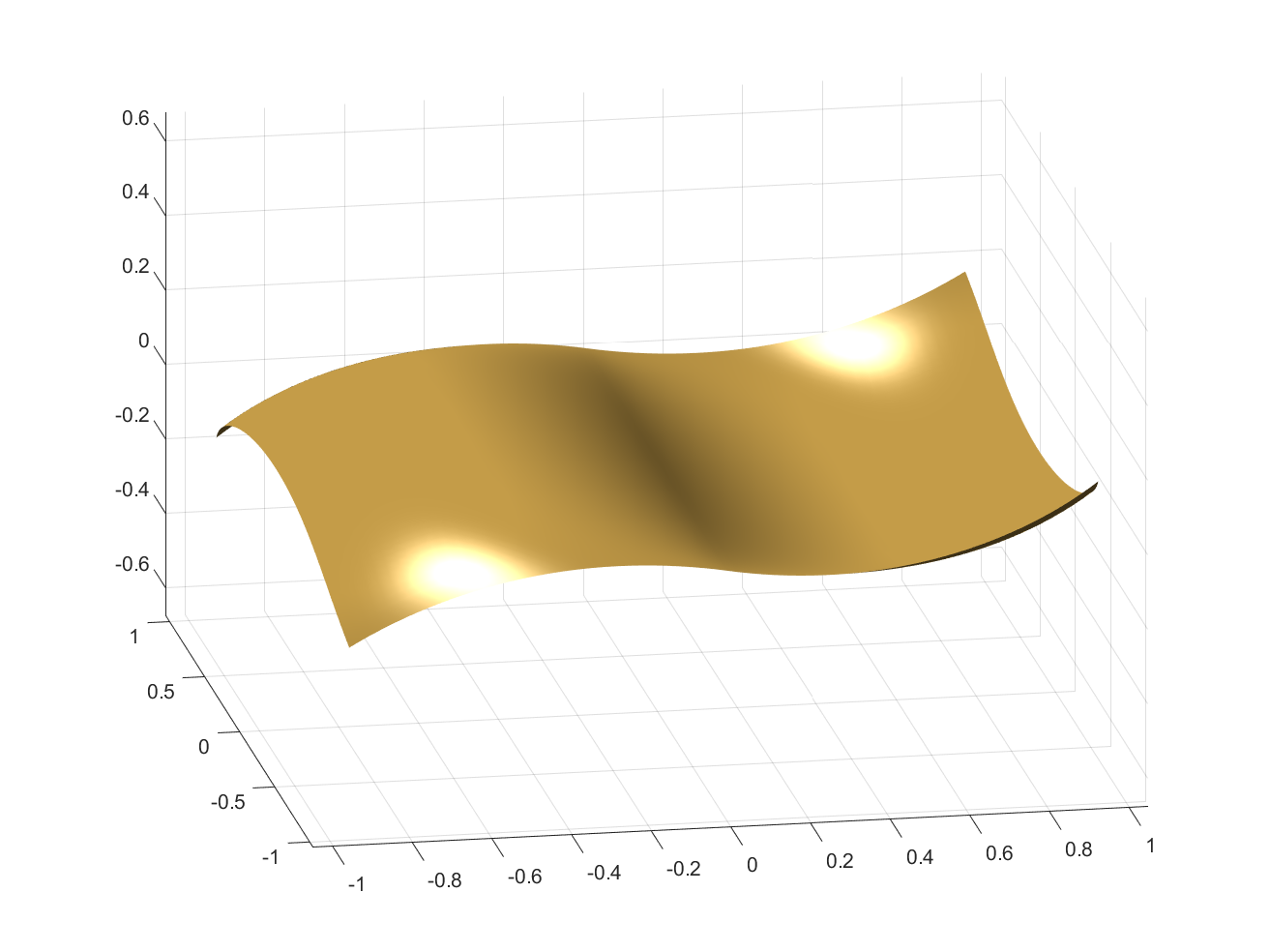}
\end{minipage}
\begin{minipage}{0.49\textwidth}
\centering
\includegraphics[width=.90\textwidth]{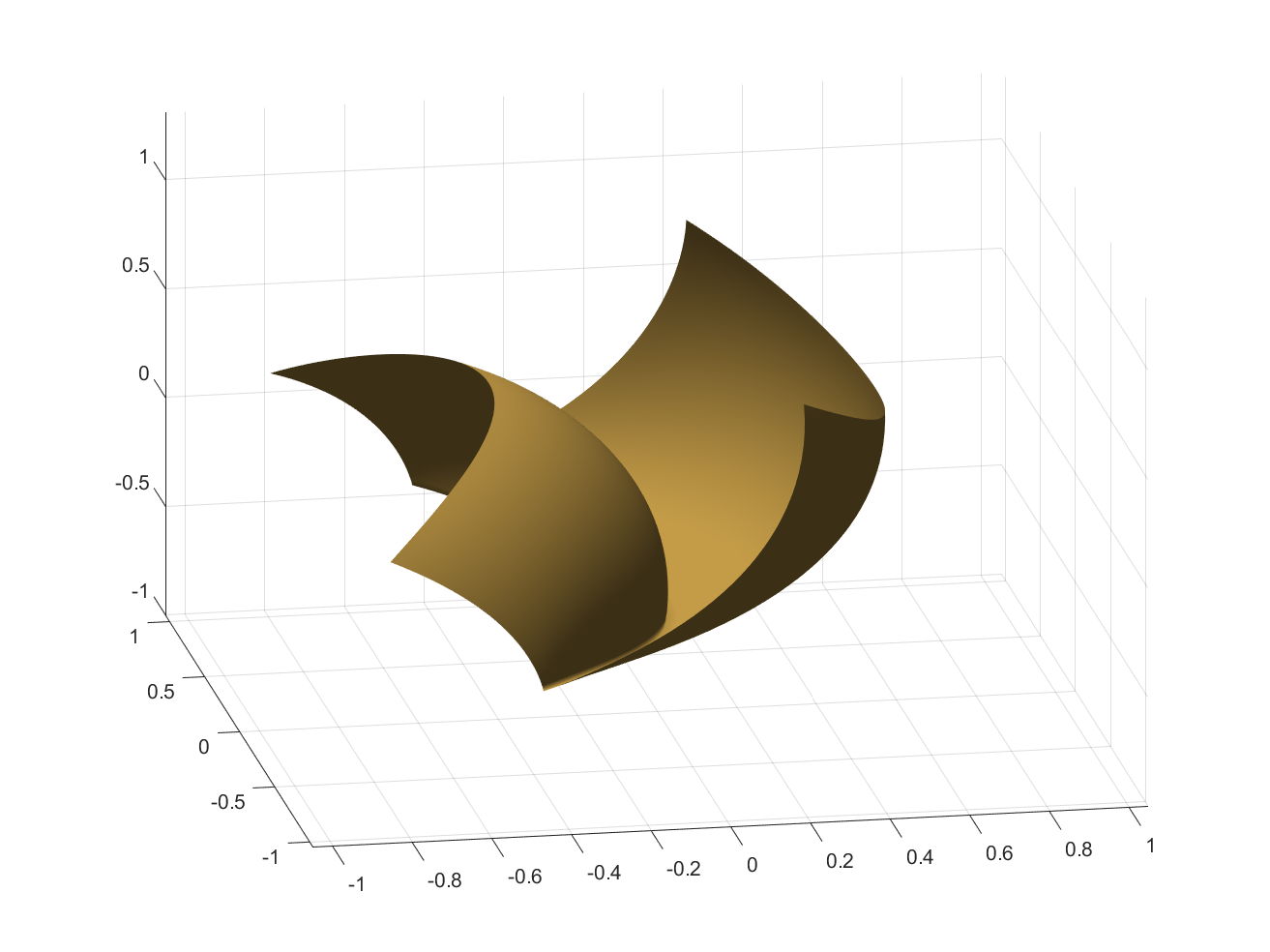}
\end{minipage}
\begin{minipage}{0.49\textwidth}
\centering
\includegraphics[width=.90\textwidth]{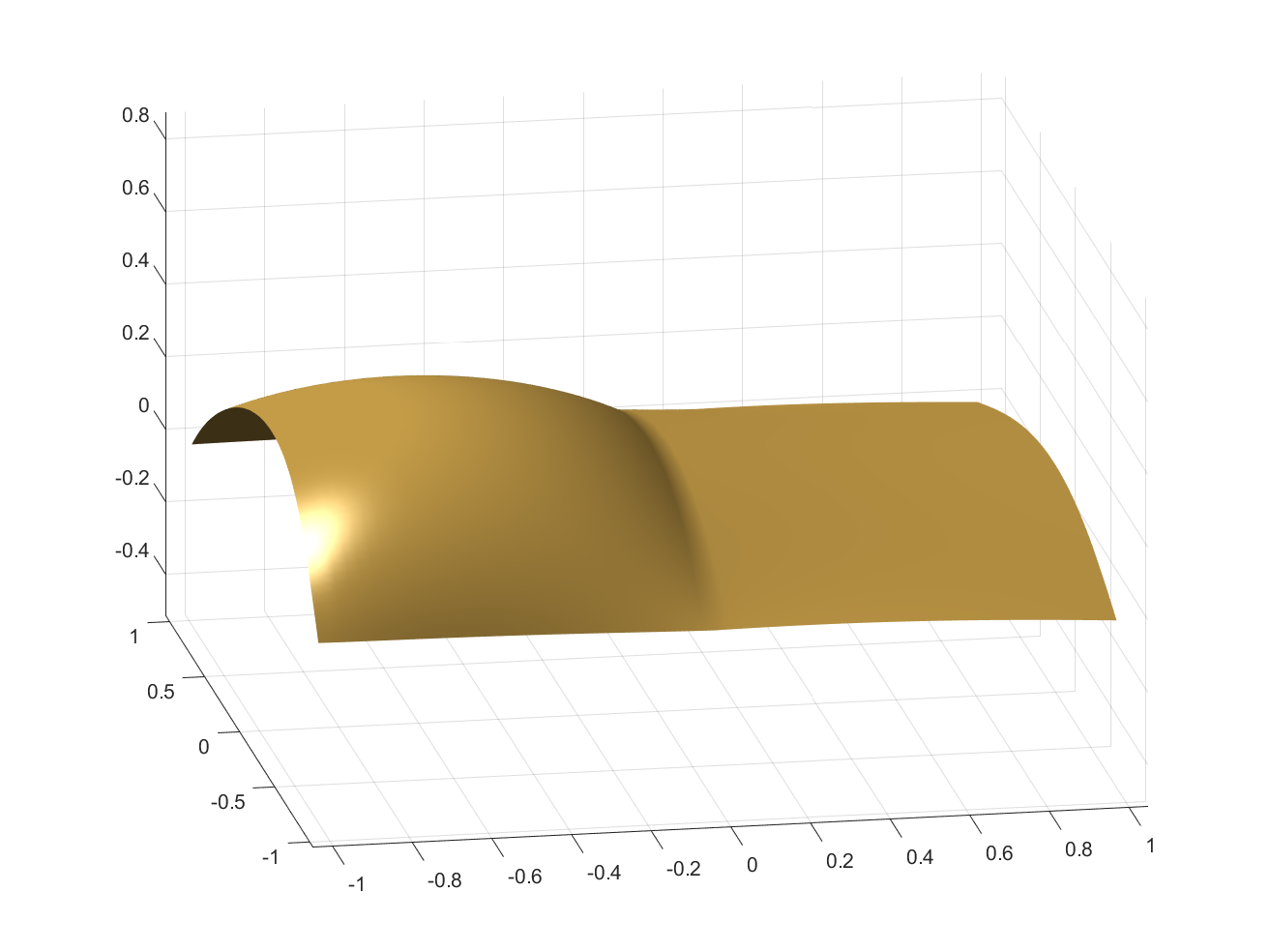}
\end{minipage}
\begin{minipage}{0.49\textwidth}
\centering
\includegraphics[width=.90\textwidth]{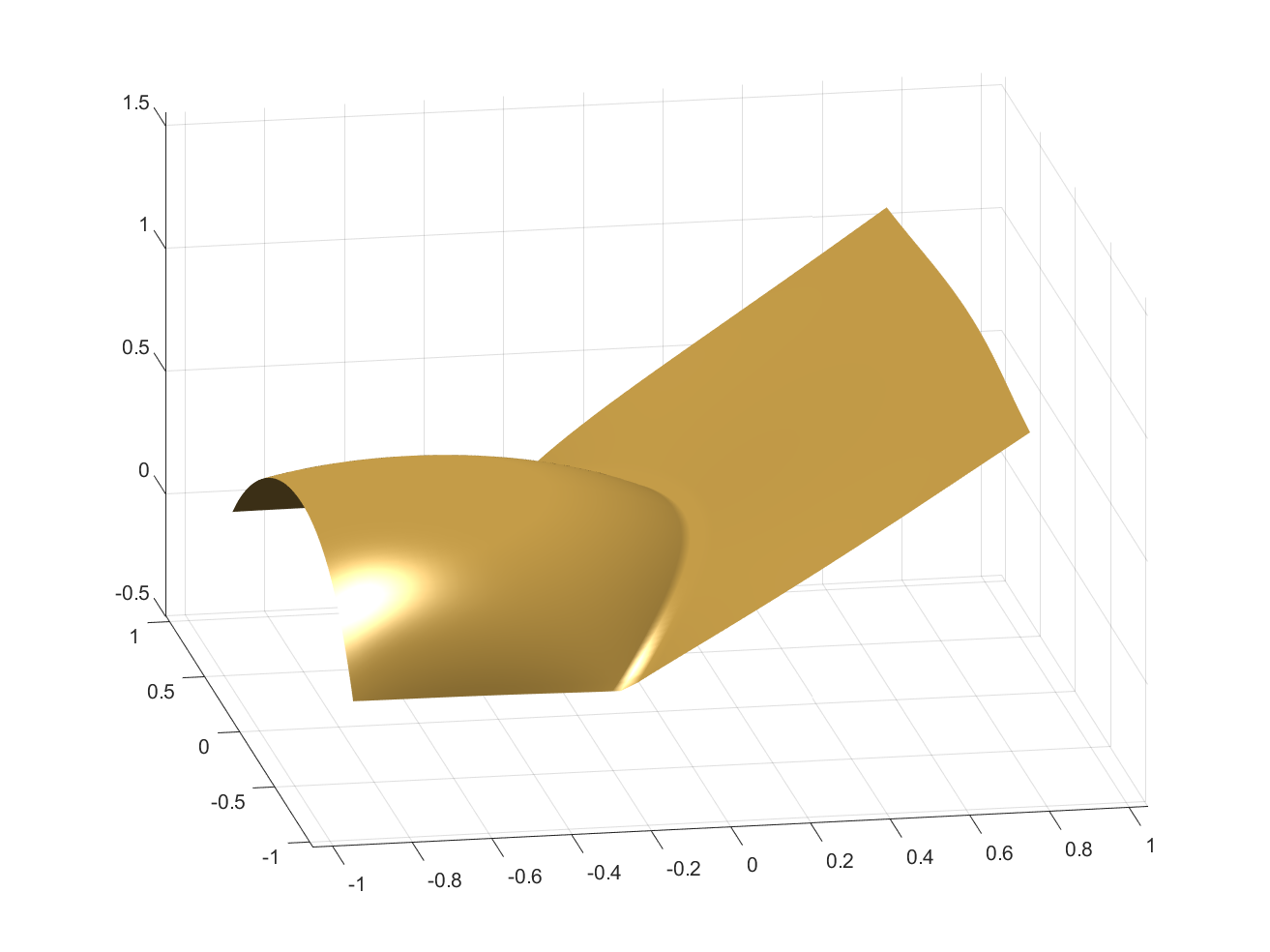}
\end{minipage}
\begin{minipage}{0.49\textwidth}
\centering
\includegraphics[width=.90\textwidth]{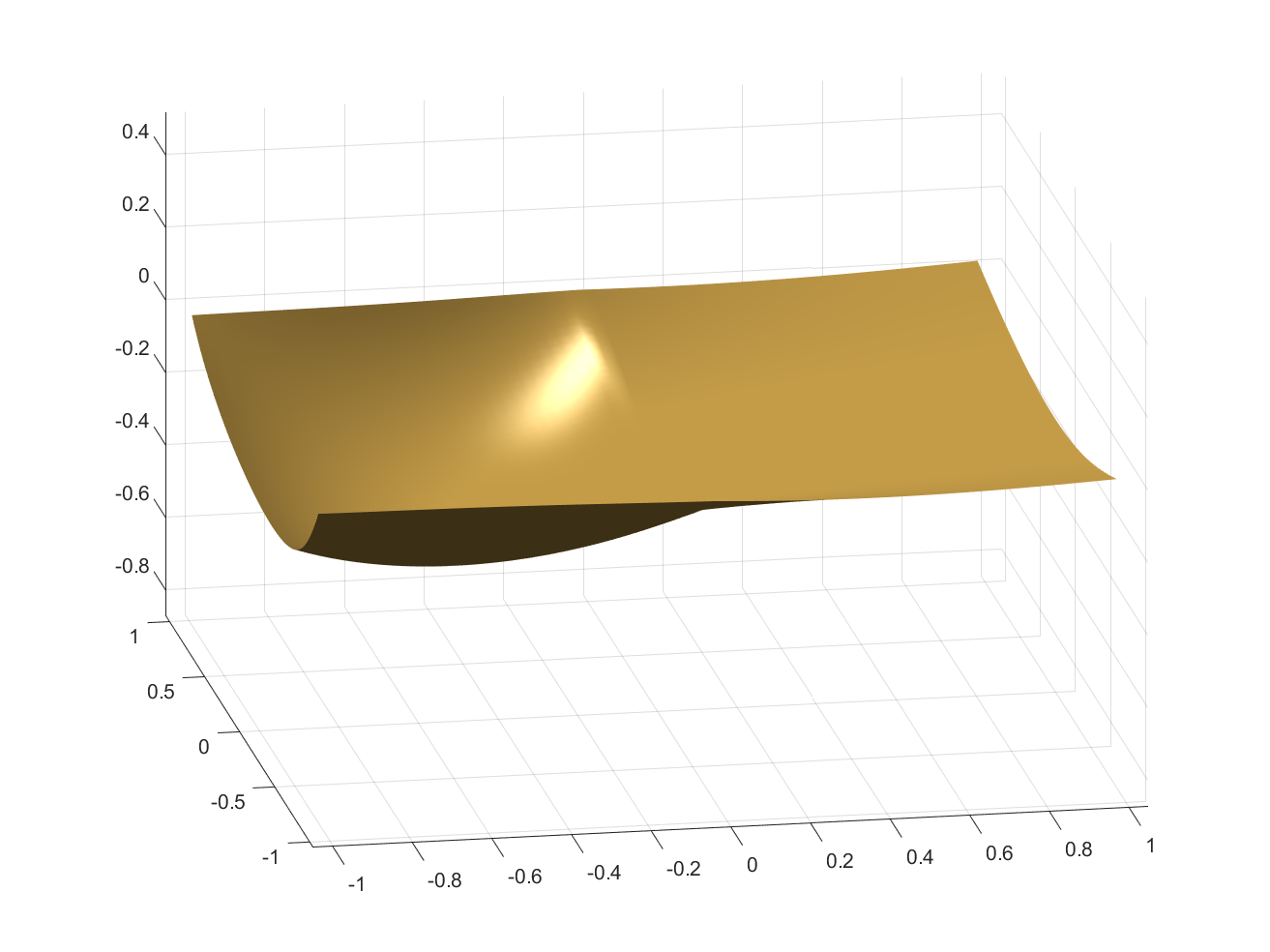}
\end{minipage}
\begin{minipage}{0.49\textwidth}
\centering
\includegraphics[width=.90\textwidth]{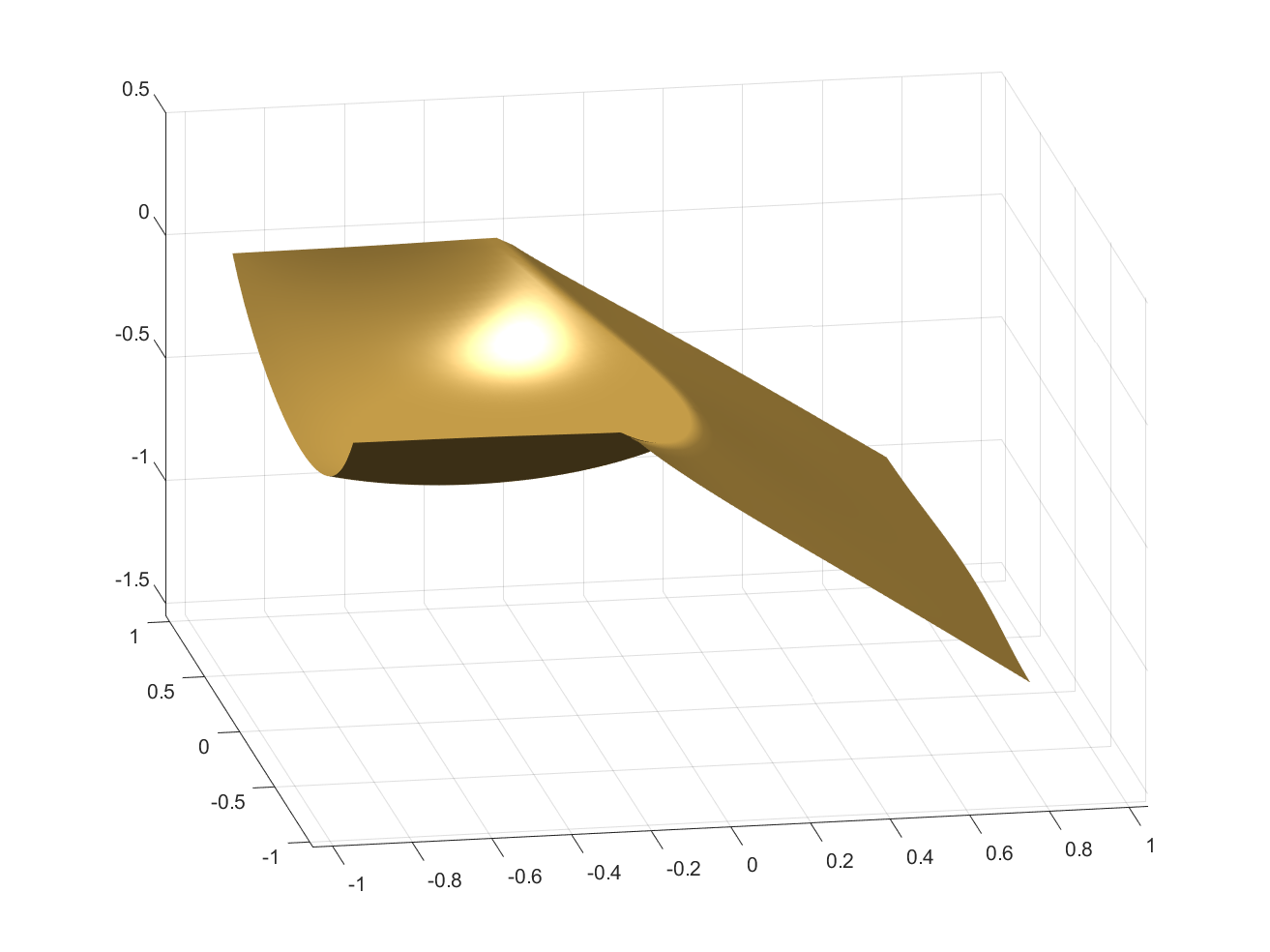}
\end{minipage}
\caption{Visualization of the numerical solutions for $\alpha_1=1,1,1,-1$, $\alpha_2=1,-1,0,0$ (from top to bottom) combined with a straight (left) and a curved (right) crease. In the last two rows simple support boundary conditions are imposed along
the upper and lower edge of the left subdomain. {\cc Drastic differences in the models with straight and curved crease geometries can be observed. For instance, the straight crease only leads to minor deflections of the right subdomain in the last two test, whereas the curved crease leads to significant deflections that appear in snapping mechanisms of Venus flytraps.}}
\label{fig:bilayer_folding}
\end{figure}

\newpage

\section*{Acknowledgements} 
{\cb The authors} acknowledge support by the DFG via the priority programme SPP 2256 \textit{Variational Methods for Predicting Complex Phenomena in Engineering Structures and Materials} ({\cc 441528968 and 441138507}). 
%%BA 2268/7-1

{\small\bibliography{Literatur}}

\Addresses

\end{document}

%% file: usepackages.tex
\usepackage[a4paper, left=3cm, right=3cm, top=3cm, marginparwidth=20mm, marginparsep=5mm]{geometry}
\usepackage[T1]{fontenc}
\usepackage[utf8]{inputenc} %Umlaute
\usepackage{float} %Image placement [H]
\usepackage[english]{babel}
\usepackage{amssymb} %math. Symbole
\usepackage{amsthm} %Beweise
\usepackage{amsmath}
\usepackage{mathrsfs}
\usepackage{amssymb} 
\usepackage{amsfonts}
\usepackage{doi}
\usepackage{mathrsfs}
\usepackage{hyperref}
\usepackage[export]{adjustbox}
\usepackage{authblk} % Address and email
\usepackage{listings}
\usepackage{accents}
\usepackage{subcaption}
\usepackage[numbered,framed]{matlab-prettifier}
\usepackage{wrapfig}
\usepackage{graphicx} %Bilder
\usepackage{paralist} %FÃ¼r Auflistung bearbeiten
\usepackage{mathtools} %rundere geschweifte Klammer
\usepackage[hang,flushmargin]{footmisc} %no indent in footnote
\usepackage{array}
\usepackage{pgfplots}
\pgfplotsset{compat=1.3}
\usepackage{tikz}
\usepackage{bm}
\usepackage{nccmath}
\usepackage{theoremref}
\usepackage{geometry}
\usepackage{setspace}
\usepackage{fancyhdr} %Kopf- und FuÃŸzeile
\usepackage[makeroom]{cancel} %durchstreichen
\usepackage{calc}
\usepackage{url}
\usepackage[small]{titlesec} %size headlines, tiny, small, medium, big
\usepackage{emptypage}	%clear empty page
\usepgfplotslibrary{fillbetween}
\usetikzlibrary{intersections}
\usepackage{caption}
\usepackage{import}

\usepackage{lmodern}
\usepackage[ttscale=.875]{libertine}

\numberwithin{equation}{section} %Number with sections

%% file: header.tex
\pgfdeclarelayer{bg}
\pgfsetlayers{bg,main}

\usepackage{xparse}

\NewDocumentCommand{\dcurl}{sO{}m}{%
  \IfBooleanTF{#1}
    {\dcurlext{#3}}
    {\dcurlx[#2]{#3}}%
}
\NewDocumentCommand{\dcurlext}{m}{%
  \sbox0{%
    \mathsurround=0pt % just for safety
    $\left\{\vphantom{#1}\right.\kern-\nulldelimiterspace$%
  }%
  \sbox2{\{}%
  \ifdim\ht0=\ht2
    \{\kern-.625\wd2 \{#1\}\kern-.625\wd2 \}%
  \else
    \left\{\kern-.7\wd0\left\{#1\right\}\kern-.7\wd0\right\}%
  \fi
}
\NewDocumentCommand{\dcurlx}{om}{%
  \sbox0{\mathsurround=0pt$#1\{$}%
  \sbox2{\{}%
  \ifdim\ht0=\ht2
    \{\kern-.625\wd2 \{#2\}\kern-.625\wd2 \}%
  \else
    \mathopen{#1\{\kern-.7\wd0 #1\{}
    #2
    \mathclose{#1\}\kern-.7\wd0 #1\}}
  \fi
}
\NewDocumentCommand{\dsquare}{sO{}m}{%
  \IfBooleanTF{#1}
    {\dsquareext{#3}}
    {\dsquarex[#2]{#3}}%
}
\NewDocumentCommand{\dsquareext}{m}{%
  \sbox0{%
    \mathsurround=0pt % just for safety
    $\left[\vphantom{#1}\right.\kern-\nulldelimiterspace$%
  }%
  \sbox2{[}%
  \ifdim\ht0=\ht2
    [\kern-.525\wd2 [#1]\kern-.525\wd2 ]%
  \else
    \left[\kern-.6\wd0\left[#1\right]\kern-.6\wd0\right]%
  \fi
}
\NewDocumentCommand{\dsquarex}{om}{%
  \sbox0{\mathsurround=0pt$#1[$}%
  \sbox2{[}%
  \ifdim\ht0=\ht2
    [\kern-.525\wd2 [#2]\kern-.525\wd2 ]%
  \else
    \mathopen{#1[\kern-.6\wd0 #1[}
    #2
    \mathclose{#1]\kern-.6\wd0 #1]}
  \fi
}

\newcommand{\dx}{\,\mathrm{d}x}

\newcommand{\T}{^\mathsf{T}} %$A^\top,A^\intercal,A^\mathsf{T}$

\newcommand{\If}{\mathcal{C}}
\newcommand{\Ifh}{\mathcal{C}_h}

\newcommand{\Th}{\mathscr{T}_h}

\newcommand{\id}{\mathrm{id}}
\newcommand{\Nh}{\mathscr{N}_h}

\newcommand{\uproman}[1]{\uppercase\expandafter{\romannumeral#1}}

\makeatletter
\newcommand{\customlabel}[2]{%
\protected@write \@auxout {}{\string \newlabel {#1}{{#2}{}}}}
\makeatother

\definecolor{highlight}{RGB}{255,70,0}

\newtheorem{thm}{Theorem}
\newtheorem{lem}{Lemma}
\newtheorem{prop}{Proposition}
\newtheorem{rmk}{Remark}
\newtheorem{alg}{Algorithm}
\newtheorem{defn}{Definition}
\newtheorem{cor}{Corollary}

\numberwithin{thm}{section}
\numberwithin{lem}{section}
\numberwithin{prop}{section}
\numberwithin{rmk}{section}
\numberwithin{alg}{section}
\numberwithin{defn}{section}
\numberwithin{figure}{section}
\numberwithin{cor}{section}

\lstdefinestyle{myCustomMatlabStyle}{
	tabsize=2,
	xleftmargin=5pt,
	frame=single,
	framexrightmargin=-5pt,
	language=Matlab
}
\newcommand{\eps}{\ensuremath{\varepsilon}}